\documentclass[a4paper,10pt]{amsart}

\usepackage{etex}

\usepackage[a4paper,margin=2.5cm]{geometry}

\usepackage{float}
\usepackage[all]{xy}

\usepackage{amsmath}
\usepackage{amssymb}
\usepackage{amsfonts}
\usepackage{amsthm}
\usepackage{enumitem}

\usepackage{mathtools}
    \mathtoolsset{centercolon}

\usepackage{microtype}

\usepackage{xcolor}

\usepackage{hyperref}

\usepackage{mathrsfs}

\usepackage{tikz}
    \usetikzlibrary{matrix,arrows,decorations.markings,backgrounds}

\def\DynkinNodeSize{3.5mm}
\def\DynkinArrowLength{2mm}

\tikzset{
  dnode/.style={
    circle,
    inner sep=0pt,
    minimum size=\DynkinNodeSize,
    fill=white,
    draw},
  middlearrow/.style={
    decoration={markings,
      mark=at position 0.6 with
      {\draw (0:0mm) -- +(+135:\DynkinArrowLength); \draw (0:0mm) -- +(-135:\DynkinArrowLength);},
    },
    postaction={decorate}
  },
  leftrightarrow/.style={
    decoration={markings,
      mark=at position 0.999 with
      {
      \draw (0:0mm) -- +(+135:\DynkinArrowLength); \draw (0:0mm) -- +(-135:\DynkinArrowLength);
      },
      mark=at position 0.001 with
      {
      \draw (0:0mm) -- +(+45:\DynkinArrowLength); \draw (0:0mm) -- +(-45:\DynkinArrowLength);
      },
    },
    postaction={decorate}
  },
  sedge/.style={
  },
  dedge/.style={
    middlearrow,
    double distance=0.5mm,
  },
  tedge/.style={
    middlearrow,
    double distance=1.0mm+\pgflinewidth,
    postaction={draw}, 
  },
  infedge/.style={
    leftrightarrow,
    double distance=0.5mm,
  },
}

\newcommand{\q}[1]{q_{#1}}
\newcommand{\nt}{\unlhd}
\newcommand{\inv}{\mathrm{inv}}
\newcommand{\sq}{\mathrm{sq}}
\newcommand{\sm}{\setminus}

\newcommand{\eps}{\varepsilon}

\let\ol\overline
\let\wt\widetilde
\let\wh\widehat

\newcommand{\onto}{\twoheadrightarrow}

\newcommand{\GL}[1]{\mathrm{GL}_{#1}(\RR)}
\newcommand{\SL}[1]{\mathrm{SL}_{#1}(\RR)}
\renewcommand{\O}[1]{\mathrm{O}(#1)}
\newcommand{\SO}[1]{\mathrm{SO}(#1)}
\newcommand{\Sp}[1]{\mathrm{Sp}_{#1}(\RR)}
\newcommand{\U}[1]{\mathrm{U}_{#1}(\CC)}
\newcommand{\SU}[1]{\mathrm{SU}_{#1}(\CC)}

\newcommand{\UH}{\mathrm{U}_1(\HH)}

\newcommand{\Pin}[1]{\mathrm{Pin}(#1)}
\newcommand{\Spin}[1]{\mathrm{Spin}(#1)}
\newcommand{\Cl}[1]{\mathrm{Cl}(#1)}
\newcommand{\Cle}[2]{\mathrm{Cl}^{#2}(#1)}
\newcommand{\Pu}{\mathrm{Pu}}

\newcommand{\I}{\mathfrak{I}}

\renewcommand{\AA}{\mathbb{A}}
\newcommand{\CC}{\mathbb{C}}
\newcommand{\HH}{\mathbb{H}}
\newcommand{\II}{\mathbb{I}}
\newcommand{\KK}{\mathbb{K}}
\newcommand{\NN}{\mathbb{N}}
\newcommand{\OO}{\mathbb{O}}
\newcommand{\PP}{\mathbb{P}}
\newcommand{\RR}{\mathbb{R}}

\newcommand{\ZZ}{\mathbb{Z}}

\newcommand{\AAA}{\mathcal{A}}
\newcommand{\EEE}{\mathcal{E}}
\newcommand{\GGG}{\mathcal{G}}
\newcommand{\HHH}{\mathcal{H}}
\newcommand{\PPP}{\mathcal{P}}
\newcommand{\KKK}{\mathcal{K}}
\newcommand{\LLL}{\mathcal{L}}

\newcommand{\wAAA}{\widetilde{\mathcal{A}}}

\newcommand{\Wext}{{W^{\mathrm{ext}}}}
\newcommand{\Wspin}{{W^{\mathrm{spin}}}}
\newcommand{\Dext}{{D^{\mathrm{ext}}}}
\newcommand{\Dspin}{{D^{\mathrm{spin}}}}

\DeclareMathOperator{\Fix}{Fix}

\DeclareMathOperator{\diag}{diag}
\DeclareMathOperator{\Aut}{Aut}
\DeclareMathOperator{\Out}{Out}
\DeclareMathOperator{\Inn}{Inn}
\DeclareMathOperator{\id}{id}
\DeclareMathOperator{\End}{End}
\DeclareMathOperator{\Sym}{Sym}

\DeclareMathOperator{\Rea}{Re}
\DeclareMathOperator{\Ima}{Im}

\newenvironment{pmat}{\left(\begin{smallmatrix}}{\end{smallmatrix}\right)}

\newcommand{\abs}[1]{\lvert#1\rvert}
\newcommand{\gen}[1]{\langle#1\rangle}
\newcommand{\cgen}[2]{\langle#1\,\vert\,#2\rangle}

\newcommand{\Defn}[1]{\textcolor{blue}{\textbf{\boldmath #1}}}

\newtheorem{mainthm}{Theorem}

\newtheorem{proposition}{Proposition}[section]
\newtheorem{theorem}[proposition]{Theorem}
\newtheorem{lemma}[proposition]{Lemma}
\newtheorem{corollary}[proposition]{Corollary}
\newtheorem{consequence}[proposition]{Consequence}

\newtheorem{observation}[proposition]{Observation}

\theoremstyle{definition}

\newtheorem{convention}[proposition]{Convention}
\newtheorem{notation}[proposition]{Notation}

\newtheorem{remark}[proposition]{Remark}

\newtheorem{strategy}[proposition]{Strategy}
\theoremstyle{definition}
\newtheorem{definition}[proposition]{Definition}
\newtheorem{defn}[proposition]{Definition}

\setenumerate[1]{label={\rm(\alph*)},ref={\theproposition~\textrm{(\alph*)}}}

\begin{document}

\title{
Spin covers of maximal compact subgroups of Kac--Moody groups and spin-extended Weyl groups}
\author{David Ghatei}
\author{Max Horn}
\author{Ralf K\"ohl}
\author{Sebastian Wei\ss}
\address{D.G.: University of Birmingham, School of Mathematics, Edgbaston, Birmingham, B15 2TT, United Kingdom}
\address{M.H., R.K., S.W.: JLU Giessen, Mathematisches Institut, Arndtstrasse 2, 35392 Giessen, Germany}
\email{max.horn@math.uni-giessen.de}
\email{ralf.koehl@math.uni-giessen.de }

\maketitle
\begin{abstract}
Let $G$ be a split real Kac--Moody group of arbitrary type and let $K$ be its maximal compact subgroup, i.e. the
subgroup of elements fixed by a Cartan--Chevalley involution of $G$.
We construct non-trivial spin covers of $K$, thus confirming a conjecture by Damour and Hillmann \cite{DamourHillmann}. For irreducible simply laced diagrams and for all spherical diagrams these spin covers are two-fold central extensions of $K$. For more complicated irreducible diagrams these spin covers are central extensions by a finite $2$-group of possibly larger cardinality.  
Our construction is amalgam-theoretic and makes use of the generalized spin representations of maximal compact subalgebras of split real Kac--Moody algebras studied in \cite{Hainke/Koehl/Levy}.
Our spin covers contain what we call spin-extended Weyl groups which admit a presentation by generators and relations obtained from the one for extended Weyl groups by relaxing the condition on the generators so that only their eighth powers are required to be trivial.
\end{abstract}

\section{Introduction} \label{sec:intro}

In \cite[Section~3.5]{DamourHillmann} it turned out that the existence of a spin-extended Weyl group $\Wspin(E_{10})$ would be very useful for the study of fermionic billards. Lacking a concrete mathematical model of that group $\Wspin(E_{10})$, Damour and Hillmann in their article instead use images of $\Wspin(E_{10})$ afforded by various generalized spin representations as described in \cite{DamourKleinschmidtNicolai}, \cite{deBuylHenneauxPaulot}, which can be realized as matrix groups.

In \cite[Section~3.5, footnote 18, p.\ 24]{DamourHillmann}, Damour and Hillmann conjecture that the spin-extended Weyl group $\Wspin(E_{10})$ can be constructed as a discrete subgroup of a double spin cover $\Spin{E_{10}}$ of the subgroup $K(E_{10})$ of elements fixed by the Cartan--Chevalley involution of the split real Kac--Moody group of type $E_{10}$. 
The purpose of this article is to confirm this conjecture, and to generalize it to arbitrary diagrams resp.
 arbitrary generalized Cartan matrices

In the simply laced case our result is as follows:

\begin{mainthm} \label{mainthm:sl-spincover}
Let $\Pi$ be an irreducible simply laced Dynkin diagram, i.e., a Dynkin diagram affording only single edges, let $I = \{ 1, \ldots, n \}$ be a set of labels of the vertices of $\Pi$, and let $K(\Pi)$ be the subgroup of elements fixed by the Cartan--Chevalley involution of the split real Kac--Moody group of type $\Pi$. For each $i \in I$ let $G_i \cong \Spin{2}$ and for each $i \neq j \in I$ let
\[G_{ij} \cong
\begin{cases}
  \Spin{3}, & \text{if $i$, $j$ form an edge of $\Pi$,} \\
  (\Spin{2} \times \Spin{2})/\gen{(-1,-1)}, &
      \text{if $i$, $j$ do not form an edge of $\Pi$.}
\end{cases}\]
Moreover, for $i < j \in I$, let $\phi_{ij}^i : G_i \to G_{ij}$ be the standard embedding as ``upper-left diagonal block'' and $\phi_{ij}^j : G_j \to G_{ij}$ be the standard embedding as ``lower-right diagonal block''.

Then up to isomorphism there exists a uniquely determined group, denoted $\Spin{\Pi}$, whose multiplication table extends the partial multiplication provided by $\left( \bigsqcup_{i < j \in I} G_{ij} \right) / \sim$ where $\sim$ is the equivalence relation determined by $\phi_{ij}^i(x) \sim \phi_{ik}^i(x)$ for all $i \neq j, k \in I$ and $x \in G_i$.

Furthermore, there exists a canonical two-to-one central extension $\Spin\Pi \to K(\Pi)$.  
\end{mainthm}

\noindent
The system $\{ G_i, G_{ij}, \phi_{ij}^i \}$ is called an \Defn{amalgam of groups}, the pair consisting of the group $\Spin\Pi$ and the set of canonical embeddings $\tau_{i} : G_{i} \hookrightarrow \Spin\Pi$, $\tau_{ij} : G_{ij} \hookrightarrow \Spin\Pi$ a \Defn{universal enveloping group}; the canonical embeddings are called \Defn{enveloping homomorphisms}. Formal definitions and background information concerning amalgams can be found in Section~\ref{sec:amalgams}. Since all $G_i\cong\Spin 2$ are isomorphic to one another, it in fact suffices to fix one group $U\cong \Spin 2$ instead with connecting homomorphisms $\phi_{ij}^i : U \to G_{ij}$.  

The formalization of the concept of standard embedding as ``upper-left/lower right diagonal block'' can be found in Section~\ref{sec:spin2amalgams}. Note that, since the $G_i$ are only given up to isomorphism, these standard embeddings are only well-defined up to automorphism of $G_i$, which leads to some ambiguity. Since by \cite{Hartnick/Koehl/Mars} the group $K(\Pi)$ (and therefore each of its central extensions by a finite group) is a topological group, one may assume the $\phi_{ij}^i$ to be continuous, thus restricting oneself to the ambiguity stemming from the two continuous automorphisms of $\Spin 2$, the identity and the inversion homomorphisms. This ambiguity is resolved in Section~\ref{sec:spin2amalgams}.

\smallskip
Theorem~\ref{mainthm:sl-spincover} provides us with the means of characterizing $\Wspin(\Pi)$.

\begin{mainthm} \label{mainthm:sl-weyl}
Let $\Pi$ be an irreducible simply laced Dynkin diagram, let $I = \{ 1, \ldots, n \}$ be a set of labels of the vertices of $\Pi$, and for each $i \in I$ let 
\begin{itemize}
\item $\tau_i : G_i \cong \Spin 2 \hookrightarrow \Spin\Pi$ be the canonical enveloping homomorphisms,
\item $x_i \in G_i$ elements of order eight whose polar coordinates involve
the angle $\frac{\pi}{4}$, and 
\item $r_i := \tau_i(x_i)$.
\end{itemize}
Then $\Wspin(\Pi) := \langle r_i \mid i \in I \rangle$ satisfies the defining relations
\begin{align}
 r_i^8&=1,  \tag{R1} \\
 r_i^{-1} r_j^2 r_i &= r_j^2 r_i^{2n(i,j)}
    \ \qquad\text{ for } i\neq j\in I,  \tag{R2} \\
 \underbrace{r_i r_j r_i \cdots}_{m_{ij} \text{ factors}} &=
 \underbrace{r_j r_i r_j \cdots}_{m_{ij} \text{ factors}}
    \qquad\text{ for } i\neq j\in I,    \tag{R3}
\end{align}
where \[
m_{ij} = \begin{cases}
    3, & \text{if $i$, $j$ form an edge,} \\
    2, & \text{if $i$, $j$ do not form an edge,}
\end{cases}
\quad \text{and} \quad
n(i,j) =
\begin{cases}
    1, & \text{if $i$, $j$ form an edge,} \\
    0, & \text{if $i$, $j$ do not form an edge.}
\end{cases}
\]
\end{mainthm}

To be a set of defining relations means that any product of the $r_i$ that in $\Wspin(\Pi)$ represents the identity can be written as a product of conjugates of ways of representing the identity via (R1), (R2), (R3). 

\medskip
Our results in fact can be extended to arbitrary diagrams as discussed in Sections~\ref{sec:adm-amalgams}, \ref{sec:tametypes}, and \ref{extended}.

\medskip
As a by-product of our proof of Theorem~\ref{mainthm:sl-spincover} we show in Section~\ref{sec13} that for non-spherical diagrams $\Pi$ the groups $\Spin\Pi$ and $K(\Pi)$ are never simple; instead they always admit a non-trivial compact Lie group as a quotient via the generalized spin representation described in \cite{Hainke/Koehl/Levy}. The generalized spin representations of $\Spin\Pi$ is continuous, so that the obtained normal subgroups are closed.
Similar non-simplicity phenomena as abstract groups have been observed in \cite{CapraceHume}. Furthermore, we observe that for arbitrary simply laced diagrams the image of $\Wspin$ under the generalized spin representation is finite, generalizing  \cite[Lemma~2, p.~49]{DamourHillmann}.

\medskip
Sections~\ref{sec:amalgams}, \ref{sec:cartan-dynkin}, \ref{sec:SO-O}, \ref{sec:spin-pin}, \ref{sec:7} are introductory in nature; we revise the notions of amalgams, Cartan matrices and Dynkin diagrams and fix our notation for orthogonal and spin groups. Sections~\ref{sec:so2amalgams} and \ref{sec:spin2amalgams} deal with the classification theory of amalgams and, as a blueprint for Theorem~\ref{mainthm:sl-spincover}, identify $\SO n$ and $\Spin n$ as universal enveloping groups of $\SO 2$-, resp.\ $\Spin 2$-amalgams of type $A_{n-1}$. In Section~\ref{sec:spin-cover-simply-laced} we prove Theorem~\ref{mainthm:sl-spincover}.

Sections~\ref{sec:g2}, \ref{sec:bc2}, \ref{sec:rank-2-residues} provide us with the necessary tools for generalizing our findings to arbitrary diagrams; they deal with equivariant coverings of the real projective plane by the split Cayley hexagon and the symplectic quadrangle and with coverings of the real projective plane and the symplectic quadrangle by trees. In Section~\ref{sec:adm-amalgams} we study $\SO 2$- and $\Spin 2$-amalgams for this larger class of diagrams. Section~\ref{sec:tametypes} deals with the general version of Theorem~\ref{mainthm:sl-spincover}. Section~\ref{sec14} deals with the proof of Theorem~\ref{mainthm:sl-weyl} and its generalization. In Section~\ref{sec13} we observe that our findings provide epimorphisms from $\Spin\Pi$ and $K(\Pi)$ onto non-trivial compact Lie groups.

\medskip
\textbf{Acknowledgements.} We thank Thibault Damour for pointing out his conjecture to us and Arjeh Cohen for very many valuable discussions concerning maximal compact subgroups of split real Lie groups of rank two. We also thank Guntram Hainke and Paul Levy for their comments and ideas concerning spin covers and Pierre-Emmanuel Caprace for several comments on a preliminary version of this article. The third named author moreover gratefully acknowledges the hospitality of the IHES in Bures-sur-Yvette and the Albert Einstein Institute in Golm. This research has been partially funded by the EPRSC grant
EP/H02283X.

\tableofcontents

\part{Basics}

\section{Conventions}

\begin{notation}
$\NN:=\{1,2,3,\ldots\}$ denotes the set of positive integers.
\end{notation}

\begin{notation}
Throughout this article we use the convention $ij:=\{i,j\}$ if the set $\{i,j\}$ is used as an index. For example, if $G_{ij}$ is a group, then $G_{ji}$ is the same group.
Note that this does not apply to superscripts, so $G^{ij}$ and $G^{ji}$ may differ.
\end{notation}

\begin{notation}
For any group $G$, consider the following maps:
\begin{alignat*}{2}
  \inv:&\ G\to G : x\mapsto x^{-1}, &&\quad\text{the inverse map}, \\
  \sq:&\ G\to G : x\mapsto x^2, &&\quad\text{the square map}.
\end{alignat*}
Both maps commute with any group homomorphism.
\end{notation}

\begin{notation}
For any group $G$, we denote by $Z(G)$ the \Defn{centre} of $G$.
\end{notation}

\section{Amalgams} \label{sec:amalgams}

In this section we recall the concept of amalgams. More details concerning this concept can, in various formulations, be found in \cite[Part~III.$\mathcal{C}$]{Bridson/Haefliger:1999}, \cite[Section~1.3]{Ivanov/Shpectorov:2002}, \cite[Section~1]{Gloeckner/Gramlich/Hartnick:2010}.

\begin{definition} Let $U$ be a group, and $I\neq\emptyset$ a set.
A \Defn{$U$-amalgam over $I$} is a set
\[\AAA=\{ G_{ij},\; \phi_{ij}^i \mid i\neq j\in I \}\]
such that $G_{ij}$ is a group and $\phi_{ij}^i:U\to G_{ij}$ is a monomorphism for all $i\neq j\in I$.
The maps $\phi^i_{ij}$ are called \Defn{connecting homomorphisms}. 
The amalgam is \Defn{continuous} if  $U$ and $G_{ij}$ are topological groups, and
$\phi_{ij}^i$ is continuous for all $i\neq j\in I$.
\end{definition}

\begin{definition} \label{def:ama-iso}
Let $\wt\AAA=\{ \wt{G}_{ij},\; \wt{\phi}_{ij}^i \mid i\neq j\in I \}$ and
$\AAA=\{ G_{ij},\; \phi_{ij}^i \mid i\neq j\in I \}$
be $U$-amalgams over $I$. An \Defn{epimorphism}, resp.\ \Defn{isomorphism} $\alpha:\wt\AAA\to\AAA$
of $U$-amalgams is a system
\[\alpha=\{ \pi,\; \alpha_{ij} \mid i\neq j \in I \}\]
consisting of a permutation $\pi\in\Sym(I)$ and group epimorphisms, resp.\ isomorphisms
\[
\alpha_{ij}:\wt{G}_{ij}\to G_{\pi(i)\pi(j)}
\]
such that for all $i \neq j \in I$
\[ \alpha_{ij}\circ \wt{\phi}_{ij}^i=\phi_{\pi(i)\pi(j)}^{\pi(i)}\ , \]
that is, the following diagram commutes:
\[
\xymatrix{
& \wt{G}_{ij} \ar[dd]^{\alpha_{ij}} \\
  U \ar[rd]_{\phi_{\pi(i)\pi(j)}^{\pi(i)}} \ar[ru]^{\wt{\phi}_{ij}^i}  \\
    & G_{\pi(i)\pi(j)}
}
\]

More generally, let $\wt\AAA=\{ \wt{G}_{ij},\; \wt{\phi}_{ij}^i \mid i\neq j\in I \}$ be a $U$-amalgam and let
$\AAA=\{ G_{ij},\; \phi_{ij}^i \mid i\neq j\in I \}$ be a $V$-amalgam.
An \Defn{epimorphism} $\alpha : \wt\AAA \to \AAA$ is a system
\[\alpha=\{ \pi,\;\rho^i,\;\alpha_{ij} \mid i\neq j \in I \}\]
consisting of a permutation $\pi\in\Sym(I)$, group epimorphisms $\rho^i:U\to V$ and group epimorphisms
\[
\alpha_{ij}:\wt{G}_{ij}\to G_{\pi(i)\pi(j)}
\]
such that for all $i \neq j \in I$
\[ \alpha_{ij}\circ \wt{\phi}_{ij}^i=\phi_{\pi(i)\pi(j)}^{\pi(i)} \circ \rho^{\pi(i)} \ , \]
that is, the following diagram commutes:
\[
\xymatrix{
  U \ar[d]_{\rho^{\pi(i)}} \ar[r]^{\wt{\phi}_{ij}^i} & \wt{G}_{ij} \ar[d]^{\alpha_{ij}}  \\
  V \ar[r]_{\phi_{\pi(i)\pi(j)}^{\pi(i)}} & G_{\pi(i)\pi(j)}
}
\]
\end{definition}

\begin{notation}
If (and only if) in the epimorphism $\alpha : \wt{\AAA} \to \AAA$ each $\rho^i : U \to V$ is an isomorphism, then one obtains an epimorphism $\alpha' : \wt{\AAA} \to \AAA'$ of $U$-amalgams by defining $\alpha' = \{ \pi, \alpha_{ij} \mid i \neq j \in I \}$ and $\AAA' = \{ G_{ij}, (\phi^i_{ij})' \}$ via \[(\phi^i_{ij})' : U \to G_{ij} : u \mapsto (\phi^{i}_{ij} \circ \rho^i)(u) ..\]
If this $\alpha'$ turns out to be an isomorphism of $U$-amalgams, by slight abuse of terminology we also call the epimorphism $\alpha$ an \Defn{isomorphism} of amalgams.
\end{notation}

\begin{remark} \label{altiso}
More generally, an amalgam can be defined as a collection of groups $G_i$ and a collection of groups $G_{ij}$ with connecting homomorphisms $\psi^i_{ij} : G_i \to G_{ij}$. Since in our situation for all $i$ there exist isomorphisms $\gamma_i : U \to G_i$, it suffices to consider the connecting homomorphisms $\phi^i_{ij} = \psi^i_{ij} \circ \gamma_i$.

In the more general setting, an isomorphism of amalgams consists of a permutation $\pi$ of the index set $I$ and isomorphisms $\alpha_i : G_i \to \ol{G}_{\pi(i)}$ and $\alpha_{ij} : G_{ij} \to \ol{G}_{\pi(i)\pi(j)}$ such that 
\[ \alpha_{ij} \circ \psi^i_{ij} = \ol{\psi}^{\pi(i)}_{\pi(i)\pi(j)} \circ \alpha_i\ .\]

A routine calculation shows that $U$-amalgams and isomorphisms of $U$-amalgams are special cases
of amalgams and isomorphisms of amalgams as found in the literature.
\end{remark}

\begin{definition}
Given a $U$-amalgam $\AAA=\{ G_{ij},\; \phi_{ij}^i \mid i\neq j\in I \}$, 
an \Defn{enveloping group} of $\mathit{\AAA}$ is a pair $(G,\tau)$ consisting of a group $G$ and a set
\[\tau = \{ \tau_{ij}\mid i\neq j\in I \}\]
of \Defn{enveloping homomorphisms} $\tau_{ij}:G_{ij}\to G$ such that
\begin{align*}
G=\gen{\tau_{ij}(G_{ij}) \mid i\neq j\in I}\ , &&\ \forall\ i\neq j\neq k\in I:\ \tau_{ij}\circ \phi_{ij}^j=\tau_{kj}\circ \phi_{kj}^j\ ,
\end{align*}
that is, for $i\neq j\neq k\in I$ the following diagram commutes:
\[
\xymatrix{
& G_{ij} \ar[dr]^{\tau_{ij}} \\
  U \ar[dr]_{\phi_{kj}^j} \ar[ur]^{\phi_{ij}^j}  && G \\
&   G_{kj} \ar[ur]_{\tau_{kj}} 
}
\]
We write $\tau : \AAA \to G$ and call $\tau$ an \Defn{enveloping morphism}.
An enveloping group $(G,\tau)$ and the corresponding enveloping morphism are \Defn{faithful} if $\tau_{ij}$ is a monomorphism for all $i\neq j\in I$.
\end{definition}

\begin{definition}
Given a $U$-amalgam $\AAA=\{ G_{ij},\; \phi_{ij}^i\}$, an enveloping group
$\big(G,\tau\big)$ is called a \Defn{universal enveloping group} if, 
given an enveloping group $(H,\tau')$ of $\AAA$, there is a unique
epimorphism $\pi:G\to H$ such that for all $i \neq j \in I$ one has
$\pi\circ \tau_{ij}=\tau'_{ij}$. We write $\tau : \AAA \to G$ and call $\tau$ a \Defn{universal enveloping morphism}. By universality, two universal enveloping groups $(G_1,\tau_1)$ and $(G_2,\tau_2)$ of a $U$-amalgam $\AAA$ are (uniquely) isomorphic. 

The \Defn{canonical universal enveloping group} of the $U$-amalgam $\AAA$
is the pair $\big(G(\AAA),\wh\tau\big)$, where
$G(\AAA)$ is the group given by the presentation
\[G(\AAA):=\left\langle \bigcup_{i\neq j\in I} G_{ij} \mid \text{all relations in } G_{ij}, \text{ and }
\forall\ i\neq j\neq k\in I, \forall x\in U: 
\phi_{ij}^j(x)=\phi_{kj}^j(x) \right\rangle\]
and where $\wh\tau=\{\wh\tau_{ij}\mid i\neq j\in I\}$ with the canonical homomorphism $\wh\tau_{ij}:G_{ij}\to G(\AAA)$ for all $i\neq j\in I$. The canonical universal enveloping group of a $U$-amalgam is a universal enveloping group (cf.\ \cite[Lemma~1.3.2]{Ivanov/Shpectorov:2002}). 
\end{definition}

\begin{lemma}
Let $U$ and $V$ be groups and $I$ an index set. Suppose
\begin{itemize}[itemsep=1mm plus 0.5mm minus 0.5mm]
\item $\wt\AAA=\{ \wt{G}_{ij},\; \wt{\phi}_{ij}^i \mid i\neq j\in I \}$ is a $U$-amalgam over $I$,
\item $\AAA=\{ G_{ij},\; \phi_{ij}^i \mid i\neq j\in I \}$ is a $V$-amalgam over $I$,
\item $\alpha=\{ \pi,\;\rho^i,\;\alpha_{ij} \mid i\neq j\in I \}$ is an amalgam epimorphism $\wt\AAA\to\AAA$,
\item $(G,\tau)$ with $\tau=\{ \tau_{ij} \mid i\neq j\in I \})$ is an enveloping group of $\AAA$.
\end{itemize}
Then the following hold:
\begin{enumerate}
\item\label{lem:ama-envelope-lift}
There is a unique enveloping group $(G, \wt \tau)$, $\wt \tau = \{ \wt\tau_{ij} \mid i\neq j\in I \}$, of $\wt\AAA$ such that the following
diagram commutes for all $i\neq j\in I$:
\[
\xymatrix{
  \wt{G}_{ij} \ar@{-->}[rrd]^{ \wt\tau_{ij}} \ar[d]_{\alpha_{ij}} \\
  G_{ij} \ar[rr]^{\tau_{ij}} && G
}
\]
\item \label{lem:ama-epi-induces-envelope-epi}
Suppose $(\wt{G}, \wt \tau)$, $\wt \tau = \{ \wt\tau_{ij} \mid i\neq j\in I \}$, is a universal enveloping group of $\wt\AAA$.
Then there is a unique epimorphism $\wh\alpha:\wt{G}\to G$
such that the following diagram commutes for all $i\neq j\in I$:
\[
\xymatrix{
  \wt{G}_{ij} \ar[rr]^{\wt\tau_{ij}} \ar[d]_{\alpha_{ij}} &&
  \wt{G} \ar@{-->}[d]^{\wh\alpha} \\
  G_{ij} \ar[rr]^{\tau_{ij}} && G
}
\]
\item If $\alpha$ is an isomorphism and $(G,\tau)$ is a universal enveloping group,
then also $\wh\alpha$ is an isomorphism.
\end{enumerate}
\end{lemma}

\begin{proof}
\begin{enumerate}
\item
Let $i\neq j\in I$.
Since $\alpha_{ij}$ is an epimorphism, we must have $\wt\tau_{ij}:=\tau_{ij}\circ \alpha_{ij}$
for the diagrams to commute; the claimed uniqueness follows.
The fact that $\alpha_{ij}$ is an epimorphism also implies
\[ \wt\tau_{ij}(\wt{G}_{ij})=(\tau_{ij}\circ \alpha_{ij})(\wt{G}_{ij}) = \tau_{ij}(G_{ij})\ ,
   \qquad\text{ and so }\qquad
   G=\gen{\tau_{ij}(\wt{G}_{ij})}=\gen{\wt\tau_{ij}(\wt{G}_{ij})}\ ..\]
Moreover, for $i\neq j\neq k\in I$ we find
\begin{align*}
  \wt\tau_{ij}\circ \wt{\phi}_{ij}^j
= \tau_{ij}\circ \alpha_{ij} \circ \wt{\phi}_{ij}^j
&= \tau_{ij}\circ \phi_{\pi(i)\pi(j)}^{\pi(j)} \circ \rho^{\pi(j)} \\
&= \tau_{kj}\circ \phi_{\pi(k)\pi(j)}^{\pi(j)} \circ \rho^{\pi(j)}
= \tau_{kj}\circ \alpha_{kj} \circ \wt{\phi}_{kj}^j
= \wt\tau_{kj}\circ \wt{\phi}_{kj}^j\ .
\end{align*}
Hence $(G,\{ \wt\tau_{ij} \})$ is indeed an enveloping group of $\wt\AAA$.

\item
On the one hand, by (a) the lower left triangle in the following diagram commutes:
\[
\xymatrix{
  \wt{G}_{ij} \ar[rr]^{\tau_{ij}} \ar@{-->}[rrd]^{\ol\tau_{ij}\circ\alpha_{ij}} \ar[d]_{\alpha_{ij}} &&
  G \ar@{-->}[d]^{\exists!\wh\alpha} \\
  G_{ij} \ar[rr]^{\ol\tau_{ij}} && \ol{G}
}
\]
On the other hand, by the definition of universal enveloping group
there is a unique epimorphism $\wh\alpha$ making the upper right triangle
commute. The claim follows.
\item follows from (b) by exchanging the roles of $G_{ij}$, $G$ and $\widetilde G_{ij}$, $\widetilde G$. 
\qedhere
\end{enumerate}
\end{proof}

\begin{notation} \label{nota:amalgamcomm}
We denote the situation in Lemma~\ref{lem:ama-envelope-lift} by the commutative diagram
\[
\xymatrix{
  \wt{\mathcal{A}} \ar[rrd]^{ \wt\tau} \ar[d]_{\alpha} \\
  \mathcal{A} \ar[rr]^{\tau} && G
}
\]
and the situation in Lemma~\ref{lem:ama-epi-induces-envelope-epi} by the commutative diagram
\[
\xymatrix{
  \wt{\mathcal{A}} \ar[rr]^{\wt\tau} \ar[d]_{\alpha} &&
  \wt{G} \ar[d]^{\wh\alpha} \\
  \mathcal{A} \ar[rr]^{\tau} && G
}
\]
\end{notation}

The following proposition will be crucial throughout this article. The typical situation in our applications will be $U=\SO{2}$, $\wt{U}=\Spin{2}$, $\wt{V}=\{\pm1\}$.

\begin{proposition} \label{prop:env-grp-central-cover}
Let $U$, $\wt{U}$ and $\wt{V}\leq\wt{U}$ be groups and $I$ an index set. Suppose
\begin{itemize}[itemsep=1mm plus 0.5mm minus 0.5mm]
\item $\wt\AAA=\{ \wt{G}_{ij},\; \wt{\phi}_{ij}^i \mid i\neq j\in I \}$ is a $\wt{U}$-amalgam over $I$ such that $\wt{G}_{ij}=\gen{\wt\phi_{ij}^i(\wt{U}), \wt\phi_{ij}^j(\wt{U})}$,
\item $\AAA=\{ G_{ij},\; \phi_{ij}^i \mid i\neq j\in I \}$ is a $U$-amalgam over $I$,
\item $\alpha=\{ \pi,\;\rho^i,\;\alpha_{ij} \mid i\neq j\in I \}$ is an amalgam epimorphism $\wt\AAA\to\AAA$,
\item $(\wt{G},\wt\tau)$ with $\wt\tau=\{ \wt\tau_{ij} \mid i\neq j\in I \})$ is a universal enveloping group of $\wt\AAA$,
\item $(G,\tau)$ with $\tau=\{ \tau_{ij} \mid i\neq j\in I \})$ is a universal enveloping group of $\AAA$,
\item $\wh\alpha:\wt{G}\to G$ is the epimorphism induced by $\alpha$
via the commutative diagrams ($i \neq j \in I$)
\[
\xymatrix{
  \wt{G}_{ij} \ar[rr]^{\wt\tau_{ij}} \ar[d]_{\alpha_{ij}} &&
  \wt{G} \ar@{-->}[d]^{\wh\alpha} \\
  G_{ij} \ar[rr]^{\tau_{ij}} && G
}
\]
as in Lemma~\ref{lem:ama-epi-induces-envelope-epi}.
\end{itemize}
For $i\neq j\in I$ define $Z_{ij}^i := \wt\phi_{ij}^i(\wt{V})$ and
$Z_{ij}^{\widetilde\phi} := \gen{ Z_{ij}^i, Z_{ij}^j}$, as well as
$A_{ij}:=\ker(\alpha_{ij})$.
Then if \[A_{ij} \leq  Z_{ij}^{\widetilde\phi} \leq Z(\wt{G}_{ij}),\]
it follows that $\wt{G}$ is a central extension of $G$ by
$N:= \gen{ \wt\tau_{ij}(A_{ij}) \mid i\neq j\in I}$.

In this situation the epimorphism $\alpha:\wt\AAA\to\AAA$ is called an \Defn{$|N|$-fold central extension of amalgams}.
\end{proposition}

\begin{proof}
We proceed by proving the following two assertions:
\begin{enumerate}
\item $N \leq Z(\wt{G})$,
\item $\wt{G}/N \cong G$.
\end{enumerate} 
Consider the following commutative diagram:
\[\xymatrix{
\wt{V} \ar[ddrr] \ar[rr]^{\wt{\phi}^i_{ij|\wt{V}}} &&
Z_{ij}^{\widetilde\phi} \ar[dr] &&
1 \ar[d] &&
1 \ar@{-->}[d] \\
&&& Z(\wt{G}_{ij}) \ar[dr] &
    A_{ij} \ar[l] \ar[d] \ar[rr]^{\wt{\tau}_{ij|A_{ij}}} \ar[ull] &&
    N \ar[d] \\
&&\wt{U} \ar[d]^{\rho^i} \ar[rr]^{\wt{\phi}^i_{ij}} &&
  \wt{G}_{ij} \ar[d]^{\alpha_{ij}} \ar[rr]^{\wt{\tau}_{ij}} &&
  \wt{G} \ar[d]^{\widehat \alpha}\\
&& U \ar[rr]^{{\phi}^i_{ij}} &&
   G_{ij} \ar[rr]^{{\tau}_{ij}} \ar[d] &&
   G \ar@{-->}[d] \\ 
&&&& 1 && 1 
}
\]
For $i\neq j\in I$ set
\[
Z_i := \wt\tau_{ij}(Z_{ij}^i)
  = (\wt\tau_{ij} \circ \wt\phi_{ij}^i)(\wt{V})
\qquad\text{ and }\qquad
\wt{G}_i := (\wt\tau_{ij} \circ \wt\phi_{ij}^i)(\wt{U})  \ .
\]
The hypothesis implies
\[
\wt\tau_{ij}(\wt{G}_{ij})=\gen{\wt{G}_i, \wt{G}_j}
\qquad\text{ and }\qquad
\wt{G}
= \gen{ \wt\tau_{ij}(\wt{G}_{ij}) \mid i\neq j\in I}
= \gen{\wt{G}_i \mid i\in I}\ .
\]
Moreover,
\[
Z_i = \wt\tau_{ij}(Z_{ij}^i)
  \leq \wt\tau_{ij}(Z(\wt{G}_{ij}))
  \leq Z(\wt\tau_{ij}(\wt{G}_{ij}))
  = Z(\gen{\wt{G}_i, \wt{G}_j})\ ,
\]
whence $Z_i$ centralizes $\wt{G}_j$ for all $i,j\in I$. Since $\wt{G}$
is generated by the $\wt{G}_i$, one has
\[
\gen{Z_i\mid i\in I}
\leq Z(\gen{\wt{G}_i\mid i\in I})
= Z(\wt{G})\ .
\]
Therefore
\[
N = \gen{ \wt\tau_{ij}(A_{ij}) \mid i\neq j\in I}
\leq \gen{ \wt\tau_{ij}(Z_{ij}^i) \mid i\neq j\in I}
 = \gen{ Z_i \mid i\in I}
\leq Z(\wt{G})\ ,
\]
i.e., (a) holds.

\medskip

Commutativity of the diagram implies $N\leq \ker(\wh\alpha)$ and so the homomorphism theorem yields an epimorphism $\wt{G}/N\to G : gN\mapsto\wh\alpha(g)$.
In order to show that this epimorphism actually is an isomorphism we construct an inverse
map by exploiting that $G$ and $\wt{G}$
are universal enveloping groups of $\AAA$, resp.\ $\wt\AAA$.
Indeed, for $g\in G_{ij}$, let $\wt{g}\in\alpha^{-1}_{ij}(g)$. Then
\[
  \wt\tau_{ij}(\alpha^{-1}_{ij}(g))
= \wt\tau_{ij}(\wt{g} A_{ij})
\leq \wt\tau_{ij}(\wt{g}) N
= \wt\tau_{ij}(\alpha^{-1}_{ij}(g)) N
\in \wt{G}/N\ .
\]
Thus we obtain a well-defined homomorphism
\[
\wh\tau_{ij} : G_{ij} \to \wt{G}/N : g \mapsto \wt\tau_{ij}(\alpha^{-1}_{ij}(g)) N\ .
\]
Then $(\wt{G}/N, \{\wh\tau_{ij}\})$ is an enveloping group
for $\AAA$. In particular for $u\in U$ and $i\neq j\neq k\in I$ we have
\begin{align*}
\left(\wh\tau_{ij}\circ \phi_{ij}^j\right)(u)
&= \wt\tau_{ij}\left(\alpha^{-1}_{ij}\left(\phi_{ij}^j(u)\right)\right)
= \wt\tau_{ij}\left(\wt\phi_{ij}^j\left(\left(\rho^j\right)^{-1}(u)\right)\right) \\
&= \wt\tau_{kj}\left(\wt\phi_{kj}^j\left(\left(\rho^j\right)^{-1}(u)\right)\right)
= \wt\tau_{kj}\left(\alpha^{-1}_{kj}\left(\phi_{kj}^j(u)\right)\right)
= \left(\wh\tau_{kj}\circ \phi_{kj}^j\right)(u)\ .
\end{align*}
Since $(G,\{\tau_{ij}\})$ is a universal enveloping group of $\AAA$,
there exists a unique epimorphism $\beta:G\to\wt{G}/N$ such that
for $i\neq j\in I$ we have
\[ \beta\circ\tau_{ij} = \wh\tau_{ij}\ . \]
By the definition of $\wh\alpha$ and $\wh\tau_{ij}$ we find
\[ \wh\alpha\circ\wh\tau_{ij} = \tau_{ij}\ . \]
Therefore
\[ (\beta\circ\wh\alpha)\circ\wh\tau_{ij} = \wh\tau_{ij}
\qquad\text{ and }\qquad
(\wh\alpha\circ\beta)\circ\tau_{ij} = \tau_{ij}\ . \]
But $(G,\tau)$ and $(\wt{G},\wt\tau)$ are universal enveloping groups;
their uniqueness property implies that
$\beta\circ\wh\alpha=\id_{\wt{G}/N}$ and $\wh\alpha\circ\beta=\id_{G}$
and hence as claimed $\wt{G}/N\cong G$. We have shown assertion (b).
\end{proof}

\section{Cartan matrices and Dynkin diagrams} \label{sec:cartan-dynkin}

In this section we recall the concepts of Cartan matrices and Dynkin diagrams. For a thorough introduction see \cite[Chapter~4]{kac1994infinite}, \cite[Section~7.1]{Remy:2002}.

\begin{definition}
Let $I$ be a non-empty set.
A \Defn{generalized Cartan matrix} over $I$ is a matrix $A=(a(i,j))_{i,j\in I}$
such that for all $i\neq j\in I$,
\begin{enumerate}
\item $a(i,i)=2$,
\item $a(i,j)$ is a non-positive integer,
\item if $a(i,j)=0$ then $a(j,i)=0$.
\end{enumerate}
$A$ is of \Defn{two-spherical type} if $a(i,j)a(j,i)\in\{0,1,2,3\}$
for all $i\neq j\in I$.
\end{definition}

\begin{definition}
A \Defn{Dynkin diagram} (or short: \Defn{diagram}) is
a graph $\Pi$ with vertex set $V(\Pi)$ and edge set $E(\Pi)\subseteq
\binom{V(\Pi)}{2}$ such that each edge has an \Defn{edge valency} of $1$, $2$, $3$ or
$\infty$ and, in addition, edges with valency $2$ or $3$ are directed.
If $\{v,w\}\in E(\Pi)$ is directed from $v$ to $w$, we write $v\to w$.
%
Let $E_0(\Pi):=\binom{V(\Pi)}{2}\sm E(\Pi)$, and let
$E_1(\Pi)$, $E_2(\Pi)$,  $E_3(\Pi)$, resp. $E_\infty(\Pi)$ be the subsets of
$E(\Pi)$ of edges of valency $1$, $2$, $3$, resp.\ $\infty$.
The elements of $E_1(\Pi)$, $E_2(\Pi)$,  $E_3(\Pi)$ are called \Defn{edges of type} $A_2$, $\mathrm{C}_2$
resp.\ $\mathrm{G}_2$.
%
The diagram $\Pi$ is \Defn{irreducible} if it is connected as a graph, it is \Defn{simply laced} if all edges have valency $1$, it is \Defn{doubly laced} if all edges have valency $1$ or $2$, and it is \Defn{two-spherical} if no edge has valency $\infty$.
%
If $V(\Pi)$ is finite, then a \Defn{labelling} of $\Pi$ is a bijection $\sigma:I \to V(\Pi)$, where $I:=\{ 1,\ldots, \abs{V(\Pi)}\}$.
\end{definition}

Throughout this text, we assume all diagrams to have finite vertex set.

\begin{remark} \label{longandshort}
Let $I$ be a non-empty set and $A=(a(i,j))_{i,j\in I}$ a two-spherical
generalized Cartan matrix. Then this induces a two-spherical Dynkin diagram
$\Pi(A)$ with vertex set $V:=I$ 
as follows: For $i\neq j\in I$, there is an edge between $i$
and $j$ if and only if $a(i,j)\neq 0$. The valency of the edge then is
$q_{ij}:=a(i,j)a(j,i)\in\{1,2,3\}$. If $q_{ij}>1$, then the edge is
directed $i\leftarrow j$ if and only if $a(i,j) = -q < -1 = a(j,i)$.

%
%
%
%
%
%
%

Conversely, given a two-spherical Dynkin diagram $\Pi$ with vertex set $V$,
we obtain a two-spherical generalized Cartan matrix
$A(\Pi):=(a(i,j))_{i,j\in I}$ over $I:=V$ by setting for $i\neq j\in I$:
\[
a(i,i):=2,\qquad
a(i,j):=
\begin{cases}
 0, & \text{if }\{i,j\}\notin E(\Pi), \\
-2, & \text{if }\{i,j\}\in E_2(\Pi)
    \text{ and }i\leftarrow j, \\
-3, & \text{if }\{i,j\}\in E_3(\Pi)
    \text{ and }i\leftarrow j, \\
-1, & \text{otherwise}.\\
\end{cases}
\]

These two operations are inverse to each other, i.e.,
$\Pi(A(\Pi))=\Pi$ and $A(\Pi(A))=A$.
\end{remark}

\begin{notation}  \label{augmented}
If the generalized Cartan matrix $A$ is not of two-spherical type, it is nevertheless possible to associate a Dynkin diagram $\Pi(A)$ to it by labelling the edge between $i$ and $j$ with $\infty$ whenever $a(i,j)a(j,i) \geq 4$. In this case it is, of course, not possible to reconstruct the values of $a(i,j)$ and $a(j,i)$ from the diagram $\Pi$.

Therefore, by convention, in this article for each edge between $i$ and $j$ with label $\infty$ we consider the values of $a(i,j)$ and $a(j,i)$ as part of the \Defn{augmented} Dynkin diagram: write $-a(i,j)$ between the vertex $i$ and the $\infty$ label and $-a(j,i)$ between the vertex $j$ and the $\infty$ label.
In addition, an edge with $\infty$ label such that $a(i,j)$ and $a(j,i)$ have different parity gets directed $i \leftarrow j$, if $a(i,j)$ is even, and $i \to j$, if $a(i,j)$ is odd. See Figure~\ref{fig:annotated-dynkin} for an example.
\end{notation}

\begin{figure}[h]
\centering
$A=\begin{pmatrix}
 2 & -2 &  0 &  0 \\
-2 &  2 & -1 &  0 \\
 0 & -4 &  2 & -1 \\
 0 &  0 & -1 &  2 \\
\end{pmatrix}\ \leadsto\ $
\begin{tikzpicture}[baseline]
    \node[dnode,label=below:1,label={above right:2}] (1) at (0,0) {};
    \node[dnode,label=below:2,label={above left:2},label={above right:1}] (2) at (2,0) {};
    \node[dnode,label=below:3,label={above left:4}] (3) at (4,0) {};
    \node[dnode,label=below:4] (4) at (6,0) {};

    \path (1) edge[sedge] node[above=1.5mm] {$\infty$} (2)
          (2) edge[sedge,middlearrow] node[above=1.5mm] {$\infty$} (3)
          (3) edge[sedge] (4)
          ;
\end{tikzpicture}

\caption{An augmented Dynkin diagram.}
\label{fig:annotated-dynkin}
\end{figure}

\section{The groups ${\SO{n}}$ and $\O{n}$} \label{sec:SO-O}

In this section we fix notation concerning the compact real orthogonal groups.

\begin{definition} \label{qn}
Given a quadratic space $(\KK,V,q)$ with $\dim_\KK V<\infty$, we set
\begin{align*}
\O{q}:=\{ a\in \mathrm{GL}(V) \mid \forall\ v\in V:\ q(av)=q(v)\}\ , && \SO{q}:=\O{q}\cap \mathrm{SL}(V)\ .
\end{align*}
Given $n\in \NN$, let $\q{n}:\RR^n\to \RR : x\mapsto \sum_{i=1}^n x_i^2$ be the standard quadratic form on $\RR^n$, and
\begin{align*}
 \O{n}&:=\{ a\in \GL{n}\mid aa^t=E_n\}\cong \O{\q{n}} = \O{-\q{n}}\ , \\
\SO{n}&:=\O{n}\cap \SL{n} \cong \SO{\q{n}}=\SO{-\q{n}}\nt \O{\q{n}} = \O{-\q{n}}\ .
\end{align*}
Since an element of $\O{n}$ has determinant $1$ or $-1$, we have $[\O{n}:\SO{n}]=2$.
\end{definition}

\begin{notation} \label{nota:EI-VI-qI}
Let $n\in \NN$ and let $\EEE=(e_1,\ldots,e_n)$ be the standard basis of $\RR^n$. Given a subset $I\subseteq \{1,\ldots,n\}$, we set
\begin{align*}
\EEE_I:=\{ e_i\mid i\in I\}\ , &&
V_I:=\gen{\EEE_I}_\RR\leq \RR^n\ , &&
q_I:={\q{n}}_{|V_I}:V_I\to \RR\ .
\end{align*}
There are canonical isomorphisms
\[M_\EEE:\End(\RR^n)\to \mathrm{M}_n(\RR) : a\mapsto M_\EEE(a)
  \quad\text{and}\quad
M_{\EEE_I}:\End(V_I)\to M_{|I|}(\RR): a\mapsto M_{\EEE_I}(a)\]
that map an endomorphism into its transformation matrix with respect to the standard basis $\EEE$, resp.\ the basis $\EEE_I$.
Moreover, there is a canonical embedding \[\eps_I:\O{q_I}\to \O{\q{n}},\]
inducing a canonical embedding
\[M_\EEE\circ \eps_I\circ M_{\EEE_I}^{-1}:\O{|I|}\to \O{n}\ ,\]
which, by slight abuse of notation, we also denote by $\eps_I$.
We will furthermore use the same symbol for the (co)restriction of $\eps_I$ to $\SO{\cdot}$. The most important application of this map in this article is for $|I|=2$ with $I = \{ i, j \}$ providing the map \[\eps_{ij} : \SO{2} \to \SO{n}.\] 
\end{notation}

\section{The groups ${\Spin{n}}$ and ${\Pin{n}}$} \label{sec:spin-pin}

In this section we recall the compact real spin and pin groups. For a thorough treatment we refer to \cite{Lawson/Michelsohn:1989}, \cite{Gallier}, \cite{Meinrenken:2013}.

\begin{definition} \label{def:defcliffordconj}
Let $(\RR,V,q)$ be a quadratic space and let $T(V) = \bigoplus_{n \geq 0} V^{\otimes n}$ be the tensor algebra of $V$. The identity $V^{\otimes 0} = \RR$ provides a ring monomorphism $\RR \to T(V)$, the identity $V^{\otimes 1}=V$ a vector space monomorphism $V \to T(V)$ that allow one to identify $\RR$, $V$ with their respective images in $T(V)$. For
\[\I(q):=\langle v\otimes v-q(v) \mid v\in V \rangle\] define the \Defn{Clifford algebra of ${q}$} as \[\Cl{q}:=T(V)/\I(q).\]
Moreover, let
\[\Cl{q}^*:=\{ x\in \Cl{q} \mid \exists\ y\in \Cl{q}:\ xy=1\}\ .\]
The \Defn{transposition map} is the involution 
\[ \tau:\Cl{q}\to \Cl{q} \quad\quad \text{ induced by } \quad\quad  v_1\cdots v_k\mapsto v_k\cdots v_1, \quad v_i \in V, \]
cf.\ \cite[Section~2.2.6]{Meinrenken:2013}, \cite[Proposition~1.1]{Gallier}.
The \Defn{parity automorphism} is the map  
\[\Pi:\Cl{q}\to \Cl{q} \quad\quad \text{ given by } \quad\quad v_1\cdots v_k \mapsto (-1)^k \cdot v_1\cdots v_k, \quad v_i \in V, \]
cf.\ \cite[Section~2.2.2, Section~3.1.1]{Meinrenken:2013}, \cite[Proposition~1.2]{Gallier}.
We set 
\[
\Cle{q}{0}:=\{ x\in \Cl{q} \mid \Pi(x)=x \}
\quad\text{ and }\quad
\Cle{q}{1}:=\{ x\in \Cl{q} \mid \Pi(x)=-x\},
\]
which yields a $\ZZ_2$-grading of $\Cl{q}$, i.e.,
\[ \Cl{q}=\Cle{q}{0}\oplus \Cle{q}{1}
\quad\text{ and }\quad
\Cle{q}{i}\Cle{q}{j}\subseteq \Cle{q}{i+j}
\quad\text{ for } i,j\in \ZZ_2.
\]

Furthermore, following \cite[Section~3.1]{Gallier}, we define the \Defn{Clifford conjugation}
\[ \sigma:\Cl{q}\to\Cl{q} : x\mapsto \ol{x}:=\tau\Pi(x)=\Pi\tau(x), \]
and the \Defn{spinor norm}
\[ N:\Cl{q}\to \Cl{q} : x\mapsto x\ol{x}. \]
\end{definition}

\begin{notation}
In the following, $(\RR,V,q)$ is an anisotropic quadratic space such that $\dim_\RR V<\infty$.
\end{notation}

\begin{definition} Given $x\in \Cl{q}^*$, the map
\[ \rho_x:\Cl{q}\to\Cl{q} : y\mapsto \Pi(x)yx^{-1} \]
is the \Defn{twisted conjugation with respect to ${x}$}.
Using the canonical identification of $V$ with its image in $\Cl{q}$,
we define
\[ \Gamma(q):=\{ x\in \Cl{q}^*\mid \forall\ v\in V:\ \rho_x(v)\in V\} \]
to be the \Defn{Clifford group with respect to ${q}$}, cf.\ \cite[Section~3.1.1]{Meinrenken:2013}, \cite[Definition~1.4]{Gallier}. We obtain a representation
\[ \rho: \Gamma(q)\to \mathrm{GL}(V) : x\mapsto \rho_x, \]
which is the \Defn{twisted adjoint representation}.
\end{definition}

\begin{definition}
Given $n\in \NN$ and $V = \RR^n$, we set
\[ \Cl{n}:=\Cl{-\q{n}}
\quad\text{ and }\quad
\Gamma(n):=\Gamma(-\q{n}).
\]
Recall that $\q{n}$ is defined to be the standard quadratic form on $\RR^n$, cf.\ Definition~\ref{qn}.
\end{definition}

Note that the literature one can also find the opposite sign convention.

\begin{remark}\label{rem:cl3=quaternions}
\begin{enumerate}
\item Let $n \in \NN$ and let $e_1,\ldots,e_n$ be the standard basis of $\RR^n$. Then the following hold in $\Cl{n}$ for $1 \leq i \neq j \leq n$:
\begin{align*}
e_i^2 &= -1, \\
e_ie_j &= -e_je_i, \\
(e_ie_j)^2 &= -1.
\end{align*}
The first identity is immediate from the definition. The second identity follows from polarization, as in the tensor algebra $T(\RR^n)$ one has
\begin{align*}
\I(\q{n}) &\ni (e_i+e_j) \otimes (e_i+e_j) - q(e_i+e_j) \\ &= e_i \otimes e_i + e_i \otimes e_j + e_j \otimes e_i + e_j \otimes e_j - q(e_i) - q(e_j) - 2b(e_i,e_j) \\ &= e_i \otimes e_j + e_j \otimes e_i,
\end{align*}
where $b(\cdot,\cdot)$ denotes the bilinear form associated to $\q{n}$. The third identity is immediate from the first two.
\item One has
$\Cle{3}{0}\cong \HH$, where $\HH$ denotes the quaternions. Indeed, given a basis $e_1$, $e_2$, $e_3$ of $\RR^3$, a basis of $\Cle{3}{0}$, considered as an $\RR$-vector space, is given by $1$, $e_1e_2$, $e_2e_3$, $e_3e_1$. By (a) the latter three basis elements square to $-1$ and anticommute with one another. Note, furthermore, that under this isomorphism the Clifford conjugation is transformed into the standard involution of the quaternions and, consequently, the the spinor norm into the norm of the quaternions.  
\end{enumerate}
\end{remark}

\begin{lemma}\label{1}
The map $N:\Cl{q}\to \Cl{q}$ induces a homomorphism
\[N:\Gamma(q)\to \RR^*\] such that
\[\forall\ x\in \Gamma(q):\qquad N\big(\Pi(x)\big)=N(x)\ .\]
\end{lemma}

\begin{proof}
Cf. \cite[Proposition~1.9]{Gallier}.
\end{proof}

\begin{definition}\label{defspinpin}
The group
\[\Pin{q}:=\{ x\in \Gamma(q) \mid N(x)=1\}\leq \Gamma(q)\]
is the \Defn{pin group with respect to ${q}$}, and
\[\Spin{q}:=\Pin{q}\cap \Cle{q}{0}\leq \Pin{q}\]
is the \Defn{spin group with respect to ${q}$}. By Lemma~\ref{1} and the $\ZZ_2$-grading of $\Cl{q}$, the sets $\Pin{q}$ and $\Spin{q}$ are indeed subgroups of $\Gamma(q)$.
Given $n\in \NN$, define \[\Pin{n}:=\Pin{-q_n} \quad\quad \text{ and } \quad\quad \Spin{n}:=\Spin{-q_n}.\]
\end{definition}

\begin{theorem} \label{rho}
The following hold: 
\begin{enumerate}
\item One has $[\Pin{q}:\Spin{q}]=2$ and $\Spin{q}=\rho^{-1}\big(\SO{q}\big)$.
\item The twisted adjoint representation $\rho:\Gamma(q)\to \mathrm{GL}(V)$ induces an epimorphism $\rho : \Pin{q} \to \O{q}$. In particular, given $n\in \NN$, we obtain epimorphisms
\begin{align*}
\rho_n:=M_\EEE\circ \rho:\Pin{n}\to \O{n}\ , && \rho_n:=M_\EEE\circ \rho:\Spin{n}\to \SO{n}
\end{align*}
with $\ker(\rho_n)=\{\pm1\}$ in both cases.
\item The group $\Spin{q}$ is a double cover of the group $\SO{q}$.
\end{enumerate}
\end{theorem}

\begin{proof}
See \cite[Theorem~1.11]{Gallier}.
\end{proof}

\begin{remark}\label{rem:3}\ 
\begin{enumerate}
\item By slight abuse of notation, suppressing the choice of basis, we will also sometimes denote the map $\rho_n$ by $\rho$.
\item Let $H_1\leq \Spin{n}$ and $H_2\leq \Pin{n}$ be such that $-1\in H_1$ and $-1\in H_2$, respectively, and let $\tilde{H}_i:=\rho_n(H_i)$. Then we have
$H_i=\rho_n^{-1}(\tilde{H}_i)$.
We will explicitly determine these groups for some canonical subgroups of $\SO{n}$ and $\O{n}$.
\end{enumerate}
\end{remark}

\begin{lemma}\label{2}
Let $n\in \NN$ and $I\subseteq \{1,\ldots,n\}$. Then there is a canonical embedding
$\tilde\eps_I:\Pin{-q_I}\to \Pin{-q_n}$
satisfying \[\forall x \in \Cl{-q_I} : \qquad {\rho_{\tilde\eps_I(x)}}_{|V_I} = \eps_I \circ \rho_x\]
such that the following diagram commutes:
\[\xymatrix{
 \Pin{-q_I} \ar[d]^\rho\ar[rr]^{\tilde\eps_I} & & \Pin{-q_n} \ar[d]^{\rho} \\
\O{-q_I} \ar[rr]^{\eps_I} & & \O{-q_n}
}
\]
\end{lemma}

In analogy to Notation~\ref{nota:EI-VI-qI}
we will use the same symbol for the (co)restriction of $\tilde\eps_I$ to $\Spin{\cdot}$. The most important application of this map in this article is for $|I|=2$ with $I = \{ i, j \}$ providing the map \[\tilde\eps_{ij} : \Spin{2} \to \Spin{n}.\] 

\begin{proof}
Let $x \in \Gamma(-q_I)$. By definition, \[\forall v \in V_I : \qquad \rho_{\tilde\eps_I(x)}(v) \in V_I \subseteq \RR^n.\]
Since $e_ie_j = -e_je_i$ for all $i \neq j \in I$ by Remark~\ref{rem:cl3=quaternions}(a), for each $\RR$-basis vector $y = e_{j_1}\cdots e_{j_k}$ of $\tilde\eps_I(\Cl{-q_I})$ and all $i \in \{ 1, \ldots, n \} \backslash I$ one has \[\Pi(y)e_i = e_iy.\] Hence \[\Pi(\tilde\eps_I(x))e_i=e_i \tilde\eps_I(x)\] and, thus, for all $i \in \{ 1, \ldots, n \} \backslash I$ \[\rho_{\tilde\eps_I(x)}(e_i) = \Pi(\tilde\eps_I(x))e_i\tilde\eps_I(x)^{-1} = e_i \in \RR^n.\]

As $\tilde\eps_I(\Cl{-q_I})$ is generated as an $\RR$-algebra by the set $\{ e_i \mid i \in I \}$, we in particular have \[\rho \circ \tilde\eps_I = \eps_I \circ \rho.\]
Therefore $\eps_I(x) \in \Gamma(-q_n)$. Finally, \[N(\tilde\eps_I(x)) = \tilde\eps_I(N(x)) = \tilde\eps_I(1)=1,\] whence \[\tilde\eps_I(x) \in \Pin{-q_n}. \qedhere\]
\end{proof}

\begin{remark}
Since $\tilde\eps_I(\Spin{-q_I}) = \langle e_ie_j \mid i \neq j \in I \rangle \subseteq \Cle{-q_n}{0}$, one has \[\tilde\eps_I(\Spin{-q_I}) \subseteq \Pin{-q_n} \cap \Cle{-q_n}{0} = \Spin{-q_n}.\]
\end{remark}

\begin{consequence}\label{7}
Let $n\in \NN$ and $I\subseteq \{1,\ldots,n\}$. Then \[\rho_n^{-1}\big(\eps_I(\O{|I|})\big)=\tilde\eps_I\big(\Pin{-q_I}\big) \qquad \text{and} \qquad \rho_n^{-1}\big(\eps_I(\SO{|I|})\big)=\tilde\eps_I\big(\Spin{-q_I}\big).\]
\end{consequence}

\begin{proof}
By Lemma~\ref{2}, one has
$\rho_n\tilde\eps_I\big(\Pin{-q_I}\big)=\eps_I\big(\O{|I|}\big)$ and $\rho_n\tilde\eps_I\big(\Spin{-q_I}\big)=\eps_I\big(\SO{|I|}\big)$,
thus the assertion results from Remark~\ref{rem:3}(b).
\end{proof}

\begin{remark}\label{6}
Let $n \in \NN$, let $I \subseteq \{ 1, \ldots, n \}$ and let $m:=\abs{I}$. Then there exists an isomorphism $i : \Pin{m} \to \Pin{-q_I}$ such that the following diagram commutes:
\[\xymatrix{
\Pin{m}\ar[d]_{\rho_m} \ar[rr]^{i} && \Pin{-q_I} \ar[rr]^{\tilde\eps_I} \ar[d]^\rho \ar[lld]^{M_{\eps_I} \circ \rho} && \Pin{-q_n} \ar[rr]^{\id} \ar[d]_\rho \ar[drr]_{\rho_n} && \Pin{n} \ar[d]^{\rho_n} \\
\O{m} \ar[rr]_{{M_{\eps_I}}^{-1}} && \O{-q_I} \ar[rr]_{\eps_I} && \O{-q_n} \ar[rr]_{M_\eps} && \O{n} 
}
\]
As in \ref{nota:EI-VI-qI} we slightly abuse notation and also write  $\tilde\eps_I$ for the map $\id \circ \tilde\eps_I \circ i : \Pin{m} \to \Pin{n}$ and $\eps_I$ for the map $M_\eps \circ \eps_I \circ {M_{\eps_I}}^{-1} : \O{m} \to \O{n}$. Consequently, we obtain the following commutative diagram:
\[\xymatrix{
 \Pin{m} \ar[d]^{\rho_m}\ar[rr]^{\tilde\eps_I} & & \Pin{n} \ar[d]^{\rho_n} \\
\O{m} \ar[rr]^{\eps_I} & & \O{n}
}
\]
\end{remark}

\begin{remark}\label{coordinatesforspin}
According to \cite[Corollary~1.12]{Gallier}, the group $\Pin{n}$ is generated by the set \[\{ v \in \RR^n \mid N(v) = 1 \}\] and each element of the group $\Spin{n}$ can be written as a product of an even number of elements from this set. 
That is, each element $g \in \Spin{2}$ is of the form \[g = \prod^{2k}_{i=1} (a_ie_1 + b_ie_2) = \prod^k_{i=1} \big((a_{2i-1}a_{2i} + b_{2i-1}b_{2i}) + (a_{2i-1}b_{2i} - a_{2i}b_{2i-1})e_1e_2\big) =: a + be_1e_2.\]
The requirement $a_ie_1 + b_ie_2 \in \{ v \in \RR^n \mid N(v) = 1 \}$ is equivalent to \[a_i^2+b_i^2 = (a_ie_1+b_ie_2)(-a_ie_1-b_ie_2)=(a_ie_1+b_ie_2)\overline{(a_ie_1+b_ie_2)}=N(a_ie_1+b_ie_2)= 1.\] Morever, \[1=N(g)=N(a+be_1e_2)=(a+be_1e_2)\overline{(a+be_1e_2)}=(a+be_1e_2)(a+be_2e_1)=a^2+b^2.\]  
Certainly, $\Spin{2}$ contains all elements of the form $a+be_1e_2$ with $a^2+b^2=1$, i.e., one obtains \[\Spin{2} = \{ \cos(\alpha) + \sin(\alpha) e_1e_2 \mid \alpha \in \RR \}.\]
One has \[(\cos(\alpha) + \sin(\alpha) e_1e_2)^{-1} = \cos(\alpha)-\sin(\alpha)e_1e_2= \cos(-\alpha) + \sin(-\alpha) e_1e_2,\] i.e., the map \[\RR \to \Spin{2} : \alpha \mapsto \cos(\alpha) + \sin(\alpha) e_1e_2\] is a group homomorphism from the real numbers onto the circle group.
The twisted adjoint representation $\rho_2$ maps the element $\cos(\alpha)+\sin(\alpha)e_1e_2 \in \Spin{2}$ to the transformation
\begin{eqnarray*}
x_1e_1+x_2e_2 & \mapsto & (\cos(\alpha)+\sin(\alpha)e_1e_2)(x_1e_1+x_2e_2)(\cos(\alpha)-\sin(\alpha)e_1e_2) \\
& = & x_1\left(\cos(\alpha)^2-\sin(\alpha)^2\right)e_1 - 2x_2\cos(\alpha)\sin(\alpha)e_1 \\ & & + 2x_1\cos(\alpha)\sin(\alpha)e_2 + x_2\left(\cos(\alpha)^2-\sin(\alpha)^2\right)e_2 \\
& = & (x_1\cos(2\alpha)-x_2\sin(2\alpha))e_1 + (x_1\sin(2\alpha)+x_2\cos(2\alpha))e_2, 
\end{eqnarray*}
i.e., the rotation of the euclidean plane $\RR^2$ by the angle $2\alpha$. In other words, $\rho_2$ is the double cover of the circle group by itself, cf.\ Theorem~\ref{rho}(b).

Similarly, each element $g \in \Spin{3}$ is of the form
\[g = \prod^{2k}_{i=1} (a_ie_1 + b_ie_2 + c_ie_3) = a + be_1e_2 + ce_2e_3 + de_3e_1\]
and each element $h \in \Spin{4}$ of the form
\[h = \prod^{2k}_{i=1} (a_ie_1 + b_ie_2 + c_ie_3 + d_ie_4) = h_1 + h_2e_1e_2 + h_3e_2e_3 + h_4e_3e_1 + h_5e_1e_2e_3e_4 + h_6e_4e_3 + h_7e_4e_1 + h_8e_4e_2.\]
\end{remark}

\section{The isomorphism $\Spin{4}\cong\Spin{3}\times\Spin{3}$}

In this section we recall special isomorphisms admitted by the groups $\Spin{3}$ and $\Spin{4}$. This structural information will only become relevant in Part III (Sections~\ref{strategy} and \ref{sec:g2}) of this article.

\begin{definition}
Denote by \[\HH := \{ a+bi+cj+dk \mid a, b, c, d \in \RR \}\] the \Defn{real quaternions}, identify $\RR$ with the centre of $\HH$ via $\RR \to \HH : a \mapsto a$, let \[\bar{\cdot} : \HH \to \HH : x = a +bi +cj + dk \mapsto \overline{x} = a -bi -cj -dk\] be the \Defn{standard involution}, and let \[\mathrm{U}_1(\HH) := \{ x \in \HH \mid x\overline{x} = 1_\HH\}\] be the group of unit quaternions. 
\end{definition}

\begin{remark}\label{Spin4Spin3}
By \cite[Section~1.4]{Gallier} one has \[\Spin{3} \cong \mathrm{U}_1(\HH) \quad\quad \text{and} \quad\quad \Spin{4} \cong \Spin{3} \times \Spin{3} \cong \mathrm{U}_1(\HH) \times \mathrm{U}_1(\HH).\]

The isomorphism $\Spin 3 \cong \mathrm{U}_1(\HH)$ in fact is an immediate consequence of the isomorphism $\Cle{3}{0}\cong \HH$ from Remark~\ref{rem:cl3=quaternions}(b) plus the observation that this isomorphism transforms the spinor norm into the norm of the quaternions. 

  A canonical isomorphism $\Spin{4}  \cong \Spin{3} \times \Spin{3} \cong \mathrm{U}_1(\HH) \times \mathrm{U}_1(\HH)$ can be described as follows (see \cite[Section~1.4]{Gallier}). By Remark~\ref{coordinatesforspin} each element of $\Spin 4$ is of the form \[a + be_1e_2 + ce_2e_3 + d e_3e_1 + a'e_1e_2e_3e_4 + b'e_4e_3 + c'e_4e_1 + d'e_4e_2.\]
For \[i:=e_1e_2, \quad j:=e_2e_3, \quad k:=e_3e_1, \quad \II:=e_1e_2e_3e_4, \quad i':=e_4e_3, \quad j':=e_4e_1, \quad k':=e_4e_2\] one has
\begin{align*}
ij &= k\ ,& jk &= i\ , & ki &= j\ , \\
i\II &= \II i = i'\ , & 
j\II &= \II j = j'\ , & 
k\II &= \II k = k'\ , \\
i^2 &= j^2 = k^2 = -1\ , & 
\II^2 &= 1\ , & 
\sigma(\II) &= \II\ ,
\end{align*}
where $\sigma(\II)$ denotes the Clifford conjugate of $\II$, cf.\ Definition~\ref{def:defcliffordconj}. We conclude that for every $x \in \Spin 4$ there exist uniquely determined $u = a + bi + cj + dk, v = a' + b'i + c'j + d'k \in \HH$ such that \[x = u + \II v.\]
One computes
\[N(x) = N(u+\II v) = (u+\II v)(\overline{u}+\II\overline{v}) = u\overline{u}+v\overline{v}+\II(u\overline{v}+v\overline{u}),\] i.e., \[N(x) = 1 \Longleftrightarrow u\overline{u} + v\overline{v} = 1 \text{ and } u\overline{v}+v\overline{u}=0.\]
Hence, for $1=N(x)=N(u+\II v)$, one has
\begin{eqnarray*}
N(u+v) & = & (u+v)(\overline{u}+\overline{v}) = 1, \\
N(u-v) & = & (u-v)(\overline{u}-\overline{v}) = 1.
\end{eqnarray*} 
That is, the map \[\Spin{4} \to \Spin 3 \times \Spin 3 : u+\II v \mapsto (u+v,u-v)\] is a well-defined bijection and, since
\begin{eqnarray*}
(u+\II v)(u'+\II v') & = & uu'+vv'+\II(uv'+vu'), \\
(u+v,u-v)(u'+v',u'-v') & = & \left(uu'+vv'+uv'+vu',uu'+vv'-(uv'+vu')\right),
\end{eqnarray*}
 in fact an isomorphism of groups.

Consequently, there exist a group epimorphism \[\tilde\eta : \Spin{4} \to \Spin{3} : u+\II v \mapsto u+v. \] 
\end{remark}

\begin{remark}\label{mapsso4}
Using this isomorphism $\Spin{4} \cong \Spin{3} \times \Spin{3} \cong \mathrm{U}_1(\HH) \times \mathrm{U}_1(\HH)$ there exists a natural homomorphism \[\Spin{4} \to \SO{\HH} \cong \SO{4} : (a,b) \mapsto \left(x \mapsto axb^{-1} \right).\]  
Note that the restrictions $(a,1) \mapsto \left( x \mapsto ax \right)$ and $(1,b) \mapsto \left( x \mapsto xb^{-1} \right)$ both are injections of $\Spin{3} \cong \mathrm{U}_1(\HH)$ into $\mathrm{GL}(\HH) \cong (\HH \backslash \{ 0 \}, \cdot) $, in fact into $\SO{\HH}$, as the norm is multiplicative. Since the kernel of this action has order two, the homomorphism $\Spin{4} \to \SO{\HH} \cong \SO{4}$ must be onto by Proposition~\ref{rho}. We conclude that the group $\SO{4}$ is isomorphic to the group consisting of the maps \[\HH \to \HH : x \mapsto axb^{-1} \qquad \text{for $a, b \in \mathrm{U}_1(\HH)$;}\] for an alternative proof see \cite[Lemma~11.22]{Salzmann:1995}.

A similar argument (or a direct computation using the twisted adjoint representation) shows that the natural homomorphism  \[\Spin{3} \to \SO{\langle i, j, k \rangle_\RR} \cong \SO{3} : a \mapsto \left(x \mapsto axa^{-1} \right)\] is an epimorphism and, thus, that the group $\SO{3}$ is isomorphic to the group consisting of the maps \[\HH \to \HH : x \mapsto axa^{-1} \qquad \text{for $a \in \mathrm{U}_1(\HH)$;}\]
see also \cite[Lemma~11.24]{Salzmann:1995}.
\end{remark}

\begin{remark} \label{mapisrho}
There also exists a group epimorphism \[\eta : \SO{4} \to \SO{3}\] induced by the map
\begin{align*}
\SO{4} \cong \{ \HH \to \HH : x \mapsto axb^{-1} \mid a, b \in \mathrm{U}_1(\HH) \}
  &\ \to\  \{ \HH \to \HH \mid x \mapsto axa^{-1} \mid a \in \mathrm{U}_1(\HH) \} \cong \SO{3} \\
  (x \mapsto axb^{-1}) &\ \mapsto\  (x \mapsto axa^{-1}).
\end{align*}

Altogether, one obtains the following commutative diagram:
\[\xymatrix{
\Spin{4} \ar[rr]_{\tilde\eta} \ar[d]^{\rho_4} & & \Spin{3} \ar[d]_{\rho_3} \\
 \SO{4} \ar[rr]^\eta & & \SO{3}
}\]
\end{remark}

\section{Lifting automorphism from $\SO{n}$ to $\Spin{n}$} \label{sec:7}

\begin{notation} \label{nota:Dalpha-Salpha}\label{notationiota}\label{iotaspin}
For $\SO{2}\times\SO{2} = \{ (a,b) \mid a,b\in\SO{2} \}$ let
\[\iota_1:\SO{2}\to \SO{2}\times\SO{2} : x \mapsto (x,1)\ ,\qquad
  \iota_2:\SO{2}\to \SO{2}\times\SO{2} : x \mapsto (1,x)\ .\]
Similarly, for $\Spin{2}\times\Spin{2} = \{ (a,b) \mid a,b\in\Spin{2} \}$ let
\[\tilde\iota_1:\Spin{2}\to \Spin{2}\times\Spin{2} : x \mapsto (x,1)\ ,\qquad
  \tilde\iota_2:\Spin{2}\to \Spin{2}\times\Spin{2} : x \mapsto (1,x)\ .\]
Moreover, define
\[\rho_2\times\rho_2:\Spin{2}\times\Spin{2}\to\SO{2}\times\SO{2} :
   (a,b)\mapsto (\rho_2(a),\rho_2(b))\ .\]
Hence
\[
 (\rho_2\times\rho_2)\circ\tilde\iota_1=\iota_1\circ\rho_2\ ,\qquad
 (\rho_2\times\rho_2)\circ\tilde\iota_2=\iota_2\circ\rho_2\ .
\]
Furthermore, let \[\pi : \Spin{2} \times \Spin{2} \to \Spin{2} \times \Spin{2}/ \langle (-1,-1) \rangle\] be the canonical projection. By the homomorphism theorem of groups the map $\rho_2 \times \rho_2$ factors through $\Spin{2} \times \Spin{2}/ \langle (-1,-1) \rangle$ and induces the following commutative diagram:
\[
\xymatrix{
\Spin{2} \times \Spin{2} \ar[rr]^{\rho_2\times\rho_2} \ar[d]^{\pi} && \SO{2}\times\SO{2} \\
\Spin{2} \times \Spin{2}/ \langle (-1,-1) \rangle \ar[urr]^{\rho_2.\rho_2}
}
\]
For $\alpha\in\RR$ let 
\[D(\alpha):=\begin{pmat} \cos(\alpha) & \sin(\alpha) \\
-\sin(\alpha) & \cos (\alpha)\end{pmat} \in \SO{2}
\quad\text{ and }\quad
S(\alpha):= \cos(\alpha)+\sin(\alpha)e_1e_2 \in \Spin{2}.\]
Then $\Spin{2}=\{ S(\alpha) \mid \alpha\in \RR\}$
and $\SO{2}=\left\{ D(\alpha) \mid \alpha\in\RR\right\}$ and there is a continuous group isomorphism
\[\psi:\SO{2}\to \Spin{2} : D(\alpha)\mapsto S(\alpha).\]
By the computation in Remark~\ref{coordinatesforspin} the epimorphism $\rho_2$ from Theorem~\ref{rho}
satisfies $\rho_2=\sq\circ\psi^{-1}$, i.e.
\[ \rho_2:\Spin{2}\to \SO{2} : S(\alpha)\mapsto D(2\alpha). \]
\end{notation}

\begin{proposition} \label{prop:lift-aut-so2}
Given an automorphism $\gamma\in \Aut(\SO{2})$, there is a unique automorphism $\tilde\gamma\in \Aut(\Spin{2})$ such that $\rho_2\circ \tilde\gamma=\gamma\circ \rho_2$.
Moreover, $\gamma$ is continuous if and only if $\tilde\gamma$ is continuous.
\end{proposition}

\begin{proof}
Define $\tilde\gamma:=  \psi\circ \gamma\circ\psi^{-1}$. Then
\[ \rho_2\circ \tilde\gamma
 = (\sq\circ\psi^{-1})\circ (\psi\circ \gamma\circ\psi^{-1})
 = \sq\circ\gamma\circ\psi^{-1}
 = \gamma\circ\sq\circ\psi^{-1}
 = \gamma\circ\rho_2\ .
\]
Uniqueness follows as $\Aut(\SO{2})\to\Aut(\Spin{2}) : \gamma\mapsto \psi\circ \gamma\circ\psi^{-1}$
is an isomorphism.
\end{proof}

\begin{corollary}\label{prop:lift-aut-so2xso2}
Given an automorphism $\gamma\in \Aut(\SO{2}\times\SO{2})$,
there is a unique automorphism $\tilde\gamma\in \Aut(\Spin{2}\times\Spin{2})$ such that
\[(\rho_2\times\rho_2)\circ \tilde\gamma=\gamma\circ (\rho_2\times\rho_2).\]
\end{corollary}

\begin{proof}
Let $\tilde\gamma := \psi\circ\gamma\circ\psi^{-1}$, where
\[ \psi: \SO{2}\times\SO{2}\to\Spin{2}\times\Spin{2},\ 
	(D(\alpha),D(\beta)) \mapsto (S(\alpha), S(\beta))\ , \]
and observe that $\rho_2\times\rho_2 = \sq\circ\psi^{-1}$. Proceed
as in the proof of Proposition~\ref{prop:lift-aut-so2}.
\end{proof}

\begin{proposition}\label{prop:lift-aut-soN}
Let $n\geq 3$. Given an automorphism $\gamma\in \Aut\big(\SO{n}\big)$,
there is a unique automorphism $\tilde\gamma\in \Aut\big(\Spin{n}\big)$ such that
\[\rho_n\circ \tilde\gamma=\gamma\circ \rho_n.\]
\end{proposition}

\begin{proof}
For $n\geq 3$, both $\SO{n}$ and $\Spin{n}$ are perfect, cf.\ 
\cite[Corollary~6.56]{Hofmann/Morris:1998}. By Theorem~\ref{rho}(b) the group $\Spin{n}$ is a central extension of $\SO{n}$. Since $\Spin{n}$ is simply connected (see, e.g., \cite[Section~1.8]{Gallier}, it in fact is the universal central extension of $\SO{n}$.

The universal property of universal central extensions (cf.\ e.g.\ \cite[Section~1.4C]{Hahn/OMeara:1989}) yields the claim: Indeed, there are unique homomorphisms $\tilde\gamma, \tilde\gamma' : \Spin{n} \to \Spin{n}$ such that \[\gamma \circ \rho_n = \rho_n \circ \tilde \gamma \quad \text{and} \quad \gamma^{-1} \circ \rho_n = \rho_n \circ \tilde\gamma'.\]   
Hence
\[\rho_n \circ \tilde\gamma \circ \tilde\gamma' = \gamma \circ \rho_n \circ \tilde\gamma' = \gamma \circ \gamma^{-1} \circ \rho_n = \rho_n\] and, similarly, \[\rho_n \circ \tilde\gamma' \circ \tilde\gamma = \rho_n.\]
The universal property therefore implies $\tilde\gamma \circ \tilde\gamma' = \id = \tilde\gamma' \circ \tilde\gamma$, i.e., $\tilde\gamma$ is an automorphism.

In fact, all automorphisms are continuous by van der
Waerden's Continuity Theorem, cf.\ \cite[Theorem~5.64]{Hofmann/Morris:1998}.
\end{proof}

For the following proposition recall the definitions of $\eps_{ij}$ in Notation~\ref{nota:EI-VI-qI} and of $\tilde\eps_{ij}$ in Lemma~\ref{2}.

\begin{proposition} \label{prop:zusatz}
Let $\phi : \Spin{2} \to \Spin{n}$ be a homomorphism such that
\[\ker(\rho_n \circ \phi) = \{ 1, -1 \} \qquad \text{and} \qquad \rho_n \circ \phi = \eps_{ij} \circ \rho_2\] for some $i \neq j \in I$. Then $\phi = \tilde\eps_{ij}$.
\end{proposition}

\begin{proof}
By Consequence~\ref{7} one has \[\phi(\Spin{2}) \subseteq ({\rho_n}^{-1} \circ \eps_{ij} \circ \rho_2)(\Spin{2}) = {\rho_{n}}^{-1}(\eps_{ij}(\SO{2})) = \tilde\eps_{ij}(\Spin{2}).\]
By hypothesis $\ker\phi \subseteq \{ 1, -1\}$. If $-1 \in \ker\phi$, then $1 = \phi(-1) = \phi(S(\pi)) = \phi(S(\frac{\pi}{2}))^2$, i.e., $\phi(S(\frac{\pi}{2})) \in \{ 1, -1\}$, whence $S(\frac{\pi}{2}) \in \ker(\rho_n \circ \phi)$, a contradiction. Consequently,  $\phi$ is a monomorphism.

Consider the following commuting diagram:
\[
\xymatrix{ && \Spin{n} \\
\Spin{2} \ar[rr]^\phi \ar[d]_{\rho_2} && \phi(\Spin{2}) \ar@{^{(}->}[u]  \ar[rr]^{{\tilde\eps_{ij}}^{-1}} \ar[d]_{\rho_n} && \Spin{2} \ar[d]_{\rho_2} \\
\SO{2} \ar[rr]_{\eps_{ij}} && \eps_{ij}(\SO{2}) \ar@{^{(}->}[d] \ar[rr]_{{\eps_{ij}}^{-1}} && \SO{2} \\ && \SO{n}
}
\]
One has $\rho_2 \circ {\tilde\eps_{ij}}^{-1} \circ \phi = \rho_2 = \id \circ \rho_2$.
Since $\phi$ is injective, the map ${\tilde\eps_{ij}}^{-1} \circ \phi$ is an automorphism of $\Spin{2}$. Hence Proposition~\ref{prop:lift-aut-so2} implies ${\tilde\eps_{ij}}^{-1} \circ \phi = \id$.
\end{proof}

\part{Simply laced diagrams}

\section{$\SO{2}$-amalgams of simply laced type} \label{sec:so2amalgams}

In this section we discuss amalgamation results for compact real orthogonal groups. The results and exposition are similar to \cite{Borovoi:1984}, \cite{Gramlich:2006}. The key difference is that the amalgams in the present article are constructed starting with the circle group $\SO{2}$ instead of the perfect group $\mathrm{SU}(2)$. This leads to some subtle complications that we will need to address below.

Recall the maps $\eps_{12}, \eps_{23} : \SO{2} \to \SO{3}$ from Notation~\ref{nota:EI-VI-qI} and the maps $\iota_1, \iota_2 : \SO{2}\to \SO{2}\times \SO{2}$ from Notation~\ref{notationiota}.

\begin{definition} \label{defstandardso}
Let $\Pi$ be a simply laced diagram with labelling $\sigma:I\to V$.
An \Defn{$\SO{2}$-amalgam with respect to $\Pi$ and $\sigma$} is an amalgam
$\AAA=\{ G_{ij},\; \phi_{ij}^i \mid i\neq j\in I \}$
such that
\[\forall\ i\neq j\in I: \qquad
G_{ij}=\begin{cases}
\SO{3}, & \text{if }\{i,j\}^\sigma\in E(\Pi), \\
\SO{2}\times \SO{2}, & \text{if }\{i,j\}^\sigma\notin E(\Pi),
\end{cases}\]
and for $i<j\in I$,
\begin{align*}
\phi_{ij}^i\big(\SO{2}\big)=\begin{cases}
\eps_{12}\big(\SO{2}\big), & \text{if  $\{i,j\}^\sigma\in E(\Pi)$}, \\
\iota_1\big(\SO{2}\big), & \text{if }\{i,j\}^\sigma\notin E(\Pi),
\end{cases} &&
\phi_{ij}^j\big(\SO{2}\big)=\begin{cases}
\eps_{23}\big(\SO{2}\big), & \text{if  $\{i,j\}^\sigma\in E(\Pi)$}, \\
\iota_2\big(\SO{2}\big), & \text{if }\{i,j\}^\sigma\notin E(\Pi).
\end{cases}
\end{align*}
The \Defn{standard $\SO{2}$-amalgam with respect to $\Pi$ and $\sigma$} is the $\SO{2}$-amalgam
\[\AAA\big(\Pi,\sigma,\SO{2}\big):=\{ G_{ij},\; \phi_{ij}^i \mid i\neq j\in I \}\]
with respect to $\Pi$ and $\sigma$ with
\begin{align*} \forall\ i< j\in I: &&
\phi_{ij}^i=\begin{cases}
\eps_{12},&\text{if }\{i,j\}^\sigma\in E(\Pi), \\
\iota_1, &\text{if }\{i,j\}^\sigma\notin E(\Pi).
\end{cases} &&
\phi_{ij}^j=\begin{cases}
\eps_{23},&\text{if }\{i,j\}^\sigma\in E(\Pi), \\
\iota_2, &\text{if }\{i,j\}^\sigma\notin E(\Pi).
\end{cases}
\end{align*}
\end{definition}

\begin{remark} \label{rem:continuous}
The key difference between the standard $\SO{2}$-amalgam and an arbitrary $\SO{2}$-amalgam $\AAA=\{ G_{ij},\; \phi_{ij}^i \mid i\neq j\in I \}$ with respect to $\Pi$ and $\sigma$ is that, for instance, $\eps_{12}^{-1} \circ \phi_{ij}^i$ can be an arbitrary automorphism of $\SO{2}$.
Automatic continuity (like van der Waerden's Continuity Theorem, cf.\ \cite[Theorem~5.64 and Corollary~6.56]{Hofmann/Morris:1998}) fails for automorphisms  of the circle group $\SO{2}$ whereas it does hold for the group $\SO{3}$. 
Hence, obviously, not every automorphism of $\SO{2}$ is induced by an automorphism of $\SO{3}$ and so it is generally not possible to undo the automorphism $\eps_{12}^{-1} \circ \phi_{ij}^i$ inside $\SO{3}$.
Therefore Goldschmidt's Lemma (see \cite[Lemma~2.7]{Goldschmidt:1980}, also \cite[Proposition~8.3.2]{Ivanov/Shpectorov:2002}, \cite[Lemma~6.16]{Gloeckner/Gramlich/Hartnick:2010}) implies that for each diagram $\Pi$ there exist plenty of pairwise non-isomorphic abstract $\SO{2}$-amalgams.

However, by \cite[Section~4.G]{Kac/Peterson:1983}, \cite[Corollary~7.16]{Hartnick/Koehl/Mars}, a split real Kac--Moody group and its maximal compact subgroup (i.e., the group of elements fixed by the Cartan--Chevalley involution) both carry natural group topologies that induce the Lie group topology on their respective fundamental subgroups of ranks one and two and make the respective embeddings continuous.

It is therefore meaningful to use \emph{continuous} $\SO{2}$-amalgams for studying these maximal compact subgroups.
Such continuous amalgams are uniquely determined by the underlying diagram $\Pi$, as we will see in Theorem~\ref{thm:uniqueness-so-sl} below.
\end{remark}

\begin{convention} \label{fix}
For each group isomorphic to one of $\SO{2}$, $\SO{3}$, $\SO{2} \times \SO{2}$, we fix a matrix representation that allows us to identify the respective groups accordingly. Our study of amalgams by Goldschmidt's Lemma then reduces to the study of automorphisms of these groups.
\end{convention}

\begin{lemma}\label{10}
Let
\[D:=\begin{pmatrix} & & 1 \\ & -1 & \\ 1 & & \end{pmatrix}\in \SO{3}\ .\]
Then the map
$\gamma_D:\SO{3}\to \SO{3} : A\mapsto D\cdot A\cdot D^{-1}=D\cdot A\cdot D$
is an automorphism of $\SO{3}$ such that \[\gamma_D\circ \eps_{12}=\eps_{23} \qquad \text{and} \qquad \gamma_D\circ \eps_{23}=\eps_{12}.\]
\end{lemma}

\begin{proof}
Given $\begin{pmat} x & y \\ -y & x\end{pmat}\in \SO{2}$, we have
\begin{align*}
\begin{pmatrix} & & 1 \\ & -1 & \\ 1 & & \end{pmatrix}\cdot \begin{pmatrix} x & y & \\
-y & x & \\
& & 1 \end{pmatrix}\cdot \begin{pmatrix} & & 1 \\ & -1 & \\ 1 & & \end{pmatrix}=\begin{pmatrix}  &  & 1\\
y& -x & \\
x& y&  \end{pmatrix}\cdot \begin{pmatrix} & & 1 \\ & -1 & \\ 1 & & \end{pmatrix}=\begin{pmatrix}1 & &  \\ & x & y \\  & -y& x\end{pmatrix}\ .
\end{align*}
The second assertion follows analogously.
\end{proof}

The only influence of the labelling $\sigma$ of an amalgam is the choice which of the vertices $i^\sigma$, $j^\sigma$ corresponds to which subgroup of $G_{ij}$. We now show that this choice does not affect the isomorphism type of the amalgam. 

\begin{consequence}\label{14}
Let $\Pi$ be a simply laced diagram with labellings $\sigma_1,\sigma_2:I\to V$. Then
\[\AAA\big(\Pi,\sigma_1,\SO{2}\big)\cong \AAA\big(\Pi,\sigma_2,\SO{2}\big).\]
\end{consequence}

\begin{proof}
Denote $\AAA:=\AAA\big(\Pi,\sigma_1,\SO{2}\big)$ and $\ol{\AAA}:= \AAA\big(\Pi,\sigma_2,\SO{2}\big)$.
Let $D\in \SO{3}$ be as in Lemma~\ref{10} and let $\pi:=\sigma_2^{-1}\circ\sigma_1\in \Sym(I)$. Notice that
\begin{align*}
\ol{G}_{\pi(i)\pi(j)}=\SO{3}\ &\Leftrightarrow\ \{\pi(i),\pi(j)\}^{\sigma_2}\in E(\Pi)\ \Leftrightarrow\ \{ i,j\}^{\pi\sigma_2}\in E(\Pi) \\
&\Leftrightarrow\ \{ i,j\}^{\sigma_1\sigma_2^{-1}\sigma_2}\in E(\Pi)\ \Leftrightarrow\ \{ i,j\}^{\sigma_1}\in E(\Pi)\ \Leftrightarrow\ G_{ij}=\SO{3}\ ..
\end{align*}
Given $i<j\in I$ with $\{ i,j\}^{\sigma_1} \in E(\Pi)$, let
\[\alpha_{ij}:=\begin{cases}
\id_{\SO{3}},&\text{if }\pi(i)<\pi(j), \\
\gamma_D,&\text{if }\pi(i)>\pi(j),
\end{cases}\] and given $i<j \in I$ with $\{ i,j\}^{\sigma_1} \not\in E(\Pi)$, let
\[\alpha_{ij}: \SO{2} \times \SO{2} \to \SO{2} \times \SO{2},
(x,y) \mapsto
  \begin{cases}
  (x,y),&\text{if }\pi(i)<\pi(j), \\
  (y,x),&\text{if }\pi(i)>\pi(j).
  \end{cases}
\]

Then the system
$\alpha:=\{ \pi, \alpha_{ij} \mid i\neq j\in I \}:\AAA\to\ol{\AAA}$
is an isomorphism of amalgams. Indeed, given $i<j\in I$ with $\{ i,j\}^{\sigma_1} \in E(\Pi)$, one has
\begin{align*}
\alpha_{ij}\circ \phi_{ij}^i=\alpha_{ij}\circ \eps_{12}=
\begin{cases}
\id_{\SO{3}}\circ\eps_{12}=\eps_{12}=\ol{\phi}_{\pi(i)\pi(j)}^{\pi(i)},&\text{if }\pi(i)<\pi(j), \\
\gamma_D\circ \eps_{12}=\eps_{23}=\ol{\phi}_{\pi(i)\pi(j)}^{\pi(j)} ,&\text{if }\pi(i)>\pi(j),
\end{cases}
\end{align*}
and
\begin{align*}
\alpha_{ij}\circ \phi_{ij}^j=\alpha_{ij}\circ \eps_{23}=
\begin{cases}
\id_{\SO{3}}\circ\eps_{23}=\eps_{23}=\ol{\phi}_{\pi(i)\pi(j)}^{\pi(i)},&\text{if }\pi(i)<\pi(j), \\
\gamma_D\circ \eps_{23}=\eps_{12}=\ol{\phi}_{\pi(i)\pi(j)}^{\pi(j)},&\text{if }\pi(i)>\pi(j).
\end{cases}
\end{align*}
The case $i<j\in I$ with $\{ i,j\}^{\sigma_1} \not\in E(\Pi)$ is verified similarly.
\end{proof}

\begin{definition}\label{def:std-ama-SO2}
As we have just seen, the labelling of a standard $\SO{2}$-amalgam is
irrelevant for its isomorphism type. Hence,
for a simply laced diagram $\Pi$, we write
$\AAA\big(\Pi,\SO{2}\big)$ to denote this isomorphism type and, moreover, by slight abuse of notation to denote any representative $\AAA\big(\Pi,\sigma,\SO{2}\big)$ of this isomorphism type. It is called
the \Defn{standard $\SO{2}$-amalgam with respect to $\Pi$}.
\end{definition}

\begin{lemma}\label{lem:rank1-inv-A2}
Let $B:=\begin{pmat} 1 & \\ & -I_2\end{pmat}$, $C:=\begin{pmat} -I_2 & \\ & 1\end{pmat}\in \SO{3}$. Then the following holds:
\begin{enumerate}
\item The map $\gamma_B:\SO{3}\to \SO{3} : A\mapsto B\cdot A\cdot B^{-1}$
is an automorphism of $\SO{3}$ such that
\[\gamma_B\circ \eps_{12}=\eps_{12}\circ \inv \qquad \text{and} \qquad \gamma_B\circ \eps_{23}=\eps_{23}.\]
\item The map $\gamma_C:\SO{3}\to \SO{3} : A\mapsto C\cdot A\cdot C^{-1}$
is an automorphism of $\SO{3}$ such that
\[\gamma_C\circ \eps_{12}=\eps_{12} \qquad \text{and} \qquad  \gamma_C\circ \eps_{23}=\eps_{23}\circ\inv.\]
\end{enumerate}
\end{lemma}

\begin{proof}
\begin{enumerate}
\item Given $A:= \begin{pmat} x & y \\ -y & x\end{pmat}\in \SO{2}$, we have
\begin{align*}
\gamma_B( \eps_{12}(A))
= B \cdot \begin{pmatrix} x & y & \\
-y & x & \\
& & 1 \end{pmatrix}\cdot B^{-1}
=\begin{pmatrix} x& -y &  \\ y& x & \\  & & 1 \end{pmatrix}
= \eps_{12}(A^{-1})\ .
\end{align*}
The second statement follows analogously.
\item is shown with a similar computation. \qedhere
\end{enumerate}
\end{proof}

\begin{theorem}\label{thm:uniqueness-so-sl}\label{cons:std-ams-SO2-iso}
Let $\Pi$ be a simply laced diagram with labelling $\sigma:I\to V$ and 
let $\AAA=\{ G_{ij},\; \phi_{ij}^i \mid i\neq j\in I \}$
be a continuous $\SO{2}$-amalgam with respect to $\Pi$ and $\sigma$. Then
$\AAA\cong \AAA\big(\Pi,\sigma,\SO{2}\big)$.
\end{theorem}

\begin{proof}
Denote $\ol{\AAA}:=\AAA\big(\Pi,\sigma,\SO{2}\big)$.
The only continuous automorphisms of the circle group $\SO{2}$
are $\id$ and the inversion $\inv$.
Since $\AAA$ is continuous by hypothesis, for all $i<j \in I$ with $\{i,j\}^\sigma \in E(\Pi)$ we have
\begin{align*}
\phi_{ij}^i\in \{ \eps_{12}, \eps_{12}\circ \inv\}\ , && \phi_{ij}^j\in \{ \eps_{23}, \eps_{23}\circ \inv\}\ .
\end{align*}
Let $B,C\in \SO{3}$ be as in Lemma~\ref{lem:rank1-inv-A2}, let $\pi:=\id_{I}$, and given $i>j\in I$ such that $\{i,j\}^\sigma\in E(\Pi)$, let
\[\alpha_{ij}:=\begin{cases}
\id_{\SO{3}},&\text{if }\phi_{ij}^i=\eps_{12},\hphantom{{}\circ\inv}\ \phi_{ij}^j=\eps_{23}, \\
\gamma_B, &\text{if }\phi_{ij}^i=\eps_{12}\circ\inv,\ \phi_{ij}^j=\eps_{23}, \\
\gamma_C, &\text{if }\phi_{ij}^i=\eps_{12},\hphantom{{}\circ\inv}\ \phi_{ij}^j=\eps_{23}\circ \inv, \\
\gamma_B\circ \gamma_C, &\text{if }\phi_{ij}^i=\eps_{12}\circ\inv,\ \phi_{ij}^j=\eps_{23}\circ \inv.
\end{cases}\]

For $i<j \in I$ with $\{i,j\}^\sigma \not\in E(\Pi)$, define $\alpha_{ij} := \big(\iota_1 \circ( \phi_{ij}^i)^{-1}\big) \times \big(\iota_2 \circ( \phi_{ij}^j)^{-1}\big)$.
Then the system
$\alpha:=\{\pi, \alpha_{ij} \mid i\neq j\in I\}:\AAA\to \ol{\AAA}$
is an isomorphism of amalgams.
\end{proof}

The following is well-known, e.g.\ \cite[Theorem~1.2]{Medts/Gramlich/Horn}.

\begin{theorem} \label{thm:univso}
For $n\geq3$, the group $\SO{n}$ is a universal enveloping group of $\AAA\big(A_{n-1},\SO{2}\big)$.
\end{theorem}

\begin{proof}
Let $I:=\{1,\ldots,n-1\}$. The group $\SO{n}$ acts flag-transitively on the simply connected projective geometry $\mathcal{G} := \PP_{n-1}(\RR)$; simple connectedness follows from \cite[Theorem~13.32]{Tits:1974}, \cite[Theorem~2]{Tits:1981}, \cite[Proposition~11.1.9, Theorem~11.1.13]{Buekenhout/Cohen:2013}, flag-transitivity of the action from the Iwasawa/$QR$-decomposition of $\SL{n}$. A maximal flag is given by
\[\gen{e_1}_\RR \leq \gen{e_1,e_2}_\RR \leq \cdots \leq \gen{e_1,\ldots,e_{n-1}}_\RR.\]
Let $T$ be the subgroup of $\SO{n}$ of diagonal matrices; it is isomorphic to $C_2^{n-1}$ (where $C_2$ is a cyclic group of order 2). Moreover, for $1 \leq i \leq n-1$, let $H_i \cong \SO{2}$ be the circle group acting naturally on $\langle e_i, e_{i+1}\rangle_\RR$ and, for $1 \leq i \leq n-2$, let $H_{i,i+1} \cong \SO{3}$ be the group acting naturally on $\langle e_i,e_{i+1},e_{i+2}\rangle_\RR$. Furthermore, for $1 \leq i < j-1 \leq n-2$, let $H_{ij} := H_i H_j \cong \SO{2} \times \SO{2}$.   
Then the stabilizer of a sub-flag of co-rank one is of the form $H_iT$, $1 \leq i \leq n-1$, and the stabilizer of a sub-flag of co-rank two is of the form $H_{ij}T$, $1 \leq i < j \leq n-1$. 

The group $T \cong C_2^{n-1}$ admits a presentation with all generators and relations contained in the rank two subgroups $H_{ij}$ of $\SO{n}$: Indeed, $T$ is generated by the groups $C_2 \cong T_i := \gen{-1} \leq H_i \cong \SO{2}$ and for each $1 \leq i \neq j \leq n-1$ the relation $T_iT_j = T_jT_i$ is visible within $H_{ij}$.
Therefore, by an iteration of Tits's Lemma (see \cite[Corollary~1]{Tits:1986}, \cite[Corollary~1.4.6]{Ivanov/Shpectorov:2002}) with respect to the above maximal flag, the group $H:=\SO{n}$ is the universal enveloping group of the amalgam
$\AAA(\GGG, H)=\{ H_{ij},\; \Phi_{ij}^i \mid i\neq j\in I\}$,
where $\Phi_{ij}^i:H_i\to H_{ij}$ is the inclusion map for each $i\neq j\in I$.
One has
\[\forall\ i\in I:\qquad H_i=\eps_{\{i, i+1\}}\big(\SO{2}\big)\ ,\]
and
\[\forall\ i<j\in I:\qquad H_{ij}=\begin{cases}
\eps_{\{i,i+1,i+2\}}\big(\SO{3}\big),&\text{if }j=i+1, \\
\eps_{\{i,i+1\}}\big(\SO{2}\big)\times\eps_{\{j,j+1\}}\big(\SO{2}\big),&\text{if }j\neq i+1.
\end{cases}\]
As a consequence, the system
\[\alpha=\{ \id_I,\alpha_{ij},\alpha_i \mid i\neq j\in I\}: \AAA\big(A_{n-1},\SO{2}\big)\to \AAA(\GGG,H)\]
with
\[\forall\ i\in I:\qquad \alpha_i=\eps_{\{i,i+1\}}:\SO{2}\to H_i\]
and
\[
\forall\ i<j \in I:\qquad \alpha_{ij}=\begin{cases}
\eps_{\{i,i+1,i+2\}},&\text{if }j=i+1, \\
\eps_{\{i,i+1\}}\times\eps_{\{j,j+1\}}, &\text{if }j\neq i+1,
\end{cases}
\]
is an isomorphism of amalgams.
\end{proof}

\begin{remark}
The above proof mainly relies on geometric arguments in the Tits building of type $A_{n-1}$. We exploit this to generalize the above statements to other diagrams, see Theorems~\ref{thm:K-univ-sl} and \ref{thm:K-univ-adm}. The crucial observation to make is that --- via the local-to-global principle --- it basically suffices to understand the rank two situation in order to understand arbitrary types.
\end{remark}

\section{$\Spin{2}$-amalgams of simply laced type} \label{sec:spin2amalgams}

In analogy to Section~\ref{sec:so2amalgams} we now study the amalgamation of groups isomorphic to $\Spin{3}$, continuously glued to one another along circle groups. In particular, we describe $\Spin{n}$ as the universal enveloping group of its $\Spin{2}$-amalgam and relate the classification of continuous $\Spin{2}$-amalgams to the classification of continuous $\SO{2}$-amalgams via the lifting of automorphisms.

Recall the maps $\tilde\eps_{12}, \tilde\eps_{23} : \Spin{2} \to \Spin{3}$ from Lemma~\ref{2} and the maps $\tilde\iota_1, \tilde\iota_2 : \Spin{2} \to \Spin{2}\times \Spin{2}$ from Notation~\ref{iotaspin}.

\begin{definition}\label{18}
Let $\Pi$ be a simply laced diagram with labelling $\sigma:I\to V$.
A \Defn{$\Spin{2}$-amalgam with respect to $\Pi$ and $\sigma$} is an amalgam $\AAA=\{ G_{ij},\; \phi_{ij}^i \mid i\neq j\in I \}$ such that
\[\forall\ i\neq j\in I: \qquad
G_{ij}=\begin{cases}
\Spin{3},&\text{if }\{i,j\}^\sigma\in E(\Pi), \\
\Spin{2}\times\Spin{2},&\text{if }\{i,j\}^\sigma\notin E(\Pi), 
\end{cases}\]
and for $i<j\in I$,
\begin{align*}
\phi_{ij}^i\big(\Spin{2}\big)
    =\begin{cases}
        \tilde\eps_{12}\big(\Spin{2}\big),&\text{if  $\{i,j\}^\sigma\in E(\Pi)$}, \\
        \iota_1(\Spin{2}),&\text{if }\{i,j\}^\sigma\notin E(\Pi),
    \end{cases}
&& \phi_{ij}^j\big(\Spin{2}\big)
    =\begin{cases}
        \tilde\eps_{23}\big(\Spin{2}\big),&\text{if  $\{i,j\}^\sigma\in E(\Pi)$}, \\
        \iota_2(\Spin{2}),&\text{if }\{i,j\}^\sigma\notin E(\Pi).
    \end{cases}
\end{align*}
The \Defn{standard $\Spin{2}$-amalgam with respect to $\Pi$ and $\sigma$} is the (continuous) $\Spin{2}$-amalgam \[\AAA\big(\Pi,\sigma,\Spin{2}\big):=\{ G_{ij},\; \phi_{ij}^i, \mid i\neq j\in I \}\] with respect to $\Pi$ and $\sigma$ with
\begin{align*} \forall\ i< j\in I: &&
\phi_{ij}^i=\begin{cases}
\tilde\eps_{12},&\text{if }\{i,j\}^\sigma\in E(\Pi), \\
\tilde\iota_1,&\text{if }\{i,j\}^\sigma\notin E(\Pi).
\end{cases} &&
\phi_{ij}^j=\begin{cases}
\tilde\eps_{23},&\text{if }\{i,j\}^\sigma\in E(\Pi), \\
\tilde\iota_2,&\text{if }\{i,j\}^\sigma\notin E(\Pi).
\end{cases}
\end{align*}
\end{definition}

\begin{notation} \label{rem:lift-so2-ama-to-spin2-ama}
Proposition~\ref{prop:lift-aut-so2} enables us to lift $\SO{2}$-amalgams to
$\Spin{2}$-amalgams: Let $\Pi$ be a simply laced diagram with labelling
$\sigma:I\to V$ and let $\AAA=\{ G_{ij}, \phi_{ij}^i \mid i\neq j\in I\}$ be
an $\SO{2}$-amalgam with respect to $\Pi$ and $\sigma$. Given $i>j\in I$,
there are $\gamma_{ij}^i,\gamma_{ij}^j\in \Aut(\SO{2})$ such that
\begin{align*}
\phi_{ij}^i=
  \begin{cases}
  \eps_{12}\circ \gamma_{ij}^j,&\text{if }\{i,j\}^\sigma\in E(\Pi), \\
  \iota_1\circ \gamma_{ij}^j,&\text{if }\{i,j\}^\sigma\notin E(\Pi),
  \end{cases}
&&
\phi_{ij}^j=
  \begin{cases}
  \eps_{23}\circ \gamma_{ij}^j,&\text{if }\{i,j\}^\sigma\in E(\Pi), \\
  \iota_2\circ \gamma_{ij}^j,&\text{if }\{i,j\}^\sigma\notin E(\Pi).
  \end{cases}
\end{align*}
We then lift $\gamma_{ij}^i,\gamma_{ij}^j$ as in Lemma~\ref{prop:lift-aut-so2}
to $\tilde\gamma_{ij}^i,\tilde\gamma_{ij}^j\in\Aut(\Spin{2})$ and set
\begin{align*}
\tilde\phi_{ij}^i:=
  \begin{cases}
  \tilde\eps_{12}\circ \tilde\gamma_{ij}^i,&\text{if }\{i,j\}^\sigma\in E(\Pi), \\
  \tilde\iota_1\circ \tilde\gamma_{ij}^i,&\text{if }\{i,j\}^\sigma\notin E(\Pi),
  \end{cases}
&&
\tilde\phi_{ij}^j:=
  \begin{cases}
  \tilde\eps_{23}\circ \tilde\gamma_{ij}^j,&\text{if }\{i,j\}^\sigma\in E(\Pi), \\
  \tilde\iota_2\circ \tilde\gamma_{ij}^j,&\text{if }\{i,j\}^\sigma\notin E(\Pi),
  \end{cases}
\end{align*}
and
\[\widetilde{G}_{ij}:=\begin{cases}
\Spin{3},&\text{if }\{i,j\}^\sigma\in E(\Pi), \\
\Spin{2}\times\Spin{2},&\text{if }\{i,j\}^\sigma\notin E(\Pi).
\end{cases}\]
\end{notation}

\begin{definition} \label{rhoij}
Let $\Pi$ be a simply laced diagram with labelling $\sigma:I\to V$ and let $\AAA=\{ G_{ij}, \phi_{ij}^i \mid i\neq j\in I\}$ be an $\SO{2}$-amalgam with respect to $\Pi$ and $\sigma$. Then
$\wAAA:=\{ \widetilde{G}_{ij}, \tilde\phi_{ij}^i \mid i\neq j\in I\}$
is the \Defn{induced $\Spin{2}$-amalgam with respect to $\Pi$ and $\sigma$}.
We also set
\[ \rho_{ij} := 
\begin{cases}
 \rho_3 ,&\text{if }\{i,j\}^\sigma\in E(\Pi), \\
 \rho_2\times\rho_2,&\text{if }\{i,j\}^\sigma\notin E(\Pi).
 \end{cases}
\]
\end{definition}

\begin{lemma}\label{lem:ama-embed-and-rho-commute-An}
Let $\Pi$ be a simply laced diagram with labelling $\sigma:I\to V$, let $\AAA=\{ G_{ij}, \phi_{ij}^i \mid i\neq j\in I\}$ be an $\SO{2}$-amalgam with respect to $\Pi$ and $\sigma$, and
let $\tilde\phi_{ij}^i$ be as introduced in Notation~\ref{rem:lift-so2-ama-to-spin2-ama}.
Then for all $i \neq j \in I$
\[
\phi_{ij}^i\circ \rho_2 = \rho_{ij} \circ \tilde\phi_{ij}^i.
\]
\end{lemma}

\begin{proof}
Without loss of generality suppose $i<j\in I$. 
Let $\gamma_{ij}^i,\gamma_{ij}^j\in \Aut\big(\SO{2})$,
$\tilde\gamma_{ij}^i,\tilde\gamma_{ij}^j\in \Aut\big(\Spin{2})$,
and $\tilde\phi_{ij}^i, \tilde\phi_{ij}^j$
be as in Notation~\ref{rem:lift-so2-ama-to-spin2-ama}.
Then, if $\{i,j\}^\sigma\in E(\Pi)$, we find that
\[
      \phi_{ij}^i \circ \rho_2
    = \eps_{12}\circ \gamma_{ij}^{i} \circ \rho_2 
    = \eps_{12} \circ \rho_2 \circ \tilde\gamma_{ij}^{i}
    \overset{\ref{6}}{=} \rho_3 \circ \eps_{12} \circ \tilde\gamma_{ij}^{i}
    = \rho_3 \circ \tilde\phi_{ij}^{i}\ .
\]
Similarly we also conclude $\phi_{ij}^j \circ \rho_2 = \rho_3\circ \tilde\phi_{ij}^{j}$.
Moreover, in case $\{i,j\}^\sigma\notin E(\Pi)$ we deduce
\[
      \phi_{ij}^i \circ \rho_2
    = \eps_{12}\circ \gamma_{ij}^{i} \circ \rho_2 
    = \eps_{12} \circ \rho_2 \circ \tilde\gamma_{ij}^{i}
    =  (\rho_2\times\rho_2)\circ\tilde\iota_1 \circ \tilde\gamma_{ij}^{i}
    =  (\rho_2\times\rho_2)\circ \tilde\phi_{ij}^{i}.
\]
Again we conclude by a similar argument that also
$\phi_{ij}^j \circ \rho_2 = (\rho_2\times\rho_2)\circ \tilde\phi_{ij}^{j}$.
\end{proof}

\begin{remark} \label{rem:spin-ama-to-so-ama-and-back}
Clearly the construction of an induced $\Spin{2}$-amalgam is symmetric and can also be applied backwards: Starting with a $\Spin{2}$-amalgam
$\hat\AAA$ one can construct an $\SO{2}$-amalgam $\AAA$, such that $\hat\AAA=\wAAA$. In particular, we obtain an epimorphism for the standard $\Spin{2}$- and $\SO{2}$-amalgams with respect to $\Pi$ and $\sigma$, which we denote by \[\pi_{\Pi,\sigma} = \{ \mathrm{id}_I, \rho_2, \rho_{ij} \} : \AAA\big(\Pi,\sigma,\Spin{2}\big) \to \AAA\big(\Pi,\sigma,\SO{2}\big).\]
\end{remark}

\begin{proposition}\label{prop:lift-ama-iso-sl}
Let $\Pi$ be a simply laced diagram with labelling $\sigma:I\to V$, let
$\AAA_1$ and $\AAA_2$ be $\SO{2}$-amalgams with respect to $\Pi$ and $\sigma$, and let $\alpha=\{ \pi, \alpha_{ij}\mid i\neq j\in I\}:\AAA_1\to \AAA_2$ be an isomorphism of amalgams.
Then there is a unique isomorphism $\tilde\alpha=\{ \pi, \tilde\alpha_{ij}\mid i\neq j\in I\}:\wAAA_1\to \wAAA_2$ such that for all $i \neq j \in I$
\[\rho_{\pi(i)\pi(j)}\circ \tilde\alpha_{ij}=\alpha_{ij}\circ \rho_{ij} .\]
\end{proposition}

\begin{proof}
Suppose $\AAA_1=\{ G_{ij}, \phi_{ij}^i \mid i\neq j\in I\}$
and $\AAA_2=\{ H_{ij}, \psi_{ij}^i \mid i\neq j\in I\}$.
Let $i\neq j\in I$.  Since both amalgams are defined with respect to $\Pi$
and $\sigma$, and since $G_{ij}\cong H_{\pi(i)\pi(j)}$, we conclude that
$\{i,j\}^\sigma$ is an edge if and only if
$\{i,j\}^{\pi\sigma}$ is an edge.  Hence up to relabeling and identifying $G_{ij}$ with its image under $\alpha_{ij}$, we
may assume $\pi=\id$ and $G_{ij}=H_{ij}$ and, thus,
$\alpha_{ij}\in\Aut(G_{ij})$ with $\alpha_{ij} \circ \phi_{ij}^i = \psi_{ij}^i$.
We distinguish two cases.

\begin{description}
\item[Case I: \boldmath$\{i,j\}^\sigma\in E(\Pi)$]
Then $\widetilde{G}_{ij}=\Spin{3}$.
Let $\tilde\alpha_{ij} \in \Aut(\Spin{3})$ be the unique automorphism from
Proposition~\ref{prop:lift-aut-soN} satisfying $\rho_3\circ
\tilde\alpha_{ij} = \alpha_{ij}\circ \rho_3$. It remains to verify that
this is compatible with the amalgam structure.  Indeed,
\begin{align*}
\rho_3 \circ \tilde\alpha_{ij} \circ \tilde\phi_{ij}^i \circ (\tilde\psi_{ij}^i)^{-1}
&= \alpha_{ij} \circ \rho_3 \circ \tilde\phi_{ij}^i \circ (\tilde\psi_{ij}^i)^{-1} 
 \overset{\ref{lem:ama-embed-and-rho-commute-An}}{=}
    \alpha_{ij} \circ \phi_{ij}^i \circ \rho_2 \circ (\psi_{ij}^i)^{-1} \\
&\overset{\ref{lem:ama-embed-and-rho-commute-An}}{=}
    \underbrace{\alpha_{ij} \circ \phi_{ij}^i \circ (\psi_{ij}^i)^{-1}}_{=\id_{\SO{3}}} \circ\; \rho_3
= \rho_3 = \rho_3 \circ \id_{\Spin{3}}. 
\end{align*}
Hence, by uniqueness in Proposition~\ref{prop:lift-aut-soN}, one has
$\tilde\alpha_{ij} \circ \tilde\phi_{ij}^i \circ (\tilde\psi_{ij}^i)^{-1} = \id_{\Spin{3}}$,
i.e., $\tilde\alpha_{ij} \circ \tilde\phi_{ij}^i = \tilde\psi_{ij}^i$.

\item[Case II: \boldmath$\{i,j\}^\sigma\notin E(\Pi)$]
Then $\widetilde{G}_{ij}=\Spin{2}\times\Spin{2}$.
Let $\tilde\alpha_{ij} \in \Aut(\Spin{2}\times\Spin{2})$ be the unique
automorphism from Corollary~\ref{prop:lift-aut-so2xso2} satisfying
$(\rho_2\times\rho_2)\circ \tilde\alpha_{ij} = \alpha_{ij}\circ
(\rho_2\times\rho_2)$.  It remains to verify that this is compatible with
the amalgam structure.  Indeed,
\begin{align*}
(\rho_2\times\rho_2) \circ \tilde\alpha_{ij} \circ \tilde\phi_{ij}^i \circ (\tilde\psi_{ij}^i)^{-1}
&= \alpha_{ij} \circ (\rho_2\times\rho_2) \circ \tilde\phi_{ij}^i \circ (\tilde\psi_{ij}^i)^{-1} 
 \overset{\ref{lem:ama-embed-and-rho-commute-An}}{=}
    \alpha_{ij} \circ \phi_{ij}^i \circ \rho_2 \circ (\psi_{ij}^i)^{-1} \\
&\overset{\ref{lem:ama-embed-and-rho-commute-An}}{=}
    \underbrace{\alpha_{ij} \circ \phi_{ij}^i \circ (\psi_{ij}^i)^{-1}}_{=\id_{\SO{2}\times\SO{2}}}
           \circ\; (\rho_2\times\rho_2)
= (\rho_2\times\rho_2) ..
\end{align*}
By uniqueness in Corollary~\ref{prop:lift-aut-so2xso2}, one concludes as in the previous case that $\tilde\alpha_{ij} \circ \tilde\phi_{ij}^i = \tilde\psi_{ij}^i$.
\qedhere
\end{description}
\end{proof}

\begin{corollary}
Let $\Pi$ be a simply laced diagram with labellings $\sigma_1,\sigma_2$. Then $\AAA\big(\Pi,\sigma_1,\Spin{2}\big)\cong \AAA\big(\Pi,\sigma_2,\Spin{2}\big)$.
\end{corollary}

\begin{proof}
Let $\AAA_1:=\AAA\big(\Pi,\sigma_1,\SO{2}\big)$ and $\AAA_2:=\AAA\big(\Pi,\sigma_2,\SO{2}\big)$.
The definitions imply
$\wAAA_1=\AAA\big(\Pi,\sigma_1,\Spin{2}\big)$ and $\wAAA_2=\AAA\big(\Pi,\sigma_2,\Spin{2}\big)$.
Moreover, one has $\AAA_1\cong \AAA_2$ by Consequence~\ref{14}. The claim now follows by applying Proposition~\ref{prop:lift-ama-iso-sl}.
\end{proof}

\begin{definition} \label{def:std-ama-Spin2}
As before, for a simply laced diagram $\Pi$ with labelling $\sigma$, we write $\AAA\big(\Pi,\Spin{2}\big)$ to denote the isomorphism type of $\AAA\big(\Pi,\sigma,\Spin{2}\big)$ and, by slight abuse of notation, any representative of this isomorphism type.
It is called the \Defn{standard $\Spin{2}$-amalgam with respect to $\Pi$}.
\end{definition}

\begin{theorem} \label{thm:uniqueness-spin-sl}
Let $\Pi$ be a simply laced diagram with labelling $\sigma:I\to V$ and let
$\wAAA$ be a continuous $\Spin{2}$-amalgam with respect to $\Pi$ and $\sigma$.
Then $\wAAA\cong \AAA\big(\Pi,\Spin{2}\big)$.
\end{theorem}

\begin{proof}
Let $\AAA$ be the continuous $\SO{2}$-amalgam that induces $\wAAA$, which exists by Remark~\ref{rem:spin-ama-to-so-ama-and-back}.
By Theorem~\ref{thm:uniqueness-so-sl}, one has
$\AAA\cong \AAA\big(\Pi,\SO{2}\big)$.
Proposition~\ref{prop:lift-ama-iso-sl} yields the claim, since $\AAA\big(\Pi,\SO{2}\big)$ induces $\AAA\big(\Pi,\Spin{2}\big)$. \qedhere
\end{proof}

\begin{theorem}\label{thm:univspin}
For $n\geq3$, the group $\Spin{n}$ is the universal enveloping group of $\AAA\big(A_{n-1},\Spin{2}\big)$.
\end{theorem}

\begin{proof} The proof runs along the same lines as the proof of Theorem~\ref{thm:univso}.
Let $I:=\{1,\ldots,n-1\}$. The group $\Spin{n}$ acts flag-transitively via the twisted adjoint representation (cf.\ Theorem~\ref{rho}(b)) on the simply connected projective geometry $\GGG := \PP_{n-1}(\RR)$ with fundamental maximal flag $\langle e_1\rangle_\RR \leq \langle e_1,e_2 \rangle_\RR \leq \cdots \leq \langle e_1,\ldots,e_{n-1}\rangle_\RR$.
By an iteration of Tits's Lemma (\cite[Corollary~1]{Tits:1986}, \cite[Corollary~1.4.6]{Ivanov/Shpectorov:2002}) with respect to the above maximal flag, the group $H:=\Spin{n}$ is the universal enveloping group of the amalgam
\[\AAA(\GGG, H)=\{ H_{ij}, \phi_{ij}^i \mid i\neq j\in I\} ,\]
where the $H_{ij}$ are the ``block-diagonal'' rank two subgroups  and $\Phi_{ij}^i:H_i\to H_{ij}$ is the inclusion map for each $i\neq j\in I$. By Consequence~\ref{7} and Remark~\ref{6}, one has
\[\forall\ i\in I:\qquad H_i=\tilde\eps_{\{i, i+1\}}\big(\Spin{2}\big)\ ,\]
and
\[\forall\ i<j\in I:\qquad H_{ij}=\begin{cases}
\hphantom{\big\langle}\tilde\eps_{\{i,i+1,i+2\}}\big(\Spin{3}\big),&\text{if }j=i+1, \\
\tilde\eps_{\{i, i+1\}}\big(\Spin{2}\big) \cdot  \tilde\eps_{\{j,j+1\}}\big(\Spin{2}\big),&\text{if }j\neq i+1.
\end{cases}\]
As a consequence, the system
\[\alpha=\{ \id_I,\alpha_{ij},\alpha_i \mid i\neq j\in I\}: \AAA\big(A_{n-1},\Spin{2}\big)\to \AAA(\GGG,H)\]
with
\[\forall\ i\in I:\qquad \alpha_i=\tilde\eps_{\{i,i+1\}}:\Spin{2}\to H_i\]
and
\[
\forall\ i<j \in I:\qquad \alpha_{ij}=\begin{cases}
\tilde\eps_{\{i,i+1,i+2\}},&\text{if }j=i+1, \\
\eps_{\{i,i+1\}}\cdot\tilde\eps_{\{j,j+1\}},&\text{if }j\neq i+1,
\end{cases}
\]
is an epimorphism of $\Spin{2}$-amalgams. In fact, each $\alpha_i$ and each $\alpha_{ii+1}$ is an isomorphism, only the $\alpha_{ij} : \Spin{2} \times \Spin{2} \to \Spin{2}.\Spin{2}$, $j \neq i+1$, have a kernel of order two: the $-1$ in the left-hand factor gets identified with the $-1$ in the right-hand factor. We may conclude that $\Spin{n}$ is an enveloping group of $\AAA\big(A_{n-1},\Spin{2}\big)$. It remains to prove universality.

Let $(G,\tau_{ij})$ be an arbitrary enveloping group of $\AAA\big(A_{n-1},\Spin{2}\big)$ and let $1 \leq i < j \leq n-1$ with $j \neq i+1$. 
By definition the following diagram commutes for $1 \leq a \neq b \neq c \leq n-1$:
\[
\xymatrix{
& G_{ab} \ar[dr]^{\tau_{ab}} \\
  \Spin{2} \ar[dr]_{\phi_{bc}^{b}} \ar[ur]^{\phi_{ab}^{b}}  && G \\
&   G_{bc} \ar[ur]_{\tau_{bc}} 
}
\]
In particular, one has:
\begin{align*}
(\tau_{i,i+2} \circ \phi_{i,i+2}^{i})(-1) & = (\tau_{i,i+1} \circ \phi_{i,i+1}^{i})(-1)  & \text{(set $b=i$, $a=i+1$, $c=i+2$)}\\
&= (\tau_{i,i+1} \circ \phi_{i,i+1}^{i+1})(-1) & \text{(since $\tilde\eps_{12}(-1)=\tilde\eps_{23}(-1)$)} \\
&= (\tau_{i+1,i+2} \circ \phi_{i+1,i+2}^{i+1})(-1) & \text{(set $b=i+1$, $a=i$, $c=i+2$)}\\
&= (\tau_{i+1,i+2} \circ \phi_{i+1,i+2}^{i+2})(-1) & \text{(since $\tilde\eps_{12}(-1)=\tilde\eps_{23}(-1)$)} \\
&= (\tau_{i,i+2} \circ \phi_{i,i+2}^{i+2})(-1) & \text{(set $b=i+2$, $a=i$, $c=i+1$)}.
\end{align*}

We conclude by induction that $\tau_{ij} : \Spin{2} \times \Spin{2} \to G$, $j \neq i+1$, always factors through $\Spin{2}.\Spin{2}$ or, in other words, $\tau : \AAA\big(A_{n-1},\Spin{2}\big) \to G$ always factors through $\AAA(\GGG, H)$. That is, the universal enveloping group $\Spin{n}$ of $\AAA(\GGG, H)$ is also a universal enveloping group of $\AAA\big(A_{n-1},\Spin{2}\big)$. 
\end{proof}

\begin{remark}
The proof of Theorem~\ref{thm:univspin} would become a bit easier if one replaced $\Spin{2} \times \Spin{2}$ by $\Spin{2}.\Spin{2}$ in Definition~\ref{18}. The setup we chose, on the other hand, makes it easier to deal with reducible diagrams. Of course, one could a priori try to just restrict oneself to the case of irreducible diagrams, which in the case of simply laced diagrams is unproblematic. However, when dealing with arbitrary diagams it will turn out that it is more natural to also allow reducible diagrams. 
\end{remark}

\section{Spin covers of simply laced type} \label{sec:spin-cover-simply-laced}

\begin{notation}
Let $\Pi=(V,E)$ be a (finite) simply laced diagram with labelling $\sigma:I\to V$ and let $c(\Pi)$ denote the {number of connected components} of $\Pi$.
A \Defn{component labelling} of $\Pi$ is a map $\KKK:V\to\{1,\ldots,c(\Pi)\}$
such that $u,v\in V$ are in the same connected component of $\Pi$
if and only if $\KKK(u)=\KKK(v)$.

Throughout this section, let $\Pi$ be a (finite) simply laced diagram, $\sigma : I \to V$ and labelling and $\KKK$ a component labelling.
\end{notation}

Generalizing Theorem~\ref{thm:univso}, the universal enveloping group of a continuous $\SO{2}$-amalgam over an arbitrary simply laced diagram $\Pi$ is isomorphic to the maximal compact subgroup of the corresponding split real Kac--Moody group (cf.\ Theorem~\ref{thm:K-univ-sl}). The goal of this section is to construct and investigate its spin cover, which will arise as the universal enveloping group of the continuous $\Spin{2}$-amalgam over the same simply laced diagram $\Pi$. In the case of $E_{10}$ its existence has been conjectured by Damour and Hillmann in  \cite[Section~3.5, p.\ 24]{DamourHillmann}.

Additional key ingredients, next to transitive actions on buildings and the theory of $\SO{2}$- and $\Spin{2}$-amalgams developed so far, will be the generalized spin representations constructed in \cite{Hainke/Koehl/Levy} and the Iwasawa decomposition of split real Kac--Moody groups studied, for example, in \cite{Medts/Gramlich/Horn}.
For definitions and details on Kac--Moody theory we refer the reader to \cite{kac1994infinite}, \cite{Remy:2002}, \cite[Section~2]{Hainke/Koehl/Levy}, \cite[Section~7]{Hartnick/Koehl/Mars}, \cite[Chapter~5]{Marquis:2013}.

\begin{theorem}\label{thm:K-univ-sl}
Let $\Pi$ be a simply laced diagram, let $G(\Pi)$ be the corresponding simply connected split real Kac--Moody group, and let $K(\Pi)$ be its maximal compact subgroup, i.e., the subgroup fixed by the Cartan--Chevalley involution. Then there exists a faithful universal enveloping morphism $\tau_{K(\Pi)} : \AAA\big(\Pi,\SO{2}\big) \to K(\Pi)$.
\end{theorem}

\begin{proof}
For $i \in I$ denote by $G_i$ the fundamental rank one subgroups of $G(\Pi)$ and, for $i \neq j \in I$, by $G_{ij}$ the fundamental rank two subgroups. The groups $G_i$ are isomorphic to $\SL{2}$ and the groups $G_{ij}$ are isomorphic to $\SL{3}$ or to $\SL{2} \times \SL{2}$, depending on whether the vertices $i^\sigma$, $j^\sigma$ of $\Pi$ are joined by an edge or not. The Cartan--Chevalley involution $\omega$ leaves the groups $G_i$, $G_{ij}$ setwise invariant and, in fact, induces the transpose-inverse map on these groups. Define $H_i := \Fix_{G_i}(\omega) \cong \SO{2}$ and $H_{ij} := \Fix_{G_{ij}}(\omega)$, the latter being isomorphic to $\SO{3}$ or to $\SO{2} \times \SO{2}$, and let $\psi_{ij}^i : H_i \to H_{ij}$ denote the (continuous) inclusion map for $i\neq j\in I$. 
By
\cite[Theorem~1.2]{Medts/Gramlich/Horn}, the group $K(\Pi)$ is the
universal enveloping group of the amalgam $\AAA_1:=\{H_{ij},
\psi_{ij}^i, \mid i\neq j\in I\}$.

Given $i<j \in I$ such that $\{i,j\}^\sigma\in E(\Pi)$,
by Theorem~\ref{thm:uniqueness-so-sl} applied to the subdiagram of $\Pi$ of type $A_2$ consisting of the vertices $i^\sigma$, $j^\sigma$ there is a continuous isomorphism $\alpha_{ij}:H_{ij}\to
\SO{3}$ such that
\begin{align*}
(\alpha_{ij}\circ \psi_{ij}^i)\big(H_i)=\eps_{12}\big(\SO{2}\big)\ , &&
(\alpha_{ij}\circ \psi_{ij}^j)\big(H_j)=\eps_{23}\big(\SO{2}\big)\ ..
\end{align*}
Let $i\in I$ and choose $j\in I$ such that $\{i,j\}^\sigma\in E(\Pi)$. Define
\[\alpha_i:=\begin{cases}
\eps_{12}^{-1}\circ \alpha_{ij}\circ \psi_{ij}^i:H_i\to \SO{2},&\text{if }i>j, \\
\eps_{23}^{-1}\circ \alpha_{ij}\circ \psi_{ij}^i:H_i\to \SO{2},&\text{if }i<j.
\end{cases}\]
For $i<j \in I$ such that $\{i,j\}^\sigma\not\in E(\Pi)$, define $\alpha_{ij} := \alpha_i \times \alpha_j : H_{ij} = H_i \times H_j \to \SO{2} \times \SO{2}$. For arbitrary $i \neq j \in I$, let
\[
K_{ij}:=\begin{cases}
\SO{3},&\text{if }\{i,j\}^\sigma\in E(\Pi), \\
\SO{2}\times \SO{2},& \text{if }\{i,j\}^\sigma\notin E(\Pi),
\end{cases}
\]
and
\[\bar\phi_{ij}^i:=\alpha_{ij}\circ\psi_{ij}^i\circ \alpha_{i}^{-1}:\SO{2}\to K_{ij}.\]
Then
$\AAA_2:=\{ K_{ij}, \bar\phi_{ij}^i, \mid i\neq j\in I \}$
is an $\SO{2}$-amalgam with respect to $\Pi$ and $\sigma$. Moreover, the system
$\alpha=\{ \id_I,\alpha_{ij},\alpha_i \mid i\neq j\in I \}:\AAA_1\to \AAA_2$
is an isomorphism of amalgams. Indeed, given $i\neq j\in I$, one has
$\bar\phi_{ij}^i\circ \alpha_i=\alpha_{ij}\circ \psi_{ij}^i\circ \alpha_{i}^{-1}\circ \alpha_i=\alpha_{ij}\circ \psi_{ij}^i$.
Finally, $\alpha_i$ is continuous for each $i\in I$, whence $\bar\phi_{ij}^i$ is continuous for all $i\neq j\in I$. Therefore, $\AAA_2$ is a continuous $\SO{2}$-amalgam with respect to $\Pi$ and $\sigma$ so that
$\AAA_1\cong \AAA_2\cong \AAA\big(\Pi,\SO{2}\big)$
by Theorem~\ref{thm:uniqueness-so-sl}.
\end{proof}

\begin{notation} \label{nota:sl-ama-and-epi}
For consistency, we fix the groups and connecting morphisms in the
standard $\SO{2}$-amalgam with respect to $\Pi$ as follows (cf.\ Definition~\ref{def:std-ama-SO2}):
\[ \AAA\big(\Pi,\SO{2}\big)=\{ K_{ij}, \phi_{ij}^i,\mid i\neq j\in I\}\ . \]
Similarly for the standard $\Spin{2}$-amalgam with respect to $\Pi$ (cf.\ Definition~\ref{def:std-ama-Spin2}):
\[ \AAA\big(\Pi,\Spin{2}\big)=\{ \widetilde{K}_{ij}, \tilde\phi_{ij}^i \mid i\neq j\in I\}\ . \]
We denote the epimorphism of amalgams from Remark~\ref{rem:spin-ama-to-so-ama-and-back} by 
\[ \pi_{\Pi} : \AAA\big(\Pi,\Spin{2}\big) \to \AAA\big(\Pi,\SO{2}\big)\ .\]
\end{notation}

\begin{remark}\label{addtopology}
As discussed in Remark~\ref{rem:continuous} the amalgam $\AAA\big(\Pi,\SO{2}\big)$ consists of compact Lie groups with continuous connecting homomorphisms. On the other hand, the group $K(\Pi)$ naturally carries a Hausdorff group topology that is $k_\omega$: Indeed, $K(\Pi)$ is the subgroup of the unitary form studied in \cite[Section~6]{Gloeckner/Gramlich/Hartnick:2010} fixed by complex conjugation and it is the subgroup of the real Kac--Moody group $G(\Pi)$ studied in \cite[Section~7]{Hartnick/Koehl/Mars} fixed by the Cartan--Chevalley involution; both ambient groups are $k_\omega$ (by \cite[Theorem~6.12]{Gloeckner/Gramlich/Hartnick:2010}, resp.\ \cite[Theorem~7.22]{Hartnick/Koehl/Mars}) and, hence, so is any subgroup fixed by a continuous involution (cf.\ \cite[Proposition~4.2(b)]{Gloeckner/Gramlich/Hartnick:2010}). Note that the $k_\omega$-group topologies on $K(\Pi)$ induced from the real Kac--Moody group $G(\Pi)$ and from the unitary form coincide, as both are induced from the $k_\omega$-group topology on the ambient complex Kac--Moody group (cf.\ \cite[Theorem~6.3]{Gloeckner/Gramlich/Hartnick:2010}, resp.\ \cite[Theorem~7.22]{Hartnick/Koehl/Mars}). 

Furthermore, a straightforward adaption of the proof of \cite[Proposition~6.9]{Gloeckner/Gramlich/Hartnick:2010} implies that this $k_\omega$-group topology is the finest group topology with respect to the enveloping homomorphisms $\tau_{ij} : K_{ij} \to K(\Pi)$. In other words, the obvious analog of \cite[Theorem~6.12]{Gloeckner/Gramlich/Hartnick:2010}, \cite[Theorem~7.22]{Hartnick/Koehl/Mars} holds for $(K(\Pi),\tau_{\Pi})$. In particular, to any enveloping morphism $\phi = (\phi_{ij}) : \AAA\big(\Pi,\SO{2}\big) \to H$ into a Hausdorff topological group $H$ with continuous homomorphisms $\phi_{ij} : K_{ij} \to H$ there exists a unique {\em continuous} homomorphism $\phi : K(\Pi) \to H$ such that the following diagram commutes:
\[\xymatrix{
  \AAA\big(\Pi,\SO{2}\big) \ar[drr]_\phi \ar[rr]^{\tau_{\Pi}} && K(\Pi) \ar[d]^\psi \\ && H
}\]
\end{remark}

Theorems \ref{thm:univso} and \ref{thm:univspin} state that the double cover $\Spin{n}$ of $\SO{n}$ is the universal enveloping group of the two-fold central extension $\AAA\big(A_{n-1},\Spin{2}\big)$ of the amalgam $\AAA\big(A_{n-1},\SO{2}\big)$ as defined in Proposition~\ref{prop:env-grp-central-cover}.
In view of Theorem~\ref{thm:K-univ-sl} it is therefore natural to introduce the following notion.

\begin{definition} \label{defn:sl-spin-group}
The \Defn{spin group ${\Spin\Pi}$ with respect to $\Pi$} is the
canonical universal enveloping group of the (continuous) amalgam
$\AAA\big(\Pi,\Spin{2}\big)$ with the canonical universal enveloping morphism
\[\tau_{\Spin\Pi} = \{ \tau_{ij} \mid i \neq j \in I \} : \AAA\big(\Pi,\Spin{2}\big) \to {\Spin\Pi}.\]
\end{definition}

\begin{lemma}\label{lem:sl-K-envelops-spin-amalgam}
$K(\Pi)$ is an enveloping group of the amalgam $\AAA\big(\Pi,\Spin{2}\big)$. There exists a unique central extension $\rho_\Pi : \Spin\Pi \to K(\Pi)$ that makes the following diagram commute (cf.\ Notation~\ref{nota:amalgamcomm}):
\[
\xymatrix{
  \AAA\big(\Pi,\Spin{2}\big) \ar[rr]^{\tau_{\Spin\Pi}} \ar[d]_{\pi_\Pi} &&
  \Spin{\Pi} \ar@{-->}[d]^{\rho_{\Pi}} \\
\AAA\big(\Pi,\SO{2}\big) \ar[rr]^{\tau_{K(\Pi)}} && K(\Pi)
},
\]
where \[\tau_{K(\Pi)} = \{ \psi_{ij} : K_{ij} \to
K(\Pi) \mid i \neq j \in I \} : \AAA\big(\Pi,\SO{2}\big) \to K(\Pi)\]
is the universal enveloping morphism whose existence is guaranteed by Theorem~\ref{thm:K-univ-sl}.
\end{lemma}

\begin{proof}
As in Definition~\ref{rhoij}, for $i \neq j \in I$ let $\rho_{ij} : \widetilde{K}_{ij} \to K_{ij}$ be the epimorphism
$\rho_3$ if $\{i,j\}^\sigma$ is an edge, and $\rho_2\times\rho_2$ otherwise.
Then, by Lemma~\ref{lem:ama-envelope-lift}, the group $K(\Pi)$ with the homomorphisms $\xi_{ij}:=\psi_{ij}\circ
\rho_{ij}:\widetilde{K}_{ij}\to K(\Pi)$ for all $i \neq j \in I$ is an enveloping
group of $\AAA\big(\Pi,\sigma,\Spin{2}\big)$. 
By universality of $\tau_{\Spin{\Pi}} : \AAA\big(\Pi,\Spin{2}\big) \to {\Spin\Pi}$, Lemma~\ref{lem:ama-epi-induces-envelope-epi} provides a unique epimorphism $\rho_\Pi : \Spin\Pi \to K(\Pi)$ with the properties as claimed. This epimorphism is a central extension by Proposition~\ref{prop:env-grp-central-cover}.
\end{proof}

The following is a straightforward generalization of the observation we made towards the end of the proof of Theorem~\ref{thm:univspin}.

\begin{lemma} \label{lem:unique-minus-one-sl}
Let $i\neq j\in I$ and $k \neq \ell\in I$. If $i^\sigma$ and
$k^\sigma$ are in the same connected component of $\Pi$, then
\[\tau_{ij}( \tilde\phi_{ij}^i(-1_{\Spin{2}}))
 =\tau_{k\ell}( \tilde\phi_{k\ell}^k(-1_{\Spin{2}}))\ .\]
\end{lemma}

\begin{proof}
As $i^\sigma$ and $k^\sigma$ are in the same connected component,
there exists a sequence $i_0:=i,i_1,\ldots,i_n:=k\in I$ such that
$\{i_r^\sigma,i_{r+1}^\sigma\}$ are edges for $0\leq r<n$.
Thus $\widetilde{K}_{i_r i_{r+1}}=\Spin{3}$ and by the definition of $\Spin\Pi$
as the canonical universal enveloping group of the amalgam $\AAA\big(\Pi,\Spin{2}\big)$, we have
\[
 \tilde\phi_{i_r i_{r+1}}^{i_r}(-1_{\Spin{2}}))
= -1_{\Spin{3}}
= \tilde\phi_{i_r i_{r+1}}^{i_{r+1}}(-1_{\Spin{2}}))\ .
\]
Hence
\begin{align*}
  \tau_{ij}( \tilde\phi_{ij}^i(-1_{\Spin{2}}))
&=\tau_{i_0i_1}( \tilde\phi_{i_0i_1}^{i_0}(-1_{\Spin{2}}))
 =\tau_{i_0i_1}( \tilde\phi_{i_0i_1}^{i_1}(-1_{\Spin{2}}))
 = \cdots \\
&=\tau_{i_{n-1}i_n}( \tilde\phi_{i_{n-1}i_n}^{i_n}(-1_{\Spin{2}}))
 =\tau_{i_{n-1}k}( \tilde\phi_{i_{n-1}k}^{k}(-1_{\Spin{2}})) \\
&=\tau_{k\ell}( \tilde\phi_{k\ell}^k(-1_{\Spin{2}}))\ .
\end{align*}
where the first and last equality hold due to the definition
of enveloping homomorphisms. 
\end{proof}

Thus the following is well-defined.

\begin{definition}\label{minus1}
For $i\neq j\in I$ define
\[-1_{\Spin\Pi,\KKK(i)}:=\tau_{ij}(\tilde\phi_{ij}^i(-1_{\Spin{2}}))
  \quad\text{ and }\quad
  Z:=\gen{ -1_{\Spin\Pi,1}, \ldots, -1_{\Spin\Pi,c(\Pi)}} \leq \Spin\Pi\ .\]
\end{definition}

\begin{observation} \label{finitecentralextension}
The following are true:
\begin{enumerate}
\item $Z$ is contained in the centre of $\Spin\Pi$.
\item $|Z|\leq 2^{c(\Pi)}$.
\end{enumerate}
\end{observation}
The first assertion is immediate from Proposition~\ref{prop:env-grp-central-cover} applied to the $\Spin{2}$-amalgam $\AAA\big(\Pi,\Spin{2}\big)$ and the $\SO{2}$-amalgam $\AAA\big(\Pi,\SO{2}\big)$ with $\wt U = \Spin{2}$, $\wt V = \langle -1 \rangle$ and $U = \SO{2}$. The second follows from the fact that $Z$ is abelian by assertion (a) and admits a generating system of $c(\Pi)$ involutions by definition.

\medskip

The remainder of this section is mostly devoted to proving the following result:

\begin{theorem} \label{Zmaximal}
One has $|Z|=2^{c(\Pi)}$.
\end{theorem}

We start the proof of this theorem by revisiting Remark~\ref{rem:cl3=quaternions}.

\begin{lemma}\label{24}
Let $V$ be an $\RR$-vector space and let $X_i,X_j\in \End(V)$ be such that
\begin{align*}
X_i^2=-\id_V=X_j^2\ , && X_iX_j=-X_jX_i\ .
\end{align*}
Then the map
\begin{align*}
\psi: \Spin{3} & \to \mathrm{GL}(V) \\
    a + b e_1e_2 + c e_2e_3 + de_1e_3
   & \mapsto a \id_V +  b X_i + c X_j + d X_i X_j
\end{align*}
is a group monomorphism such that
\[
 \psi(\tilde\eps_{12}(S(\alpha))) = \cos(\alpha)\id_V + \sin(\alpha) X_i\ ,\quad
 \psi(\tilde\eps_{23}(S(\alpha))) = \cos(\alpha)\id_V + \sin(\alpha) X_j\ .
\]
\end{lemma}

\begin{proof}
The subspace $\HH:=\gen{\id_V, X_i,X_j,X_iX_j}_\RR$ is an $\RR$-subalgebra of $\End(V)$, the set $\{\id_V, X_i,X_j,X_iX_j\}$ is an $\RR$-basis of $\HH$, and the $\RR$-linear extension $\hat\psi: \Cl{3}^0\to \HH$ of
\begin{align*} 1\mapsto \id_V\ , &&  e_1e_2\mapsto X_i\ , &&  e_1e_3\mapsto X_iX_j\ , && e_2e_3\mapsto X_j\end{align*}
is an isomorphism of algebras:
Indeed, since $\id_V$, $X_i$, $X_j$ and $X_iX_j$ satisfy the same relations as $1$, $e_1e_2$, $e_2e_3$ and $e_1e_3$, the map $\hat\psi$ is a homomorphism of rings. Since $X_i\neq 0_{\End(V)}$, one has $\ker(\hat\psi)\neq \Cl{3}^0$. By Remark~\ref{rem:cl3=quaternions}, $\Cl{3}^0$ is a skew field and, thus, simple as a ring. Therefore, $\hat\psi$ is injective and, hence, bijective, because $\dim_\RR \HH\leq 4$, i.e., $\hat\psi$ is an isomorphism of algebras.

Consequently, the restriction $\psi$ of $\hat\psi$ to $\Spin{3}$ is injective with values in $\mathrm{GL}(V)$, i.e., $\psi : \Spin{3} \to \mathrm{GL}(V)$ is a group monomorphism.
The final statement is immediate from the definitions.
\end{proof}

\begin{lemma}\label{25}
Let $V$ be an $\RR$-vector space and let $X_i,X_j\in \End(V)$ be such that
\begin{align*}
X_j\notin \gen{\id_V, X_i}_\RR\ , && X_i^2=-\id_V=X_j^2\ , && X_iX_j=X_jX_i\ .
\end{align*}
Then
 the map
\begin{align*}
\psi:  \gen{ \tilde\eps_{12}\big(\Spin{2}\big), \tilde\eps_{34}\big(\Spin{2}\big) } \subseteq \Spin{4}
   &\to \mathrm{GL}(V) \\
    a + b e_1e_2 + c e_3e_4 + de_1e_2e_3e_4
   &\mapsto a \id_V +  b X_i + c X_j + d X_i X_j
\end{align*}
is a group monomorphism such that
\[
 \psi(\tilde\eps_{12}(S(\alpha))) = \cos(\alpha)\id_V + \sin(\alpha) X_i\ ,\quad
 \psi(\tilde\eps_{34}(S(\alpha))) = \cos(\alpha)\id_V + \sin(\alpha) X_j\ .
\]
\end{lemma}

\begin{proof}
The subspace $\AA:=\gen{\id_V, X_i,X_j,X_iX_j}_\RR$ is an $\RR$-subalgebra of $\End(V)$, the set $\{\id_V, X_i,X_j,X_iX_j\}$ is an $\RR$-basis of $\AA$, and the $\RR$-linear extension $\hat\psi: \tilde{\AA}:=\gen{1,e_1e_2,e_3e_4,e_1e_2e_3e_4}_\RR\subseteq \Cl{4}^0 \to \AA$ of
\begin{align*} 1\mapsto \id_V\ , &&  e_1e_2\mapsto X_i\ , && e_3e_4\mapsto X_j, &&  e_1e_2e_3e_4\mapsto X_iX_j\ \end{align*}
is an isomorphism of algebras: Indeed, since $\id_V$, $X_i$, $X_j$ and $X_iX_j$ satisfy the same relations as $1$, $e_1e_2$, $e_3e_4$ and $e_1e_2e_3e_4$, the map $\psi$ is a homomorphism of rings. The hypothesis $X_j\notin \gen{\id_V, X_i}_\RR$ implies that \[\{1,X_i,X_j,X_iX_j\}\] is $\RR$-linearly independent. Therefore, $\hat\psi$ is injective and, thus, bijective, because $\dim_\RR \AA\leq 4$.

Consequently, the restriction $\psi$ of $\hat\psi$ to $\langle \tilde\eps_{12}\big(\Spin{2}\big), \tilde\eps_{34}\big(\Spin{2}\big)\rangle\subseteq \Spin{4}$ is injective with values in $\mathrm{GL}(V)$, i.e., $\psi : \langle \tilde\eps_{12}\big(\Spin{2}\big), \tilde\eps_{34}\big(\Spin{2}\big)\rangle\subseteq \Spin{4} \to \mathrm{GL}(V)$ is a group monomorphism as claimed.
The final statement is immediate from the definitions.
\end{proof}

\begin{remark}
We are now in a position to use the results of \cite{Hainke/Koehl/Levy} in order to confirm the conjecture concerning $\Spin{\Pi}$ made in footnote 18 on page 24 of \cite{DamourHillmann}.
The definition of a generalized spin representation can be found in \cite[Definition~4.4]{Hainke/Koehl/Levy}, the definition and existence of a maximal one in \cite[Corollary~4.8]{Hainke/Koehl/Levy}. 

We point out that \cite[Example~4.1]{Hainke/Koehl/Levy} uses a convention for Clifford algebras different from the one used in the present article; however, \cite[Corollary~4.8]{Hainke/Koehl/Levy} is formulated and proved without making any reference to Clifford algebras whatsoever. 
\end{remark}

\begin{theorem} \label{m1}
Let
\begin{itemize}
\item $\Pi$ be an irreducible simply laced diagram with labelling $\sigma:I\to V$,
\item $\mathfrak{g}$ be the Kac--Moody algebra corresponding to $\Pi$ and $\mathfrak{k}$ its maximal compact subalgebra with Berman generators $Y_1$, \ldots, $Y_n$ (cf.\ \cite[Section~2.2]{Hainke/Koehl/Levy}),
\item $\mu: \mathfrak{k} \to \End(\CC^s)$, $s\in\NN$, be a maximal generalized spin representation (cf.\ \cite[Corollary~4.8]{Hainke/Koehl/Levy}), and
\item $X_i:=2\mu(Y_i)$ for each $i\in I$.
\end{itemize}

Then, for each $i \neq j \in I$, there exist subgroups $H_{ij} \leq \mathrm{GL_{s}}(\CC)$
and an enveloping morphism
\[\Psi_\AAA = \{ \psi_{ij} \mid i \neq j \in I \} : \AAA\big(\Pi,\Spin{2}\big) \to H:=\gen{ H_{ij} \mid i\neq j\in I }\]
with injective $\psi_{ij}$ whenever $\{i,j\}^\sigma \in E(\Pi)$. 
\end{theorem}

\begin{proof}
According to \cite[Remark~4.5]{Hainke/Koehl/Levy}, given $i\neq j\in I$,
one has $X_i^2=-\id_V = X_j^2$, and
\[X_iX_j=\begin{cases}
-X_jX_i,&\text{if }\{i,j\}^\sigma\in E(\Pi)\ , \\
X_jX_i,&\text{if }\{i,j\}^\sigma\notin E(\Pi)\ .
\end{cases}\]
Moreover, $X_j\notin \gen{\id_V, X_i}_\RR$, as $\mu$ is maximal.
Thus Lemma~\ref{24} provides group monomorphisms $\psi_{ij} : \wt{K}_{ij} \to \mathrm{GL}_s(\CC)$, if $\{i,j\}^\sigma \in E(\Pi)$, and Lemma~\ref{25} provides group homomorphisms $\psi_{ij} : \wt{K}_{ij} \to \mathrm{GL}_s(\CC)$ with kernel $\langle -1, -1 \rangle$, if $\{i,j\}^\sigma \not\in E(\Pi)$. This allows one to define
\[  H_{ij} := \mathrm{im}(\psi_{ij}). \]
Restriction of the ranges of the maps $\psi_{ij}$ to $H_{ij}$ thus provides 
\begin{align*}
\text{group isomorphisms }\quad&
 \psi_{ij}:\wt{K}_{ij} = \Spin{3}\to H_{ij},&
    \text{if }\{i,j\}^\sigma\in E(\Pi)\ , \\
\text{group epimorphisms }\quad&
 \psi_{ij}:\wt{K}_{ij} = \Spin{2}\times\Spin{2}\to H_{ij},&
    \text{if }\{i,j\}^\sigma\notin E(\Pi)\ ,
\end{align*}
satisfying
\[ \forall i \neq j \in I: \qquad \psi_{ij}\big(\tilde\phi_{ij}^j(\cos(\alpha)+\sin(\alpha)e_1e_2)\big)=\cos(\alpha)\id_V+\sin(\alpha)X_j\ .\]
In particular, one has
\[\forall\ i\neq j\neq k:\qquad \psi_{ij}\circ \tilde\phi_{ij}^j=\psi_{kj}\circ \tilde\phi_{kj}^j\ .
\]
The set $\Psi_\AAA := \{ \psi_{ij} \mid i \neq j \in I \}$ is the desired enveloping morphism.
\end{proof}

\begin{remark}\label{inparticular}
Let everything be as in Theorem~\ref{m1}.
By universality of $\tau_{\Spin\Pi} : \AAA\big(\Pi,\Spin{2}\big) \to \Spin\Pi$ (cf.\ Definition~\ref{defn:sl-spin-group}) there exists an epimorphism \[\Xi:\Spin\Pi\to H\]
such that the following diagram commutes:
\[
\xymatrix{
\Spin{\Pi} \ar[rr]^\Xi && H \\
\Spin{2} \ar[u]^{\tau_{ij} \circ \tilde\phi_{ij}^i} \ar[urr]_{\psi_{ij} \circ \tilde\phi_{ij}^i} 
}
\]
The commutative diagram in Lemma~\ref{lem:sl-K-envelops-spin-amalgam} and the finiteness of the central extension $\Spin{\Pi} \to K(\Pi)$ by Observation~\ref{finitecentralextension} in fact allow one to lift the topological universality statement from Remark~\ref{addtopology} concerning $\tau_\Pi : \AAA\big(\Pi,\SO{2}\big) \to K(\Pi)$ to a topological universality statement concerning $\tau_{\Spin\Pi} : \AAA\big(\Pi,\Spin{2}\big) \to \Spin\Pi$. Moreover, the maps $\psi$ constructed in Lemmas~\ref{24} and \ref{25} are certainly continuous with respect to the Lie group topologies, if $\dim_\RR(V) < \infty$.

In particular, the enveloping morphism $\Psi_\AAA = \{ \psi_{ij} \}$ from the theorem consists of continuous maps, so that by universality $\Xi : \Spin\Pi\to H$ is continuous as well.
\end{remark}

As an immediate consequence we record: 

\begin{corollary} \label{m1cor}
Let $\Pi$ be an irreducible simply laced diagram.
Then $1_{\Spin\Pi}\neq -1_{\Spin\Pi}$.
\end{corollary}

\begin{proof}
We conclude from Remark~\ref{inparticular} 
\[\Xi(-1_{\Spin\Pi})=(\Xi \circ \tau_{12} \circ \tilde\phi_{12}^1)(-1_{\Spin{2}})=(\xi_{12}\circ\tilde\phi_{12}^1)(-1_{\Spin{2}})=\cos(\pi)\id_V+\sin(\pi)X_1=-\id_V, \]
and, hence, $1_{\Spin\Pi}\neq -1_{\Spin\Pi}$. 
\end{proof}

\begin{theorem} \label{m2}
Let $\Pi$ be a simply laced diagram.
Then the universal enveloping group \[\left(\Spin\Pi,\tau_{\Spin\Pi} =  \{ \tau_{ij} \mid i \neq j \in I \}\right)\]
of $\AAA\big(\Pi,\Spin{2}\big)$ is a $2^{c(\Pi)}$-fold central extension of
the universal enveloping group $K(\Pi)$ of $\AAA\big(\Pi,\SO{2}\big)$.
\end{theorem}

\begin{proof}
Let $\Pi_1,\ldots,\Pi_{c(\Pi)}$ be the connected components of $\Pi$. Then \[\Spin\Pi = \Spin{\Pi_1} \times \cdots \times \Spin{\Pi_{c(\Pi)}}.\]
Indeed,
\begin{align*}
\tau_{\Spin\Pi} : \AAA\big(\Pi,\Spin{2}\big) & \to \Spin{\Pi_1} \times \cdots \times \Spin{\Pi_{c(\Pi)}} \\
\tau_{ij} : \wt K_{ij} & \to \tau_{ij}(\wt K_{ij}) & \text{if $\KKK(i) = \KKK(j)$} \\
\tau_{ij} : \wt K_{ij} & \to (\tau_{ij} \circ \tilde\phi_{ij}^i)(\Spin{2}) \times (\tau_{ij} \circ \tilde\phi_{ij}^j)(\Spin{2}) : (x,y) \mapsto \tau_{ij}(x,y)  & \text{if $\KKK(i) \neq \KKK(j)$}
\end{align*}
is an enveloping morphism.

It therefore suffices to prove the theorem for irreducible simply laced diagrams $\Pi$. In this case, however, it is immediate from Proposition~\ref{prop:env-grp-central-cover} applied to the $\wt U$-amalgam $\AAA\big(\Pi,\Spin{2}\big)$ and the $U$-amalgam $\AAA\big(\Pi,\SO{2}\big)$ with $\wt U = \Spin{2}$, $U=\SO{2}$ and $\wt V = \langle -1 \rangle$ combined with Lemma~\ref{lem:unique-minus-one-sl} and Corollary~\ref{m1cor}.
\end{proof}

We have proved Theorem~\ref{Zmaximal} and Theorem~\ref{mainthm:sl-spincover} from the introduction.

\part{Non-simply laced rank two diagrams}

\section{Strategies for reducing the general case to the simply-laced one}    \label{strategy}

Until now we exclusively studied spin covers of maximal compact subgroups of split real Kac--Moody groups of simply laced type. Our next goal is to generalize this concept to arbitrary Dynkin diagrams resp.\ generalized Cartan matrices. We pursue this goal via two strategies: The first one is via epimorphisms between maximal compact subgroups induced by local epimorphisms on amalgam-level in rank two where we replace non-simple edges by non-edges, simple edges or double edges; the second one is via embeddings into larger groups by unfolding the diagrams resp.\ the Cartan matrices to simply-laced cover diagrams as in \cite{Hainke/Koehl/Levy}.

The first strategy will allow us to transform arbitrary Dynkin diagrams resp.\ generalized Cartan matrices into doubly laced ones. The second strategy will work for the resulting doubly laced generalized Cartan matrices.  A combination of both strategies allows us to deal with arbitrary generalized Cartan matrices.

\begin{strategy} \label{simplelacing}
In order to deal with the two non-simply laced spherical diagrams of rank two -- $\mathrm{C}_2$ and $\mathrm{G}_2$ -- we consider point-line models of the Tits buildings of the split real Lie groups $\Sp{4}$ and $\mathrm{G}_2(2)$, the so-called symplectic quadrangle and the so-called split Cayley hexagon. As in the proof of Theorem~\ref{thm:univso}, the Iwasawa decomposition implies that the maximal compact subgroups $\U{2} \leq \Sp{4}$ and $\SO{4} \leq \mathrm{G}_2(2)$ act flag-transitively on the respective point-line geometries.

Their unique double covers $\SO{2} \times \SU{2} \onto \U{2}$ and $\Spin{4} \onto \SO{4}$ fit into the commutative diagrams  
\begin{eqnarray*}
& \xymatrix{\SO{2} \times \SU{2} \ar[rr] \ar[d]  && \mathrm{Spin}(3) \ar[d] \\
\U{2} \ar[rr] && \mathrm{SO}(3)
} & \text{(cf.\ Remark~\ref{coordinatesrev})}
\\ \text{and} \\
& \xymatrix{\mathrm{Spin}(4) \ar[rrr] \ar[d] &  && \mathrm{Spin}(3)  \ar[d] \\ 
\mathrm{SO}(4)  \ar[rrr] &  && \mathrm{SO}(3)
} & \text{(cf.\ Proposition~\ref{surjG2})}
\end{eqnarray*}
which allow one to transform point and line stabilizers in $\U{2}$ and $\SO{4}$ into point, resp.\ line stabilizers in $\SO{3}$ in a way that is compatible with the covering maps. This in turn will allow us to transform a $\Spin{2}$-amalgam for a given two-spherical diagram $\Pi$ into a $\Spin{2}$-amalgam for the simply laced diagram $\Pi^{\mathrm{sl}}$ that one obtains from $\Pi$ by replacing all edges by simple edges. As a consequence --- based on Theorem~\ref{m2} --- in Theorem~\ref{m3} below we will be able to prove that the spin cover $\Spin{\Pi}$ is a non-trivial central extension of $K(\Pi)$ for suitable two-spherical diagrams $\Pi$.

As a caveat we point out that the compatibility of the covering maps in the $\mathrm{C}_2$ case is quite subtle and actually fails under certain circumstances, due to the phenomena described in Lemma~\ref{lem:ama-embed-and-rho-commute-B2}. In order to control these subtleties we introduce the notion of admissible colourings of Dynkin diagrams in Definition~\ref{adtypedef}. These subtleties are also why we actually only replace certain double edges by single edges and additionally employ Strategy~\ref{unfolding} below.

When trying to deal with non-two-spherical diagrams further subtleties arise. The non-spherical Cartan matrices of rank two are of the form \[\begin{pmatrix} 2 & -r \\ -s & 2\end{pmatrix}\] for $r, s \in \NN$ such that $rs \geq 4$. The isomorphism type of the maximal compact subgroup $K$ of the corresponding split real Kac--Moody group depends (only) on the parities of $r$ and $s$. Indeed, in all cases $K$ is isomorphic to a free amalgamated product
\[
K \cong K_1T_K *_{T_K} K_2T_K.
\] where $K_1 \cong \SO{2} \cong K_2$ with $T_K = \{ 1, t_1, t_2, t_1t_2 \} \cong \ZZ/2\ZZ \times  \ZZ/2\ZZ$ and $T_K \cap K_1 = \langle t_1 \rangle$, $T_K \cap K_2 = \langle t_2 \rangle$ and $K_i \unlhd K_iT_K$. We conclude that the isomorphism type of $K$ is known once the action of $t_1$ on $K_2$ and the action of $t_2$ on $K_1$ are known. It turns out that $t_1$ centralizes $K_2$ if and only if $r$ is even and inverts $K_2$ if and only if $r$ is odd; similarly, $t_2$ centralizes $K_1$ if and only if $s$ is even and inverts $K_1$ if and only if $s$ is odd (cf.\ Remark~\ref{tcentralizes}). 

To these four cases of parities of $r$ and $s$ correspond three cases of epimorphisms from $K$ onto compact Lie groups:  $K \onto \SO{2} \times \SO{2}$, if both $r$ and $s$ are even; $K \onto \SO{3}$, if both $r$ and $s$ are odd; $K \onto \U{2}$, if $r$ and $s$ have different parities.  
A study of various double covers of $K$ will, in analogy to what we sketched above for diagrams of type $\mathrm{C}_2$ and $\mathrm{G}_2$, enable us to replace edges labelled $\infty$ by non-edges, simple edges, resp.\ double edges, thus allowing us to understand the non-two-spherical situation as well. 
Again, the case in which $r$ and $s$ have different parities lead to some subtleties that we get control of with the concept of admissible colourings introduced in Definition~\ref{adtypedef}.

Following this strategy leads directly to Proposition~\ref{thm:Spin(Delta)-covers-Spin(Delta-sl)}. 
\end{strategy}

\begin{strategy} \label{unfolding}
Let $\Pi$ with type set $I$ and (generalized) Cartan matrix $A = (a(i,j))_{i,j \in I}$ be an irreducible doubly laced diagram that admits two root lengths.
Then the unfolded Dynkin diagram is the simply laced Dynkin diagram $\Pi^{\mathrm{un}}$ with type set \[I^\mathrm{un} := \{ \pm i \mid i \in I, i \text{ short root} \} \cup \{ i \mid i \in I, i \text{ long root} \}\] and edges defined via the generalized Cartan matrix $A^{\mathrm{un}}= (a^{\mathrm{un}}(i,j))_{i,j \in I^{\mathrm{un}}}$ given by 
\[a^{\mathrm{un}}(i,j) =
\begin{cases}
0, & \text{if $|i|$, $|j|$ have different lengths and $a(|i|,|j|) = 0$}, \\
-1, & \text{if $|i|$, $|j|$ have different lengths and $a(|i|,|j|) \neq 0$}, \\
a(|i|,|j|), & \text{if $|i|$, $|j|$ have the same length and $ij > 0$},   \\
0, & \text{if $|i|$, $|j|$ have the same length and $ij<0$};
\end{cases}\]
(cf.\ Definition~\ref{unfoldeddiagram}).

There exists an embedding of $K(\Pi)$ into $K(\Pi^{\mathrm{un}})$ that by Corollary~\ref{corunfolding} allows one to related the respective spin covers to one another.
\end{strategy}

\section{Diagrams of type $\mathrm{G}_2$} \label{sec:g2}

In this section we prepare Strategy~\ref{simplelacing} for diagrams of type $\mathrm{G}_2$.

\begin{defn}
Denote by $\HH:=\{ a+bi+cj+dk \mid a,b,c,d\in \RR\}$
the \Defn{real quaternions}.
Then the \Defn{standard involution} of $\HH$ is given by
\[\bar{\cdot}:\HH\to\HH: x=a+bi+cj+dk \mapsto \ol{x}=a-bi-cj-dk\ .\]
The set of \Defn{purely imaginary quaternions}, cf.\ \cite[11.6]{Salzmann:1995}, is
\[ \Pu \HH
:= \{ x\in\HH \mid x=-\ol{x} \}
 = \{ bi+cj+dk \mid b,c,d\in \RR\}\subset \HH\ .
\]
\end{defn}

\begin{defn}
The \Defn{split Cayley algebra} $\OO$ is defined as the vector space
$\HH \oplus \HH$ endowed with the multiplication \[xy = (x_1,x_2)(y_1,y_2) =
(x_1y_1 + y_2 \ol{x_2},\ y_1x_2 + \ol{x_1} y_2),\] cf.\
\cite[Section~5.1]{Cohen:1995}.
The \Defn{real split Cayley hexagon} $\HHH(\RR)$ consists of the one-
and two-dimensional real subspaces of $\OO$ for which the restriction
of the multiplication map is trivial, i.e., $\HHH(\RR) = (\PPP, \LLL,
\subset)$ with the point set
\[\PPP := \{ \gen{x}_\RR \mid x \in \OO, x^2 = 0 \neq x \}\]
and the line set
\[\LLL := \{ \gen{x, y}_\RR \mid \gen{x}_\RR \neq \gen{y}_\RR \in \PPP, xy = 0 \},\] cf.\
\cite[Section~5.1]{Cohen:1995}, also \cite[Section~2.4.9]{Maldeghem:1998}.
\end{defn}

\begin{lemma}\label{lem:cayley-zero-square}
Let $x=(x_1,x_2)\in\OO$.
Then
$x^2 = 0 \text{ if and only if } x_1 \in \Pu \HH \text{ and } \ol{x_1} x_1 - \ol{x_2} x_2 = 0$.
\end{lemma}
\begin{proof}
Suppose $x_1 \in \Pu \HH$ and $\ol{x_1} x_1 - \ol{x_2} x_2 = 0$. Then $-x_1x_1 - x_2 \ol{x_2} = \ol{x_1} x_1 - \ol{x_2} x_2 = 0$ and $x_1x_2 + \ol{x_1} x_2 = x_1x_2 - x_1x_2 = 0$, and, thus, $x^2 = 0$. Conversely, suppose $x^2 = 0$. Then $0 = x^2 = (x_1x_1+x_2\ol{x_2},x_1x_2 + \ol{x_1} x_2)$, so if $x_2=0$, then $x_1=0$, and there is nothing to show. For $x_2 \neq 0$ multiplication of $x_1x_2 + \ol{x_1} x_2 = 0$ from the right with $x_2^{-1}$ gives $x_1 + \ol{x_1} = 0$ or, equivalently, $x_1 \in \Pu \HH$. But now $0 = x_1x_1 + x_2 \ol{x_2} = - \ol{x_1} x_1 + \ol{x_2} x_2$, and the claim follows.
\end{proof}

\begin{definition}
Let $N:\HH\to \RR : x\mapsto x\ol{x}$ be the norm map associated to the
standard involution of the real quaternions.  Moreover, let
\[\UH:=\{x\in\HH \mid x\ol{x}=1 \}\]
be the group of real quaternions of norm one.
\end{definition}

By Remark~\ref{mapsso4} (see also \cite[Lemma 11.22]{Salzmann:1995}), the group $\SO{4}$ is isomorphic to
the group consisting of the maps 
\[ \HH \to \HH : x \mapsto a x b^{-1} \quad\text{ for }\quad a,b\in \UH \ . \]

\begin{lemma} \label{flagtransg2}
The group $\SO{4} \cong  \{ \HH \to \HH : x \mapsto a x b^{-1} \mid a,b\in \UH \} $ acts flag-transitively on the split Cayley hexagon $\HHH(\RR)$ via 
\[ \gen{(x_1,x_2)}_\RR \mapsto \gen{(a x_1 a^{-1},\  a x_2 b^{-1})}_\RR\ . \]
\end{lemma}

\begin{proof}
We show that the map 
\[ f_{a,b}:\OO\to\OO : (x_1,x_2) \mapsto (a x_1 a^{-1},\ a x_2 b^{-1})\]
is an algebra automorphism of $\OO$ for all $a, b \in \UH$. Then it
also induces an automorphism of $\HHH(\RR)$, because it is defined via the multiplication in $\OO$. We have
\begin{align*}
 &\ f_{a,b}(x_1,x_2)\cdot f_{a,b}(y_1,y_2) \\
=&\ \left( a x_1 a^{-1},\ a x_2 b^{-1} \right) \cdot \left( a y_1 a^{-1},\ a y_2 b^{-1} \right)  \\
=&\ \left( (a x_1 a^{-1}) (a y_1 a^{-1}) + (a y_2 b^{-1}) \ol{(a x_2 b^{-1})},\ 
           (a y_1 a^{-1}) (a x_2 b^{-1}) + \ol{(a x_1 a^{-1})} (a y_2 b^{-1}) \right) \\
=&\ \left( a (x_1y_1 + y_2 \ol{x_2}) a^{-1},\ a (y_1x_2 + \ol{x_1} y_2) b^{-1} \right) \\
=& f_{a,b}\big((x_1,x_2)(y_1,y_2)\big) \, ,
\end{align*}
whence the map $f_{a,b}$ is multiplicative. Since it is certainly an
$\RR$-linear bijection, it is an algebra automorphism of $\OO$.
Flag-transitivity is an immediate consequence of the Iwasawa
decomposition and the fact that $\SO{4}$ is the maximal compact subgroup of the (simply connected semisimple) split real group 
$\mathrm{G}_2(2)$ of type $\mathrm{G}_2$.
\end{proof}

\begin{remark}
There exists a nice direct proof of flag-transitivity without making use of the Iwasawa decomposition and the structure theory of $\mathrm{G}_2(2)$ that in particular illustrates how to
compute point and line stabilizers and, thus, helps our understanding how to
properly embed the circle group into $\SO{4}$ for our amalgamation problem.

Let \[\Pu \OO := \Pu \HH \oplus \HH\] be the set of \Defn{purely imaginary
split octonions} and consider the points of the real (projective) quadric
\[\ol{x_1} x_1 - \ol{x_2} x_2 = 0\] in $\Pu \OO$, i.e., the set of isotropic
one dimensional real subspaces of $\Pu \OO$. By Remark~\ref{mapsso4} (see also \cite[11.24]{Salzmann:1995}),
the group $\SO{3}$ is isomorphic to the group consisting of the maps \[\Pu
\HH \to \Pu \HH : x \mapsto a x a^{-1} \text{ for } a\in\UH\] and acts
transitively on the set $\{ \{ x, -x \} \subset \Pu\HH \mid x\ol{x} = 1 \}$.  Moreover, for
each $a, x, z \in \UH$, there exists a unique solution $b \in \UH$ for the
equation $z = a xb^{-1}$.

Hence $\SO{4}$ acts transitively on the set \[\left\{ \big\{ (x,y), (-x,-y) \big\} \subset \Pu\HH \times \HH \mid x\ol{x} = 1 = y\ol{y} \right\}.\] But this
implies point transitivity on the projective real quadric $\ol{x_1} x_1 -
\ol{x_2} x_2 = 0$ in $\Pu \OO$, which, in turn, implies point transitivity
on $\HHH(\RR)$ by Lemma~\ref{lem:cayley-zero-square}.

Now choose one point of $\HHH(\RR)$, say $\langle (i,i) \rangle_\RR$. Then a
point $\langle y \rangle_\RR = \langle (y_1,y_2)\rangle_\RR$ is collinear to
this point if and only if \[(iy_1 - y_2i, y_1i - iy_2) = 0 \Longleftrightarrow
y_1 = -iy_2i.\] So the question of transitivity of the stabilizer of $\langle
(i,i)\rangle_\RR$ in $\SO{4}$ on the line pencil of $\gen{(i,i)}_\RR$ in
$\HHH(\RR)$ is equivalent to the question of transitivity of the stabilizer
of $\gen{i}_\RR$ in $\SO{3}$ on the line pencil of $\gen{i}_\RR$ in the
projective plane $\Pu \HH \cong \Pu \HH \oplus \{ 0 \} \subset \Pu \OO$.

But since the latter is transitive, so is the former, and hence $\SO{4}$
acts flag-transitively on $\HHH(\RR)$ by means of the maps given in
Lemma~\ref{flagtransg2}.
\end{remark}

The stabilizer of the point $\gen{(i,i)}_\RR$ contains the circle group $\SO{2}$ acting naturally diagonally on $\gen{j_1,k_1}_{\RR} \oplus \gen{j_2,k_2}_{\RR}$. The stabilizer of the line $\gen{(i,i) , (j,-j)}_\RR$ in $\SO{4}$ contains $\SO{2}$ acting naturally by rotations on $\gen{i_1,j_1}_{\RR}$ and by rotations in the opposite direction on $\gen{i_2,j_2}_{\RR}$ in such a way that an element with first coordinate $\lambda i_1 + \mu j_1$ has second coordinate $\lambda i_2 - \mu j_2$.

\begin{notation}
We denote these embeddings of the circle group into $\SO{4}$ by $\eta_p$ resp.\ $\eta_l$.
Concretely, one has
\begin{align*}
\eta_p &: \SO{2}\to \SO{4} :  D(\alpha)\mapsto \begin{pmat} I_2 &  \\ & D(\alpha)\end{pmat}
    =\eps_{34}(D(\alpha))\ ,\\
\eta_l &: \SO{2}\to \SO{4} : D(\alpha)\mapsto \tilde{D}(\alpha):=
  \begin{pmatrix}
    \cos(2\alpha) & & & \sin(2\alpha) \\
     & \cos(\alpha) & -\sin(\alpha) & \\
     & \sin(\alpha) & \cos(\alpha) & \\
    -\sin(2\alpha) & & & \cos (2\alpha)
  \end{pmatrix} = \eps_{14}(D(2\alpha)) \cdot \eps_{23}(D(-\alpha))\ .
\end{align*}
\end{notation}

\begin{lemma} \label{lem:rank1-inv-G2}
Let $B:=\diag(-1,1,1,-1)$, $C:=\diag(-1,-1,1,1)\in \U{2}$. Then the following hold:
\begin{enumerate}
\item The map $\gamma_B: \SO{4}\to \SO{4} : A\mapsto B\cdot A\cdot B^{-1}=B\cdot A\cdot B$
is an automorphism of $\SO{4}$ such that
\[\gamma_B\circ \eta_p=\eta_p\circ \inv \qquad \text{and} \qquad \gamma_B\circ \eta_l=\eta_l.\]
\item The map $\gamma_C:\SO{4}\to\SO{4} : A\mapsto C\cdot A\cdot C^{-1}=C\cdot A\cdot C$
is an automorphism of $\SO{4}$ such that
\[\gamma_C\circ \eta_p=\eta_p \qquad \text{and} \qquad \gamma_C\circ \eta_l=\eta_l\circ\inv.\]
\end{enumerate}
\end{lemma}

\begin{proof}
Straightforward.
\end{proof}

\begin{notation} \label{etaspin}
In the following, let
\begin{align*}
\tilde\eta_p &: \Spin{2} \to \Spin{4} : S(\alpha)\mapsto \tilde\eps_{34}\big(S(\alpha)\big)\ , \\
\tilde\eta_l &: \Spin{2}\to \Spin{4} : S(\alpha)\mapsto \tilde\eps_{14}\big(S(2\alpha)\big)\cdot \tilde\eps_{23}\big(S(-\alpha)\big)\ ,
\end{align*}
and recall from Theorem~\ref{rho}(b) that for $n\geq 2$ the map
\[\rho_n:\Spin{n}\to \SO{n} \]
is the twisted adjoint representation.
\end{notation}

In order to generalize our definition of spin amalgams, we need
$\tilde\eta_p$ and $\tilde\eta_l$ to be injective. For the former this
is clear from its definition, for the latter we verify it now.

\begin{lemma} \label{lem:eta-l-inj}
The map $\tilde\eta_l$ is a monomorphism.
\end{lemma}

\begin{proof}
For $S(\alpha)\in \ker \tilde\eta_l$ one has
\begin{align*}
\tilde\eta_l(S(\alpha))
=\big(\cos(2\alpha)+\sin(2\alpha)e_1e_4\big)
 \big(\cos(\alpha)-\sin(\alpha)e_2e_3\big)
=1
\end{align*}
and, thus, $\alpha\in \pi\ZZ$. As $\tilde\eta_l\big ( S(\pi)\big)=-1$ and $\tilde\eta_l\big(S(2\pi)\big)=1$, one obtains
\[\ker \tilde\eta_l=\{ S(\alpha) \mid \alpha\in 2\pi\ZZ \}=\{ 1\}\ . \qedhere\]
\end{proof}

\begin{lemma} \label{lem:ama-embed-and-rho-commute-G2}
One has
\begin{align*}
\rho_4\circ \tilde\eta_p= {\eta}_p\circ \rho_2\ , && 
\rho_4\circ \tilde\eta_l= {\eta}_l\circ \rho_2\ .
\end{align*}
\end{lemma}

\begin{proof}
For $\alpha\in \RR$,
\begin{align*}
  (\rho_4\circ \tilde\eta_p)\big(S(\alpha)\big)
&=(\rho_4\circ \tilde\eps_{34})\big(S(\alpha)\big)
 =(\eps_{34}\circ\rho)\big(S(\alpha)\big)
 =(\eta_p\circ\rho)\big(S(\alpha)\big)
\end{align*}
and
\begin{align*}
  (\rho_4\circ \tilde\eta_l)\big(S(\alpha)\big)
&=\rho_4\big(\tilde\eps_{14}\big(S(2\alpha)\big)
   \cdot \tilde\eps_{23}\big(S(-\alpha)\big)\big)
 =\eps_{14}\big(D(4\alpha)\big)\cdot \eps_{23}\big(D(-2\alpha)\big) \\
&=(\eta_l\circ \rho_2)\big(S(\alpha)\big)\ .
\qedhere
\end{align*}
\end{proof}

\begin{lemma} \label{lemmaleftright}
Let $V:=\HH$ and $\EEE:=\{ 1, i , j , k\}$. Then the following hold:
\begin{enumerate}
\item For $a,b\in \UH$ the
maps \[\ell_a: \HH\to \HH : x\mapsto ax \text{ and } r_b:\HH\to \HH : x\mapsto xb^{-1}\]
preserve the norm $N:\HH\to \RR : x\mapsto x\ol{x}$. In particular,
$\ell_a,r_b\in \SO{\HH}\cong \SO{4}$.
\item \label{lem:UHxUH-covers-SO4}
The map
\[ \UH\times \UH\to\SO{4} : (a,b)\mapsto \ell_{a}\circ r_{b} \]
is a group epimorphism with kernel $\{(1,1),(-1,-1)\}$.
\end{enumerate}
\end{lemma}

\begin{proof}
This has been discussed in Remark~\ref{mapsso4}. Alternatively, it also follows from \cite[Lemma~{11.22} to Corollary~{11.25}]{Salzmann:1995}.
\end{proof}

For $\EEE=\{ 1, i , j , k\}$ define
\begin{align*}
L_a:=M_\EEE(\ell_a)\in \SO{4}\ , && R_b:=M_\EEE(r_b)\in \SO{4}\ .
\end{align*}

\begin{remark} \label{remark5}
\begin{enumerate}
\item
The map $\UH\times\UH\to\SO{4}$ from Lemma~\ref{lem:UHxUH-covers-SO4} equals the covering map $\rho_4$, cf.\ Remark~\ref{mapisrho}.
\item Given $x=a+bi+cj+dk\in \UH$, a short computation shows
\begin{align*}
L_x=\begin{pmatrix} a & - b & -c & -d \\
b & a & -d & c \\
c & d & a & -b \\
d& -c & b & a \end{pmatrix}\ , &&
R_x=\begin{pmatrix} a & b & c & d \\
-b & a & -d & c \\
-c & d & a & -b \\
-d & -c &  b & a \end{pmatrix}
\end{align*}
as $\RR$-linear maps via left action. Lemma~\ref{lem:UHxUH-covers-SO4} implies that for all $x,y \in \mathrm{U}_1(\HH)$ one has $L_xR_y = R_yL_x$ and that up to scalar multiplication with $-1$ the matrices $L_x$ and $R_y$ are uniquely determined by their product.
\item The action from Lemma~\ref{flagtransg2} translates into \[\omega: \SO{4}\times \HHH(\RR) \to \HHH(\RR) : \big( L_aR_{b} , (x,y)\big) \mapsto ( L_a R_{a}\cdot x, L_aR_{b}\cdot y).\]
\item \label{30a}
For $\alpha\in \RR$, one has $\eta_p\big(D(\alpha)\big)=L_a\cdot R_a$ with
$a=\cos(\tfrac{\alpha}{2})-\sin(\tfrac{\alpha}{2})i$, i.e.,
\begin{align*}
  \begin{pmatrix} I_2 & \\ & D(\alpha) \end{pmatrix}
= \begin{pmatrix} D(\frac{\alpha}{2}) &  \\ & D(\frac{\alpha}{2}) \end{pmatrix}
    \cdot
  \begin{pmatrix} D(-\frac{\alpha}{2}) & \\ & D(\frac{\alpha}{2})\end{pmatrix}\ .
\end{align*}
\item \label{30b}
For $\alpha\in \RR$, we have $\eta_l\big(D(\alpha)\big)=L_a\cdot R_b$ with $a=\cos(\frac{\alpha}{2})-\sin(\frac{\alpha}{2})k$ and $b=\cos(\frac{3\alpha}{2})+\sin(\frac{3\alpha}{2})k$, i.e.,
\begin{align*}
\eta_l\big(D(\alpha)\big)=
\begin{pmatrix}
\cos(\frac{\alpha}{2}) &  & \sin(\frac{\alpha}{2}) \\
& D(\frac{\alpha}{2}) & \\
-\sin(\frac{\alpha}{2}) &  & \cos(\frac{\alpha}{2})
\end{pmatrix}\cdot
\begin{pmatrix}
\cos(\frac{3\alpha}{2}) &  & \sin(\frac{3\alpha}{2}) \\
& D(-\frac{3\alpha}{2}) & \\
-\sin(\frac{3\alpha}{2}) &  & \cos(\frac{3\alpha}{2})
\end{pmatrix}\ .
\end{align*}
\end{enumerate}
\end{remark}

The subgroup of right translations $R_b$ is normal in $\SO{4}$, the resulting quotient is isomorphic to $\SO{3}$. This canonical projection induces a surjection from the split Cayley hexagon onto the real projective plane, given by the projection $\langle (x_1,x_2) \rangle_\RR \mapsto \langle x_1 \rangle_\RR$. (Cf.\ \cite[Section~5]{Gramlich:1998}. An alternative description of this surjection can be found in \cite{GramlichVanMaldeghem}.)

The following lemma describes how the corresponding embedded circle groups behave under this surjection.

\begin{lemma} \label{lemsurjG2} \label{lemsurjG2I}
\begin{enumerate}
\item There is an epimorphism $\eta_1:\SO{4}\to  \SO{3}$ such that
\begin{align*}
\eta_1\circ \eta_p=\eps_{23}\ , && \eta_1\circ \eta_l=\eps_{12}\ .
\end{align*}
\item There is an epimorphism $\eta_2:\SO{4}\to  \SO{3}$ such that
\begin{align*}
\eta_2\circ \eta_p=\eps_{12}\ , && \eta_2\circ \eta_l=\eps_{23}\ .
\end{align*}
\end{enumerate}
\end{lemma}

\begin{proof}
By Remark~\ref{mapsso4}, the map
\[ \psi: \SO{4}\to \eps_{\{2,3,4\}}\big(\SO{3}\big): L_aR_b\mapsto L_{a}R_{a}\]
is an epimorphism (see also \cite[Corollaries~11.23 and 11.24]{Salzmann:1995}). By Remark \ref{30a},
\[\psi\circ \eta_p=\eta_p=\eps_{34}\ ,\]
and, by Remark \ref{30b},
\begin{align*}
(\psi\circ \eta_l)\big(D(\alpha)\big)&=\begin{pmatrix}
\cos(\frac{\alpha}{2}) &  & \sin(\frac{\alpha}{2}) \\
& D(\frac{\alpha}{2}) & \\
-\sin(\frac{\alpha}{2}) &  & \cos(\frac{\alpha}{2})
\end{pmatrix}
\begin{pmatrix}
\cos(\frac{\alpha}{2}) &  & -\sin(\frac{\alpha}{2}) \\
& D(\frac{\alpha}{2}) & \\
\sin(\frac{\alpha}{2}) &  & \cos(\frac{\alpha}{2})
\end{pmatrix}=\begin{pmatrix}
1 &  &  \\
& D(\alpha) & \\
& & 1
\end{pmatrix} = \eps_{23}(D(\alpha)) \ .
\end{align*}
Therefore, the map $\eta_1:=\eps_{\{2,3,4\}}^{-1}\circ \psi$ has the desired properties. The existence of the map $\eta_2$ now follows from Lemma~\ref{10}.
\end{proof}

The final result of this section allows us to carry out Strategy~\ref{simplelacing} for edges of type $\mathrm{G}_2$ in Theorem~\ref{thm:Spin(Delta)-covers-Spin(Delta-sl)} below.

\begin{proposition} \label{surjG2}\label{surjG2I}
\begin{enumerate}
\item There is an epimorphism $\tilde\eta_1:\Spin{4}\to  \Spin{3}$ such that
\begin{align*}
\tilde\eta_1\circ \tilde\eta_p=\tilde\eps_{23}\ , && \tilde\eta_1\circ \tilde\eta_l=\tilde\eps_{12}\ .
\end{align*}
\item There is an epimorphism $\tilde\eta_2:\Spin{4}\to  \Spin{3}$ such that
\begin{align*}
\tilde\eta_2\circ \tilde\eta_p=\tilde\eps_{12}\ , && \tilde\eta_2\circ \tilde\eta_l=\tilde\eps_{23}\ .
\end{align*}
\end{enumerate}
\end{proposition}

\begin{proof}
The map $\tilde\eta_1 : \Spin{4} \to \Spin{3} : u+\II v \mapsto u+v $ makes the inner/right-hand quadrangle of the following diagram commute (see Remarks~\ref{Spin4Spin3} and \ref{mapisrho}):
\[\xymatrix{
\Spin{2} \ar[ddd]_{\rho_2} \ar[rd]_{\tilde\eta_p} \ar[rrrd]^{\tilde\eps_{23}} \\
& \Spin{4} \ar[rr]_{\tilde\eta_1} \ar[d]^{\rho_4} & & \Spin{3} \ar[d]_{\rho_3} \\
& \SO{4} \ar[rr]^\eta & & \SO{3} \\
\SO{2} \ar[ur]^{\eta_p} \ar[urrr]_{\eps_{23}}
}
\]
The lower triangle commutes by Lemma~\ref{lemsurjG2}. The left-hand quadrangle commutes by Lemma~\ref{lem:ama-embed-and-rho-commute-G2}. 
Hence \[\eps_{23} \circ \rho_2 = \eta \circ \eta_p \circ \rho_2 = \eta \circ \rho_4 \circ \tilde\eta_p = \rho_3 \circ \tilde\eta_1 \circ \tilde\eta_p\] and \[\ker(\rho_3 \circ \tilde\eta_1 \circ \tilde\eta_p)=\ker(\eps_{23}\circ\rho_2)=\ker\rho_2 = \{\pm 1_{\Spin{2}} \}.\] 
Therefore $\tilde\eps_{23} = \tilde\eta_1 \circ \tilde \eta_p$ by Proposition~\ref{prop:zusatz}. In particular, also the upper triangle of the diagram commutes.

The second claim concerning $\tilde\eta_1$ follows by an analogous argument

The claims concerning $\tilde\eta_2$ are now immediate by Lemma~\ref{10} and Proposition~\ref{prop:lift-aut-soN}
\end{proof}

\section{Diagrams of type $\mathrm{C}_2$} \label{sec:bc2}

In this section we prepare Strategy~\ref{simplelacing} for diagrams of type $\mathrm{C}_2$.

\begin{definition}
Let $\Sp{4}$ be the matrix group with respect to the $\RR$-basis $e_1$, $ie_1$, $e_2$, $ie_2$ of $\CC^2$ leaving the real alternating form \[\big( (x_1,x_2,x_3,x_4),(y_1,y_2,y_3,y_4)\big) \mapsto x_1y_2 - x_2y_1 + x_{3}y_{4} - x_{4}y_{3}\] invariant.
\end{definition}

\begin{remark} \label{coordinates}
The maximal compact subgroup of $\Sp{4}$ is the group $\U{2}$, embedded as follows. Let $e_1$, $e_2$ be the standard basis of $\CC^2$ and consider $\U{2}$ as the isometry group of the scalar product $\big((v_1,v_2),(w_1,w_2)\big) \mapsto \ol{v_1}w_1+\ol{v_2}w_2$. Defining
\begin{align*}
&& && x_1&:=\Rea(v_1), & x_2&:=\Ima(v_1), && && \\
&& && x_3&:=\Rea(v_2), & x_4&:=\Ima(v_2), && && \\
&& && y_1&:=\Rea(w_1), & y_2&:=\Ima(w_1), && && \\
&& && y_3&:=\Rea(w_2), & y_4&:=\Ima(w_2), && && 
\end{align*}
we compute
\begin{displaymath}
\ol{v_1} w_1 + \ol{v_2} w_2 = x_1y_1 + x_2y_2 + x_3y_3 + x_4y_4 + i(x_1y_2 - x_2y_1 + x_{3}y_{4} - x_{4}y_{3}).
\end{displaymath}
Since two complex numbers are equal if and only if real part and imaginary part coincide, the group $\U{2}$ preserves the form $x_1y_2 - x_2y_1 + x_{3}y_{4} - x_{4}y_{3}$ and we have found an embedding in $\Sp{4}$, acting naturally on the $\RR$-vector space $\CC^2$ with $\RR$-basis $e_1$, $ie_1$, $e_2$, $ie_2$.

As $\U{2}$ also preserves the form $x_1y_1 + x_2y_2 + x_3y_3 + x_4y_4$ with respect to the $\RR$-basis $e_1$, $ie_1$, $e_2$, $ie_2$ of $\CC^2$ it is at the same time also a subgroup of $\O{4}$, in fact of $\SO{4}$, since $\U{2}$ is connected with respect to its Lie group topology. We will give concrete coordinates for this embedding in Remark~\ref{coordinatesrev} below.
\end{remark}

The real symplectic quadrangle can be modelled as the point-line geometry consisting of the one-dimensional and two-dimensional subspaces of $\RR^4$ which are totally isotropic with respect to a nondegenerate alternating bilinear form, cf.\ \cite[Section~2.3.17]{Maldeghem:1998}, also \cite[Section~10.1]{Buekenhout/Cohen:2013}.

Using the $\RR$-basis $e_1$, $ie_1$, $e_2$, $ie_2$ and the $\RR$-alternating form on $\CC^2$ given in Remark~\ref{coordinates}, an incident point-line pair of the resulting symplectic quadrangle of $\Sp{4}$ is given by $\langle e_1 \rangle_\RR \subset \langle e_1, e_2 \rangle_\RR$. The stabilizer in $\U{2}$ of the point $\langle e_1 \rangle_\RR$ is isomorphic to $\mathrm{O}_1(\RR) \times \U{1}$ where the first factor acts diagonally on ${\langle e_1 \rangle_\RR \oplus \langle i e_1 \rangle_{\RR}}$ and the second factor acts naturally on ${\langle e_2 \rangle_{\CC}}$. The stabilizer of the line $\langle e_1, e_2 \rangle_\RR$ is isomorphic to $\mathrm{O}_2(\RR)$ acting diagonally on ${\langle e_1,e_2 \rangle_{\RR} \oplus \langle ie_1,ie_2 \rangle_{\RR}}$.

\begin{definition} \label{rank1embb2}
Let \[\zeta_p : \SO{2} \to \U{2} \subset \Sp{4} \cap \SO{4} : \begin{pmatrix} x & y \\ -y & x \end{pmatrix} \mapsto \begin{pmatrix} 1 \\ & 1 \\ & & x & y \\ & & -y & x \end{pmatrix}\] and let \[\zeta_l : \SO{2} \to \U{2} \subset \Sp{4} \cap \SO{4} : \begin{pmatrix} x & y \\ -y & x \end{pmatrix} \mapsto \begin{pmatrix} x && y \\ & x && y \\ -y & & x &  \\ & -y & & x \end{pmatrix}\] be the embeddings of the circle group arising as point-stabilizing resp.\ line-stabilizing rank one groups as above with respect to the $\RR$-bases $e_1$, $e_2$ of $\RR^2$ and $e_1$, $ie_1$, $e_2$, $ie_2$ of $\CC^2$.
\end{definition}

Recall the definitions of $D(\alpha)$ and $S(\alpha)$ from Notation~\ref{nota:Dalpha-Salpha}.

\begin{notation} \label{su2acts}
In the following, identify $\CC=\{x+iy\mid x,y\in\RR\}$ with $\{ \begin{pmat} x & y \\ -y & x \end{pmat} \mid x, y\in \RR\}$. This identification, in particular, embeds $\SO{2}$ into $\CC$ as the unit circle group. For $\alpha\in \RR$, let
\[\tilde{D}(\alpha):=\zeta_l(D(\alpha))
=\begin{pmatrix} \cos(\alpha) && \sin(\alpha) \\ & \cos(\alpha) && \sin(\alpha) \\ -\sin(\alpha) & & \cos(\alpha) &  \\ & -\sin(\alpha) & & \cos(\alpha) \end{pmatrix}
\in \U{2}\ ..\]
\end{notation}

\begin{lemma} \label{lem:rank1-inv-B2}
Let $B:=\diag(-1,1,-1,1)$, $C:=\diag(-1,-1,1,1)\in \U{2}$. Then the following hold:
\begin{enumerate}
\item The map $\gamma_B:\U{2}\to \U{2} : A\mapsto B\cdot A\cdot B^{-1}=B\cdot A\cdot B$
is an automorphism of $\U{2}$ such that
\[\gamma_B\circ \zeta_p=\zeta_p\circ \inv \qquad \text{and} \qquad \gamma_B\circ \zeta_l=\zeta_l.\]
\item The map $\gamma_C:\U{2}\to \U{2} : A\mapsto C\cdot A\cdot C^{-1}=C\cdot A\cdot C$
is an automorphism of $\U{2}$ such that
\[\gamma_C\circ \zeta_p=\zeta_p \qquad \text{and} \qquad \gamma_C\circ \zeta_l=\zeta_l\circ\inv.\]
\end{enumerate}
\end{lemma}

\begin{proof}
Straightforward.
\end{proof}

\begin{lemma}
Let $\alpha\in \RR$. Then $\tilde{D}(\alpha)\in \SU{2}$ and $\begin{pmat} D(-\alpha) & \\ & D(\alpha)\end{pmat}\in \SU{2}$.
\end{lemma}

\begin{proof}
Given $\alpha\in \RR$, one has the following (where the determinant is taken in $\mathrm{SL}_2(\CC)$):
\begin{align*}
\det\big(\tilde{D}(\alpha)\big)
=\begin{pmat} \cos(\alpha) & \\ &  \cos(\alpha)\end{pmat}^2+
 \begin{pmat} \sin(\alpha) & \\ &  \sin(\alpha)\end{pmat}^2
=\begin{pmat} \cos(\alpha)^2+ \sin(\alpha)^2 & \\ &  \cos(\alpha)^2+\sin(\alpha)^2\end{pmat}
=I_2=1_\CC
\end{align*}
and
\[\det \begin{pmat} D(-\alpha) & \\ & D(\alpha)\end{pmat}=D(-\alpha)\cdot D(\alpha)=I_2=1_\CC\ . \qedhere\]
\end{proof}

\begin{remark} \label{coordinatesrev}
Returning to the embedding $\U{2} \to \SO{4}$ mentioned in Remark~\ref{coordinates}, the group \[\SU{2} = \left\{ \begin{pmatrix} x_1+iy_1 & x_2+iy_2 \\ -x_2+iy_2 & x_1-iy_2 \end{pmatrix} \mid x_1, x_2, y_1, y_2 \in \RR, x_1^2+x_2^2+y_1^2+y_2^2=1 \right\} \leq \U{2}\] acts $\RR$-linearly on $\CC^2$ with transformation matrices \[\begin{pmatrix} x_1 & -y_1 & x_2 & -y_2 \\ y_1 & x_1 & y_2 & x_2 \\ -x_2 & -y_2 & x_1 & y_1 \\ y_2 & -x_2 & -y_1 & x_1 \end{pmatrix}\] with respect to the basis $e_1$, $ie_1$, $e_2$, $ie_2$.
Remark~\ref{remark5} implies that the map
\begin{eqnarray*}
\SU{2} & \to & \SO{4} \\
\begin{pmatrix} x_1+iy_1 & x_2+iy_2 \\ -x_2+iy_2 & x_1-iy_2 \end{pmatrix} & \mapsto & R_{x_1-y_1i+x_2j-y_2k} = \begin{pmatrix} x_1 & -y_1 & x_2 & -y_2 \\ y_1 & x_1 & y_2 & x_2 \\ -x_2 & -y_2 & x_1 & y_1 \\ y_2 & -x_2 & -y_1 & x_1 \end{pmatrix}
\end{eqnarray*} 
injects $\SU{2}$ into $\SO{4}$. The restriction of this map to $\SO{2} \subset \SU{2}$ by setting $y_1=0=y_2$ provides the transformations $\tilde D$ from Notation~\ref{su2acts}.
The group \[\left \{ \begin{pmatrix} \cos(\alpha) + i\sin(\alpha) & 0 \\ 0 & \cos(\alpha) + i\sin(\alpha) \end{pmatrix} \right \} \cong \U{1} \cong \SO{2}\] acts with transformation matrices \[\begin{pmatrix} D(-\alpha) & 0 \\ 0 & D(-\alpha) \end{pmatrix} = \begin{pmatrix} \cos(\alpha) & -\sin(\alpha) & 0 & 0 \\ \sin(\alpha) & \cos(\alpha) & 0 & 0 \\ 0 & 0 & \cos(\alpha) & -\sin(\alpha) \\ 0 & 0 & \sin(\alpha) & \cos(\alpha) \end{pmatrix},\] i.e., Remark~\ref{remark5} implies that the map
\begin{eqnarray*}
\U{1} & \to & \SO{4} \\
\begin{pmatrix} \cos(\alpha) + i\sin(\alpha) & 0 \\ 0 & \cos(\alpha) + i\sin(\alpha) \end{pmatrix} & \mapsto & L_{\cos(\alpha)+i\sin(\alpha)} = \begin{pmatrix} \cos(\alpha) & -\sin(\alpha) & 0 & 0 \\ \sin(\alpha) & \cos(\alpha) & 0 & 0 \\ 0 & 0 & \cos(\alpha) & -\sin(\alpha) \\ 0 & 0 & \sin(\alpha) & \cos(\alpha) \end{pmatrix}
\end{eqnarray*} 
injects $\U{1}$ into $\SO{4}$.
Altogether, using Remark~\ref{Spin4Spin3},  we obtain the following commutative diagram:
\[\xymatrix{
\U{1} \times \SU{2} \ar[rr] \ar[d]_{\widehat\rho} && \mathrm{U}_1(\HH) \times \mathrm{U}_1(\HH) \cong \Spin{3} \times \Spin{3} \cong \Spin{4} \ar[d]^{\rho_4} \\
\U{2} \ar[rr] && \SO{4} 
}.
\] 
\end{remark}

Our candidate for a spin cover of the group $\U{2}$
therefore is its double cover 
\[ \U{1} \times\SU{2} \to \U{2} : (z,A) \mapsto zA.\]
Note that the fundamental group of $\U{2}$ equals $\ZZ$, as the determinant map $\det : \U{2} \to \U{1} \cong \SO{2}$ induces an isomorphism of fundamental groups; its simply connected universal cover is isomorphic to $\RR \times \SU{2}$. The above double cover is unique up to isomorphism, because $\ZZ$ has a unique subgroup of index two (cf.\ \cite[Theorem~1.38]{Hatcher:2002}).

\begin{notation} \label{zetaspin}
In the following, let
\begin{align*}
\tilde{\zeta}_p &: \Spin{2} \to \SO{2}\times\SU{2} \subset \Spin{4} : S(\alpha)\mapsto \left( D(\alpha), \begin{pmat} D(-\alpha) & \\ & D(\alpha)\end{pmat}\right)\ , \\
\tilde{\zeta}_l &: \Spin{2}\to \SO{2}\times\SU{2} \subset \Spin{4} : S(\alpha)\mapsto \big( \id , \tilde{D}(\alpha)\big)
\end{align*}
and let
\[
\widehat\rho:\SO{2}\times \SU{2}\to \U{2} : (z,A)\mapsto \begin{pmat} z & \\ & z\end{pmat}\cdot A\ .
\]
\end{notation}

Recall the maps 
\begin{align*}
\rho_2:\Spin{2}\to \SO{2} & : S(\alpha)\mapsto D(2\alpha), \\
\mathrm{sq} : G \to G & : x \mapsto x^2, \\
\mathrm{inv} : G \to G & : x \mapsto x^{-1}.
\end{align*}

\begin{lemma} \label{concretedescriptiontilde}
One has 
\begin{align*}
\tilde{\zeta}_p= \tilde\eps_{34} && \tilde{\zeta}_l\circ\sq = \tilde\eps_{23}\cdot\tilde\eps_{14}
\end{align*}
\end{lemma}

\begin{proof}
Using the identification $\Spin{4} \cong \Spin{3} \times \Spin{3}$ and considering the left hand factor as transformations by left multiplication and the right hand factor as transformations by right multiplication of unit quaternions we compute
\begin{align*}
\tilde\eps_{34}\big(\cos(\alpha)+\sin(\alpha)e_1e_2\big) &= \cos(\alpha)+\sin(\alpha)e_3e_4 \\
&= \cos(\alpha)-\II i\sin(\alpha) & \text{by Remark~\ref{Spin4Spin3}} \\
&= \big(\cos(\alpha)-i\sin(\alpha),\cos(\alpha)+i\sin(\alpha)\big) & \text{by Remark~\ref{Spin4Spin3}} \\
&=\left(L_{\cos(\alpha)-i\sin(\alpha)},R_{\cos(\alpha)-i\sin(\alpha)}\right) & \text{by Lemma~\ref{lemmaleftright}, Remark~\ref{remark5}} \\ 
&=\left( D(\alpha), \begin{pmatrix} D(-\alpha) & \\ 
& D(\alpha)\end{pmatrix}\right) & \text{by Remark~\ref{coordinatesrev}} \\
&=\tilde{\zeta}_p\big(\cos(\alpha)+\sin(\alpha)e_1e_2\big)
\end{align*}
and
\begin{align*}
& \tilde\eps_{23}\big(\cos(\alpha)+\sin(\alpha)e_1e_2\big)\tilde\eps_{14}\big(\cos(\alpha)+\sin(\alpha)e_1e_2\big) \\ =& \big(\cos(\alpha)+\sin(\alpha)e_2e_3\big)\big(\cos(\alpha)+\sin(\alpha)e_1e_4\big) \\
=&\big(\cos(\alpha)\big)^2+j\cos(\alpha)\sin(\alpha) + \II\left(\big(\sin(\alpha)\big)^2-j\cos(\alpha)\sin(\alpha)\right) & \text{by Remark~\ref{Spin4Spin3}}  \\
=&\left(1,\big(\cos(\alpha)+j\sin(\alpha)\big)^2\right) & \text{by Remark~\ref{Spin4Spin3}} \\
=&\left(\id,R_{\big(\cos(\alpha)+j\sin(\alpha)\big)^2}\right) & \text{by Lemma~\ref{lemmaleftright}, Remark~\ref{remark5}} \\
=&\left(\id,\tilde D(2\alpha)\right) & \text{by Remark~\ref{coordinatesrev}} \\
=& \left(\tilde{\zeta}_l\circ\sq\right)\big(\cos(\alpha)+\sin(\alpha)e_1e_2\big). \qedhere
\end{align*}
\end{proof}

The following observation is the analog of Lemma~\ref{lem:ama-embed-and-rho-commute-G2}.

\begin{lemma} \label{lem:ama-embed-and-rho-commute-B2} 
One has
\begin{align*}
\widehat\rho\circ \tilde{\zeta}_p= {\zeta}_p\circ \rho_2\ , &&
\widehat\rho\circ\tilde{\zeta}_l\circ\sq=\widehat\rho\circ\sq\circ \tilde{\zeta}_l=
{\zeta}_l\circ \rho_2\ .
\end{align*}
Moreover,
\begin{align*}
\left(\widehat\rho\right)^{-1}\left({\zeta}_p\right)(\SO{2}) &\cong \Spin{2} \quad \text{ and } \\
\left(\widehat\rho\right)^{-1}\left({\zeta}_l\right)(\SO{2}) &\cong \{ 1, -1 \} \times \SO{2}.
\end{align*}
\end{lemma}

\begin{proof}
For $\alpha\in \RR$,
\[
(\widehat\rho\circ \tilde{\zeta}_p)\big(S(\alpha)\big)
=\widehat\rho \left( D(\alpha), \begin{pmat} D(-\alpha) & \\ & D(\alpha)\end{pmat}\right)
=\begin{pmat} 1_{\SO{2}} & \\ & D(2\alpha)\end{pmat}
=\zeta_p\big(D(2\alpha)\big)
=(\zeta_p\circ \rho_2)\big(S(\alpha)\big)
\]
and
\[
(\widehat\rho\circ\sq\circ \tilde{\zeta}_l)\big(S(\alpha)\big)
=\widehat\rho\big( 1_{\SO{2}}, \tilde{D}(2\alpha)\big)
=\tilde{D}(2\alpha)
=\zeta_l\big(D(2\alpha)\big)
=(\zeta_l\circ \rho_2)\big(S(\alpha)\big)\ .
\]
For the second claim observe that 
\begin{align*}
\left(\widehat\rho\right)^{-1}\left({\zeta}_p\right)(\SO{2}) = \left\{ \left( D(\alpha), \begin{pmat} D(-\alpha) & \\ & D(\alpha)\end{pmat}\right) \mid \alpha \in \RR \right\} &\cong \Spin{2} , \\
\left(\widehat\rho\right)^{-1}\left({\zeta}_l\right)(\SO{2}) = \left\{ \big( \pm \id , \tilde{D}(\alpha)\big) \mid \alpha \in \RR \right\} &\cong \{ 1, -1 \} \times \SO{2}. \qedhere
\end{align*}
\end{proof}

\begin{consequence} \label{concretedescription}
One has 
\begin{align*}
{\zeta}_p= \eps_{34} && \zeta_l = \eps_{23}\cdot\eps_{14}
\end{align*}
\end{consequence}

\begin{proof}
By Remarks~\ref{6} and \ref{coordinatesrev} and Lemma~\ref{lem:ama-embed-and-rho-commute-B2} this is immediate from Lemma~\ref{concretedescriptiontilde}.
\end{proof}

\begin{proposition} \label{prop:lift-aut-B2} 
Given $\tau\in \Aut\big(\U{2}\big)$, there is a unique $\tilde\tau\in \Aut\big( \SO{2}\times \SU{2}\big)$ such that
\[\widehat\rho\circ \tilde\tau=\tau\circ \widehat\rho\ .\]
\end{proposition}

\begin{proof}
Let $\tau\in \Aut\big(\U{2}\big)$. Consider the characteristic subgroups
\begin{align*}
H_1:=Z\big(\U{2}\big)=\{ z\cdot I_2 \mid z\in \U{1} \cong \SO{2}  \}\ , &&
H_2:=[\U{2},\U{2}]=\SU{2}\ ,
\end{align*}
 of $\U{2}$. 
Note that $H_1$ and $H_2$ are also characteristic in $H:=H_1\times H_2$, as $H_1=Z(H)$ and $H_2=[H,H]$.
Hence $\Aut(H)=\Aut(H_1)\times\Aut(H_2)$.
 
Let $\tau_1:=\tau_{|H_1}\in\Aut(H_1)$ and $\tau_2:=\tau_{|H_2}\in\Aut(H_2)$. Then $\tilde\tau:=(\tau_1,\tau_2)\in\Aut(H)$ satisfies
\begin{align*}
(\widehat\rho\circ \tilde\tau)(z,A)=\widehat\rho\big( \tau_1(z), \tau_2(A) \big)=\tau_1(z)\tau_2(A)=\tau(z)\tau(A)=\tau(zA)=(\tau\circ\widehat\rho)(z,A)
\end{align*}
for all $(z,A)\in \SO{2}\times \SU{2}$.
Let $\psi=(\psi_1,\psi_2)\in \Aut \big( \SO{2}\times \SU{2}\big)$ be such that
$\widehat\rho\circ \psi=\tau\circ \widehat\rho$. Given $z\in \SO{2}$ and $A\in \SU{2}$, one has
\[
  \psi_1(z)
 =\psi_1(z)\psi_2(I_2)
 =(\widehat\rho\circ \psi)(z,I_2)
 =(\tau\circ \widehat\rho)(z,I_2)
 =(\widehat\rho\circ \tilde\tau)(z,I_2)
 =\tau_1(z)\tau_2(I_2)=\tau_1(z)
\]
and
\[
  \psi_2(A)
 =(\widehat\rho\circ \psi)(1_{\SO{2}},A)
 =(\widehat\rho\circ \tilde\tau)(1_{\SO{2}},A)
 =\tau_2(A)\ .
 \qedhere 
\]
\end{proof}

The final result of this section allows us to carry out Strategy~\ref{simplelacing} for edges of type $\mathrm{C}_2$ in Theorem~\ref{thm:Spin(Delta)-covers-Spin(Delta-sl)} below.

\begin{proposition} \label{surjB2}
There is an epimorphism $\tilde{\zeta}:\SO{2}\times \SU{2}\to \Spin{3}$ such that
\begin{align*}
\tilde{\zeta} \circ \tilde{\zeta}_p=\tilde\eps_{12}\circ \inv\ , && \tilde{\zeta}\circ \tilde{\zeta}_l=\tilde\eps_{23}\ .
\end{align*}
\end{proposition}

\begin{proof}
The map
\begin{eqnarray*}
\psi:\Spin{3} & \to & \SU{2} \\
a+be_1e_2+ce_2e_3+de_3e_1 & \mapsto & R_{a+bi+cj+dk} = \begin{pmatrix} a & b & c & d \\
-b & a & -d & c \\
-c & d & a & -b \\
-d & -c & b & a \end{pmatrix}
\end{eqnarray*}
is a group isomorphism by Remarks~\ref{rem:cl3=quaternions}(b) and \ref{remark5}(b) (see also \cite[11.26]{Salzmann:1995}), with the convention that the matrix group $\SU{2}$ acts $\RR$-linearly on $\CC^2 \cong \HH$ with respect to the $\RR$-basis $1$, $i=i1$, $j$, $k=ij$ as in Definition~\ref{rank1embb2} and Notation~\ref{su2acts}. Let
\[\tilde{\zeta}: \SO{2}\times \SU{2}\to \Spin{3} : (z,A)\mapsto \psi^{-1}(A)\ .\]
For $\alpha\in \RR$, one has
\begin{align*}
(\tilde{\zeta}\circ \tilde{\zeta}_p)\big(S(\alpha)\big)
&=
\tilde{\zeta}\left( D(\alpha), \begin{pmat} D(-\alpha) & \\ & D(\alpha)\end{pmat}\right)
= \cos(\alpha)-\sin(\alpha)e_1e_2
=(\tilde\eps_{12}\circ\inv)\big(S(\alpha)\big)
\end{align*}
and
\begin{align*}
(\tilde{\zeta}\circ \tilde{\zeta}_l)\big(S(\alpha)\big)
&= \tilde{\zeta}\big( 1_{\SO{2}}, \tilde{D}(\alpha)\big)
=\zeta\big(\tilde{D}(\alpha)\big)
= \cos(\alpha)+\sin(\alpha)e_2e_3
=\tilde\eps_{23}\big(S(\alpha)\big)\ .
\qedhere
\end{align*}
\end{proof}

\begin{corollary} \label{corsurjB2}\label{corsurjB2b}\label{corsurjB2I}\label{surjB2b}\label{surjB2I}
\begin{enumerate}
\item
There exist epimorphisms $\zeta_1, \zeta_2, \zeta_3 : \mathrm{U}_2(\CC) \to \SO{3}$ such that  \begin{align*}
{\zeta_1} \circ {\zeta}_p=\eps_{12}\circ \inv\ , && {\zeta_1}\circ {\zeta}_l=\eps_{23}, \\
{\zeta_2} \circ {\zeta}_p=\eps_{12} , && {\zeta_2}\circ {\zeta}_l=\eps_{23}, \\
{\zeta_3} \circ {\zeta}_p=\eps_{23} , && {\zeta_3}\circ {\zeta}_l=\eps_{12}.
\end{align*}
\item There exist epimorphisms $\tilde{\zeta}_2, \tilde{\zeta}_3 :\SO{2}\times \SU{2}\to \Spin{3}$ such that
\begin{align*}
\tilde{\zeta}_2 \circ \tilde{\zeta}_p=\tilde\eps_{12} , && \tilde{\zeta}_2\circ \tilde{\zeta}_l=\tilde\eps_{23}, \\
\tilde{\zeta}_3 \circ \tilde{\zeta}_p=\tilde\eps_{23} , && \tilde{\zeta}_3\circ \tilde{\zeta}_l=\tilde\eps_{12}.
\end{align*}
\end{enumerate}
\end{corollary}

\begin{proof}
The kernel of the epimorphism $\SO{2} \times \SU{2} \to \U{2} : (z,A) \mapsto zA$ is equal to $\left\langle \left(D(\pi), \tilde D(\pi)\right) \right\rangle$.
One has $\tilde\zeta\left(D(\pi), \tilde D(\pi)\right) = \psi^{-1}(-I_4) = -1_{\Spin{3}}$, whence $\left\langle \left(D(\pi), \tilde D(\pi)\right) \right\rangle \subseteq \ker(\rho_3 \circ \tilde\zeta)$.
Therefore, the claim concerning $\zeta_1$ follows from Proposition~\ref{surjB2} and the homomorphism theorem of groups.

The claims about $\zeta_2$, $\zeta_3$ then follow from Lemma~\ref{lem:rank1-inv-A2}, resp.\ \ref{10}. A subsequent application of Proposition~\ref{prop:lift-aut-soN} yields the claims about $\tilde\zeta_2$, $\tilde\zeta_3$. 
\end{proof}

\section{Non-spherical diagrams of rank two} \label{sec:rank-2-residues}

In this section we prepare Strategy~\ref{simplelacing} for non-spherical diagrams of rank two. For an introduction to the concept of a Kac--Moody root datum we refer the reader to \cite[Introduction]{Tits:1987}, \cite[7.1.1, p.~172]{Remy:2002}, \cite[Definition~5.1]{Marquis:2013}. For the definition of simple connectedness see \cite[7.1.2]{Remy:2002}. Note that for a given generalized Cartan matrix, up to isomorphism, there exists a uniquely determined simply connected Kac--Moody root datum.

\begin{definition} \label{krs}
Let $r, s \in \NN$ such that $rs \geq 4$ and consider the generalized Cartan matrix of rank two given by
\[ A := (a(i,j))_{i \in \{1, 2\}}
= \begin{pmatrix} 2 & -r \\ -s & 2\end{pmatrix}
\]
and a simply connected Kac--Moody root datum $\mathcal{D} = ( \{1, 2 \}, A, \Lambda, (c_i)_{i \in \{1, 2\}}, (h_i)_{i \in \{1, 2\}})$. 

Let $G:=G(A):=G(\mathcal{D})$ be the corresponding (simply connected) real Kac--Moody group of rank two, let $T_0$ be the fundamental torus of $G$ with respect to the fundamental roots $\alpha_1$, $\alpha_2$,  and let
\[K:=K^{r,s}:=K(A)\]
be the subgroup consisting of the elements fixed by the Cartan--Chevalley involution with respect to $T_0$ of $G$ (its maximal compact subgroup).
\end{definition}

\begin{remark}
Let $G_i \cong \SL{2}$ be the corresponding fundamental subgroups of rank one and define
\[
K_i := K^{r,s}_i := G_i \cap K^{r,s} \cong \SO{2}
\quad\text{ and }\quad
T:=T_0\cap K.
\]
By (KMG3) (see \cite[p.\ 545]{Tits:1987} or, e.g., \cite[p.\ 84]{Marquis:2013}), the torus $T_0$ is generated by $\mu^{h_i}$ for $i=1,2$ and $\mu\in \RR \backslash \{ 0 \}$
arbitrary. The action of the Cartan--Chevalley involution
on the torus is given by $\mu^{h_i} \mapsto (\mu^{-1})^{h_i}=\mu^{-h_i}$. Hence
\[
T = T_0\cap K =
\left\{
\id = 1^{h_1}=1^{h_2},
t_1:=(-1)^{h_1}, t_2:=(-1)^{h_2}, t_1 t_2 = (-1)^{h_1+h_2}
\right\}
\cong \ZZ/2\ZZ \times \ZZ/2\ZZ.
\]
\end{remark}

\begin{lemma} \label{kisfreeamalgam}
$K$ is isomorphic to a free amalgamated product
\[
K_1T *_{T} K_2T.
\]
\end{lemma}
\begin{proof}
The twin building of the Kac--Moody group $G$ is a twin tree (cf.\ \cite{Ronan/Tits:1987}, \cite{Ronan/Tits:1994}). The chambers are the edges; the panels are the sets of edges sharing one vertex and, hence, correspond to vertices. The group $K$ acts edge-transitively and without inversions on each half of the twin tree of $G$ by the Iwasawa decomposition (see, e.g., \cite{Medts/Gramlich/Horn}).
The Cartan--Chevalley involution $\omega$ interchanges the two halves of the twin tree,
mapping edges to opposite edges. Hence the stabilizer in $K$ of
the fundamental edge $c^+$ also stabilizes the opposite edge $c^-=\omega(c^+)$
and, thus, the unique twin apartment spanned by them. It follows that the edge stabilizer is $T$.
Since the panels correspond to the vertices of the tree, the stabilizers of the vertices of
the fundamental edge $c^+$ are equal to $K_1T$ and $K_2T$.
The claim follows from \cite[I, \S 5]{Serre:Trees}.
\end{proof}

\begin{remark} \label{tcentralizes}
Since $K_i \unlhd K_iT$, this free amalgamated product is fully determined by the intersections $K_i \cap T$ and the action of $T$ on each $K_i$.
Note that (KMG3) implies $T \cap K_1 = \{ 1, t_1\}$ and $T \cap K_2 = \{ 1, t_2 \}$.
The action of $T$ on each $K_i$ can be extracted from the action of $T_0$ on each $G_i$, which according to  \cite[(4), p.\ 549]{Tits:1987} (or also \cite[(5.1), p.\ 86]{Marquis:2013}) is given by
\[ t x_i(\lambda) t^{-1} = x_i (t(c_i) \lambda) \]
for $t \in T_0$, $\lambda \in \RR$ and root group functor $x_i$. 
According to \cite[Section~2, p.~544]{Tits:1987} (or also \cite[Definitions 5.1 and 5.5]{Marquis:2013}) one computes for $i,j\in\{1,2\}$ that
\[ t_i(c_j) = (-1)^{h_i}(c_j) = (-1)^{h_i(c_j)} = (-1)^{a(i,j)}. \]

We conclude that $t_i$ acts trivially on $K_j$, if and only if the entry $a(i,j)$ of the generalized Cartan matrix is even; conversely it acts non-trivially (and hence by inversion) if and only if $a(i,j)$ is odd.

In symbols, for 
\[
n(i,j):=
\begin{cases}
0, & \text{if } a(i,j) \text{ is even}, \\
1, & \text{if } a(i,j) \text{ is odd}.
\end{cases}
\]
and $k_j \in K_j$ one has
\begin{eqnarray}
t_i^{-1}k_jt_i & = & k_j^{-2n(i,j)}k_j.  \label{crucialidentity} 
\end{eqnarray}
\end{remark}

\begin{notation}
Let
\[
\theta_1 := \theta_1^{r,s} : \SO{2} \to K_1^{r,s}, \qquad
\theta_2 := \theta_2^{r,s} : \SO{2} \to K_2^{r,s},
\]
be continuous isomorphisms.
\end{notation}

\begin{defn}
Let $r, s \in \NN$ such that $rs \geq 4$, let $A = \begin{pmat} 2 & -r \\ -s & 2\end{pmat}$. Then set
\[
H^{r,s}
:= \begin{cases}
 \SO{2} \times \SO{2} & \text{ if } r\equiv s \equiv 0 \pmod 2\ ,\\
 \SO{3} & \text{ if } r\equiv s \equiv 1 \pmod 2\ ,\\
 \U{2} & \text{ otherwise}.\\
\end{cases}
\]
Also set
\[
\delta_1:= 
\begin{cases}
 \iota_1   & \text{ if } r\equiv s \equiv 0 \pmod 2\ ,\\
 \eps_{12} & \text{ if } r\equiv s \equiv 1 \pmod 2\ ,\\
 \zeta_l   & \text{ if } r\equiv 0, s\equiv 1 \pmod 2,\\
 \zeta_p   & \text{ if } r\equiv 1, s\equiv 0 \pmod 2,
\end{cases}
\quad 
\delta_2:= 
\begin{cases}
 \iota_2   & \text{ if } r\equiv s \equiv 0 \pmod 2
    \qquad \text{(see \ref{notationiota})},\\
 \eps_{23} & \text{ if } r\equiv s \equiv 1 \pmod 2
    \qquad \text{(see \ref{nota:EI-VI-qI})},\\
 \zeta_p   & \text{ if } r\equiv 0, s\equiv 1 \pmod 2
    \quad \text{(see \ref{rank1embb2})},\\
 \zeta_l   & \text{ if } r\equiv 1, s\equiv 0 \pmod 2.
\end{cases}
\]
\end{defn}

\begin{proposition} \label{keyproposition}
Let $r, s \in \NN$ such that $rs \geq 4$, let $A = \begin{pmat} 2 & -r \\ -s & 2\end{pmat}$, let $G(A)$ be the corresponding simply connected real Kac--Moody group, and let $K^{r,s}$ be its maximal compact subgroup. Then
there exists a group epimorphism $\theta : K^{r,s} \to H^{r,s}$ such that
\[\theta \circ \theta_1 = \delta_1 \quad \text{ and } \quad  \theta \circ \theta_2 = \delta_2.\] 
\end{proposition}

\begin{proof}
By Lemma~\ref{kisfreeamalgam} the group $K$ is isomorphic to $K_1T *_{T} K_2T$ with $T = \{ 1, t_1, t_2, t_1t_2\} \cong \ZZ/2\ZZ \times \ZZ/2\ZZ$ and $T_0 \cap K_1 = \{ 1, t_1 \}$, $T_0 \cap K_2 = \{ 1, t_2 \}$; in particular, $K$ is generated by $K_1 \cong \SO{2}$ and $K_2 \cong \SO{2}$. As $K$ is a free amalgamated product it therefore suffices to define $\theta$ on each of the $K_i$ and to verify that the actions of the $t_i$ on the $K_j$ are compatible with the actions of the images of the $t_i$ on the images of the $K_j$. 
Define $\theta$ via
\[\theta_{|K_1} : K_1 \to \delta_1(\SO{2}) : x \mapsto \left(\delta_1 \circ {\theta_1}^{-1}\right)(x)
\quad\text{ and }\quad
\theta_{|K_2} : K_2 \to \delta_2(\SO{2}) : x \mapsto \left(\delta_2 \circ {\theta_2}^{-1}\right)(x).\]
Then this is compatible with the action of $T$. Indeed, using Remark~\ref{tcentralizes}, one observes:
\begin{enumerate}
\item
Since $r\equiv s \equiv 0 \pmod 2$, the elements $t_i$ centralize the groups $K_j$ which is compatible with the fact that $\SO{2} \times \SO{2}$ is an abelian group.
\item
Since $r\equiv s \equiv 1 \pmod 2$, the element $t_1$ inverts the group $K_2$ and the element $t_2$ inverts the group $K_1$ which is compatible with the situation in $\SO{3}$ by Lemma~\ref{lem:rank1-inv-A2}.
\item
Since $r\equiv 0 \pmod 2$, $s\equiv 1 \pmod 2$, the element $t_1$ centralizes $K_2$ and the element $t_2$ inverts the group $K_1$. This is compatible with the following computations (cf.\ Definition~\ref{rank1embb2}):
\[\forall g \in K_2 : \theta(t_1gt_1)
= \begin{pmat} -1 && 0 \\ & -1 && 0 \\ 0 & & -1 &  \\ & 0 & & -1 \end{pmat}
  \begin{pmat} 1 \\ & 1 \\ & & x & y \\ & & -y & x \end{pmat}
  \begin{pmat} -1 && 0 \\ & -1 && 0 \\ 0 & & -1 &  \\ & 0 & & -1 \end{pmat}
= \begin{pmat} 1 \\ & 1 \\ & & x & y \\ & & -y & x \end{pmat}
= \theta(g),\] 
\[\forall g \in K_1 : \theta(t_2gt_2)
= \begin{pmat} 1 \\ & 1 \\ & & -1 & 0 \\ & & 0 & -1 \end{pmat}
  \begin{pmat} x && y \\ & x && y \\ -y & & x &  \\ & -y & & x \end{pmat}
  \begin{pmat} 1 \\ & 1 \\ & & -1 & 0 \\ & & 0 & -1 \end{pmat}
= \begin{pmat} x && -y \\ & x && -y \\ y & & x &  \\ & y & & x \end{pmat}
= \theta(g^{-1}).\]
\item This is dual to (c). \qedhere
\end{enumerate}
\end{proof}

\begin{defn} \label{thetaspin}
Let $r, s \in \NN$ such that $rs \geq 4$, let $A = \begin{pmat} 2 & -r \\ -s & 2\end{pmat}$. Then set
\[
\widetilde{H}^{r,s}
:= \begin{cases}
 \Spin{2} \times \Spin{2} & \text{ if } r\equiv s \equiv 0 \pmod 2\ ,\\
 \Spin{3} & \text{ if } r\equiv s \equiv 1 \pmod 2\ ,\\
 \SO{2} \times \SU{2} & \text{ otherwise}.\\
\end{cases}
\]
Furthermore, set
\[
\tilde\delta_1:=
\begin{cases}
 \tilde\iota_1   & \text{ if } r\equiv s \equiv 0 \pmod 2,\\
 \tilde\eps_{12} & \text{ if } r\equiv s \equiv 1 \pmod 2,\\
 \tilde\zeta_l   & \text{ if } r\equiv 0, s\equiv 1 \pmod 2,\\
 \tilde\zeta_p   & \text{ if } r\equiv 1, s\equiv 0 \pmod 2,\\
\end{cases}
\quad
\tilde\delta_2:=
\begin{cases}
 \tilde\iota_2   & \text{ if } r\equiv s \equiv 0 \pmod 2
    \qquad \text{(see \ref{iotaspin})},\\
 \tilde\eps_{23} & \text{ if } r\equiv s \equiv 1 \pmod 2
    \qquad \text{(see \ref{2})},\\
 \tilde\zeta_p   & \text{ if } r\equiv 0, s\equiv 1 \pmod 2
    \quad \text{(see \ref{zetaspin})},\\
 \tilde\zeta_l   & \text{ if } r\equiv 1, s\equiv 0 \pmod 2.\\
\end{cases}
\]
The central extension \[\bar\rho : \wt H^{r,s} \to H^{r,s}\] satisfies 
\[\bar\rho =
\begin{cases}
\rho_2 \times \rho_2 & \text{ if } r\equiv s \equiv 0 \pmod 2
    \quad \text{(see \ref{iotaspin})}, \\
\rho_3 & \text{ if } r\equiv s \equiv 1 \pmod 2
    \quad \text{(see \ref{rho})},\\
\widehat \rho & \text{ otherwise }
    \qquad\qquad\qquad\text{(see \ref{zetaspin})}.
\end{cases}
\]
Let $K^{r,s}=K(A) = K_1T *_{T} K_2T$ as in Lemma~\ref{kisfreeamalgam} and let $t_1 \in K_1 \cap T$, $t_2 \in K_2 \cap T$ as in Remark~\ref{tcentralizes}. 

\medskip
Define
\begin{align*}
u_i &:= \theta(t_i) \quad\quad\quad\quad \text{(see \ref{keyproposition})}, \\
U &:= \langle u_1, u_2\rangle \cong \ZZ/2\ZZ \times \ZZ/2\ZZ.
\end{align*}
Furthermore, define
\begin{align*}
\wt U &:= \bar\rho^{\,-1}(U), \\
\wt K_i := \wt K_i^{r,s} &:= \bar\rho^{\,-1}\big(\theta(K_i)\big), \\
\intertext{and the \Defn{spin extension} }
\wt K := \wt K^{r,s} := \wt K(A) &:= \wt K_1\wt U *_{\wt U} \wt K_2\wt U
\quad\text{ of }\quad
K=K^{r,s}=K(A),
\end{align*}
let \[\hat{\hat\rho} : \wt K_1\wt U *_{\wt U} \wt K_2\wt U \to K_1T *_T K_2T\] be the epimorphism induced by $\bar\rho_{|\wt K_1}$, $\bar\rho_{|\wt K_2}$ and let \[\tilde\theta_1 : \Spin{2} \to \widetilde K_1 \quad\quad \text{and} \quad\quad \tilde\theta_2 : \Spin{2} \to \widetilde K_2\] be continuous monomorphisms such that the following diagrams commute for $i=1,2$:
\[
\xymatrix{
\Spin{2} \ar[d] \ar[rr]^{\tilde\theta_i} && \wt K_i \ar[d]_{\bar\rho_{|\wt K_i}} \\
\SO{2} \ar[rr]^{\theta_i} && K_i
}
\]
\end{defn}

\begin{remark} \label{dichotomy}
One has $\wt K_1 \cong \Spin{2}$ unless $r\equiv 0, s\equiv 1 \pmod 2$ and $\wt K_2 \cong \Spin{2}$ unless $r\equiv 1, s\equiv 0 \pmod 2$, in which case the respective group is isomorphic to $\{ 1, -1 \} \times \SO{2}$ (cf.\ Lemma~\ref{lem:ama-embed-and-rho-commute-B2}). Hence $\tilde\theta_1$ actually is a (continuous) isomorphism unless $r\equiv 0, s\equiv 1 \pmod 2$, in which case it is a (continuous) isomorphism onto the unique connected subgroup of index two of $\wt K_1$.

For $i=1$ the map on the left hand side of the above commutative diagram is
\[\begin{array}{rclll}
\Spin{2} \to \SO{2} &: & S(\alpha) \mapsto D(\alpha) & \text{if } r\equiv 0, s\equiv 1 \pmod 2 \quad\quad & \text{(see \ref{nota:Dalpha-Salpha})} \\
\Spin{2} \to \SO{2} &:& x \mapsto \rho_2(x) & \text{otherwise}. 
\end{array}\]
The dual statement holds for $\tilde\theta_2$.   

In particular, for $i=2$ the map on the left hand side of the above commutative diagram is
\[\begin{array}{rclll}
\Spin{2} \to \SO{2} &: &S(\alpha) \mapsto D(\alpha) & \text{if } r\equiv 1, s\equiv 0 \pmod 2 \quad\quad & \text{(see \ref{nota:Dalpha-Salpha})} \\
\Spin{2} \to \SO{2} &: & x \mapsto \rho_2(x) & \text{otherwise}. 
\end{array}\] 
Define
\begin{align*}
\tilde t_1 & := \begin{cases}
    \tilde\theta_1(S(\pi)), & \text{if } r\equiv 0, s\equiv 1 \pmod 2 \\
    \tilde\theta_1(S(\frac{\pi}{2})), & \text{otherwise}.
\end{cases}
\\
\tilde t_2 & := \begin{cases}
    \tilde\theta_2(S({\pi})), & \text{if } r\equiv 1, s\equiv 0 \pmod 2 \\
    \tilde\theta_2(S(\frac{\pi}{2})), & \text{otherwise}.
\end{cases}
\end{align*}
\end{remark}

The following is true by construction:

\begin{proposition}  \label{keycorollary}
Let $r, s \in \NN$ such that $rs \geq 4$, let $A = \begin{pmat} 2 & -r \\ -s & 2\end{pmat}$,
let $G(A)$ be the corresponding simply connected real Kac--Moody group, and let $K^{r,s}$ be its maximal compact subgroup,
and let $\wt{K}^{r,s}$ be its spin extension.
Then there exists a group epimorphism $\tilde\theta : \wt{K}^{r,s} \to \widetilde{H}^{r,s}$ such that
\[\tilde\theta \circ \tilde\theta_1 = \tilde\delta_1 \quad \text{ and } \quad  \tilde\theta \circ \tilde\theta_2 = \tilde\delta_2.\]
Moreover, the following diagram commutes:
\[
\xymatrix{
\Spin{2} \ar[dr]^{\tilde\theta_i} \ar@/^/[drrr]^{\tilde\delta_i} \ar[ddd]_{\rho_2 \text{ or } S(\alpha)\mapsto D(\alpha)} \\
&\wt K^{r,s} \ar[rr]^{\tilde\theta} \ar[d]_{\hat{\hat\rho}} && \widetilde{H}^{r,s} \ar[d]^{\bar\rho} \\
&K^{r,s} \ar[rr]^\theta && H^{r,s} \\
\SO{2} \ar[ur]_{\theta_i} \ar@/_/[urrr]_{\delta_i}
}
\]
where the epimorphism on the left hand side is one of $\rho_2$ or $S(\alpha) \mapsto D(\alpha)$ as described in Remark~\ref{dichotomy}.

Furthermore, for $\{ i, j \} = \{ 1, 2 \}$ and $\tilde k_j \in \wt K_j$ one has 
\begin{eqnarray}
\tilde t_i^{-1} \tilde k_j \tilde t_i & = & \tilde k_j^{-2n(i,j)}\tilde k_j.  \label{crucialidentity2} 
\end{eqnarray}
\end{proposition}

Identity (\ref{crucialidentity2}) follows from identity (\ref{crucialidentity}) in Remark~\ref{tcentralizes}

\begin{remark} \label{longvsshort}
The Cartan matrix of type $\mathrm{C}_2$ over $\{ 1, 2 \}$ with short root $\alpha_1$ and long root $\alpha_2$ (i.e., $2\to 1$; see Remark~\ref{longandshort}) is \[\begin{pmatrix} 2 & -2 \\ -1 & 2 \end{pmatrix},\]
cf.\ \cite[p.~44]{Carter:1989}.
The group $H^{r,s}$ from Proposition~\ref{keycorollary} is of type $\mathrm{C}_2$ with short root $\alpha_1$ and long root $\alpha_2$ if and only if $r$ is even and $s$ is odd. We conclude that for $i=1$ the map on the left hand side of the commutative diagram is \[\Spin{2} \to \SO{2} : S(\alpha) \mapsto D(\alpha)\] which is in accordance with Notation~\ref{zetaspin} and Lemma~\ref{lem:ama-embed-and-rho-commute-B2}.
In other words, in this example the rank one group corresponding to the long root is doubly covered by its spin cover and the rank one group corresponding to the short root is singly covered by its spin cover.

The direction introduced for edges labelled $\infty$ in Notation~\ref{augmented} was chosen to fit this observation: the arrow points away from the doubly covered vertex of the diagram towards the singly covered vertex of the diagram. 

We point out that \cite[Plate~III, p.~269]{Bourbaki:Lie4-6} incorrectly gives the transpose of the above matrix as the Cartan matrix of type $\mathrm{C}_2$ with short root $\alpha_1$ and long root $\alpha_2$.
\end{remark}

\begin{remark}\label{allodd}
If both $r$, $s$ are odd, then in $\wt K = \wt K_1 \wt U *_{\wt U} \wt K_2 \wt U$ one has \[\wt U \cong Q_8.\]
In other words, the element $S(\pi)$ of the group $\wt K_1 \cong \Spin{2}$ is identical to the element $S(\pi)$ of the group $\wt K_2 \cong \Spin{2}$, like for the groups $\tilde\eps_{12}(\Spin{2})$ and $\tilde\eps_{23}(\Spin{2})$ in $\Spin{3}$. The same is true for the groups $\tilde\eta_p(\Spin{2})$ and $\tilde\eta_l(\Spin{2})$ in $\Spin{4}$ by Proposition~\ref{surjG2I}.
\end{remark}

\begin{remark} \label{otherU}
If both $r$ and $s$ are even, then in $\wt K = \wt K_1 \wt U *_{\wt U} \wt K_2 \wt U$ one has \[\wt U \cong \ZZ/4\ZZ \times \ZZ/4\ZZ,\] if one of $r$, $s$ is even and the other odd, then one has \[\wt U \cong \ZZ/4\ZZ \times \ZZ/2\ZZ.\] 
\end{remark}

\begin{lemma} \label{lem:automorphismsarenice}
Each automorphism $\alpha$ of $K = K_1U *_U K_2U$ and each automorphism $\wt \alpha$ of $\wt K= \wt K_1 \wt U *_{\wt U} \wt K_2 \wt U$ induces an automorphism of the Bruhat--Tits tree of $G$. The set $\{ \alpha(K_1U), \alpha(K_2U) \}$ is $K$-conjugate to $\{ K_1U, K_2U \}$, the set $\{ \wt\alpha(\wt K_1\wt U), \wt\alpha(\wt K_2 \wt U) \}$ is $\wt K$-conjugate to $\{ \wt K_1 \wt U, \wt K_2 \wt U\}$.  
\end{lemma}

\begin{proof}
$K_1U$ and $K_2U$ are indecomposable as amalgamated products and do not admit $\ZZ$ as a quotient. 
Moreover, $U$ is finite.
Therefore by \cite[Theorem 6]{Karrass/Solitar:1970}, each $\alpha(K_iU)$ is conjugate to $K_1U$ or to $K_2U$.
Hence $\alpha(K_1U)$, $\alpha(K_2U)$ each stabilize a vertex of the Bruhat--Tits tree $X$ of $G$.
Since $\alpha(K_1U)\cap\alpha(K_2U)=\alpha(U)$, these two vertices are adjacent and $\alpha$ acts on $X$.
Since $K$ acts edge-transitively on $X$ the set $\{ \alpha(K_1U), \alpha(K_2U) \}$ is conjugate to $\{ K_1U, K_2U \}$.
The same argument works for $\wt \alpha$ and $\wt K_i \wt U$.
\end{proof}

\begin{proposition} \label{prop:controlautoinfty}
For each automorphism $\alpha$ of $K$ there exists a unique automorphism $\wt \alpha$ of $\wt K$ such that
\[ \hat{\hat \rho} \circ \wt \alpha = \alpha \circ \hat{\hat \rho}.\]
\end{proposition}

\begin{proof}
Since $K$ modulo its centre is isomorphic to $\wt K$ modulo its centre, the claim holds for inner automorphisms. It therefore suffices to study the outer automorphisms groups $\Out(\cdot)=\Aut(\cdot)/\Inn(\cdot)$. By Lemma~\ref{lem:automorphismsarenice} it therefore suffices to investigate automorphisms that preserve the sets $K_1U \cup K_2U$, resp.\ $\wt K_1\wt U \cup \wt K_2 \wt U$. The claim follows from Proposition~\ref{prop:lift-aut-so2}.
\end{proof}

\part{Arbitrary diagrams}

\section{Admissible colourings} \label{sec:adm-amalgams}

In this section we extend the classification results for $\SO{2}$- and $\Spin{2}$-amalgams from Sections \ref{sec:so2amalgams} and \ref{sec:spin2amalgams} to arbitrary diagrams. 

\begin{remark}
Throughout this section,
let $A=(a(i,j))_{i,j\in I}$ be a generalized Cartan matrix over the index set $I$ and let $\Pi$ be the augmented Dynkin diagram with vertex set $V$ induced by $A$ (see Notation~\ref{augmented})
with respect to a labelling $\sigma:I\to V$.
\end{remark}

We have seen in Proposition~\ref{keycorollary} that --- given two vertices $i^\sigma$, $j^\sigma$ of $\Pi$ --- some subtleties related to single versus double covers of the circle group arise in the theory of spin covers of rank two depending on the parities of $a(i,j)$ and $a(j,i)$. To this end we develop a theory of admissible colourings that will help us distinguish the respective vertices of the diagram from one another.

\begin{defn}\label{def:nij-mij}\label{adtypedef}
For $i\neq j\in I$, define
\[
n(i,j):=
\begin{cases}
0, & \text{if } a(i,j) \text{ is even}, \\
1, & \text{if } a(i,j) \text{ is odd}.
\end{cases}
\]
Let $\Pi^\mathrm{adm}$ be the graph on the vertex set $V$ with edge set
\[
\big\{ \{i,j\}^\sigma \in V \times V \mid i \neq j \in I, n(i,j)=n(j,i)=1 \big\}. 
\]
An \Defn{admissible colouring}
of $\Pi$ is a map $\kappa:V\to\{1,2\}$ such that
\begin{enumerate}
\item $\kappa(i^\sigma)=1$ whenever there exists $j \in I \backslash \{ i \}$ with $n(i,j)=0$ and $n(j,i)=1$,
\item the restriction of $\kappa$ to any connected component of the graph $\Pi^\mathrm{adm}$ is a constant map.
\end{enumerate}
\end{defn}

Let $c(\Pi,\kappa)$ be the number of connected components of $\Pi^\mathrm{adm}$ on which $\kappa$ takes the value $2$.
 
An admissible colouring $\kappa$ is \Defn{elementary} if $c(\Pi,\kappa)=1$, i.e., if there exists a unique connected component of the graph $\Pi^\mathrm{adm}$ on which $\kappa$ takes the value $2$.
The admissible colouring $\kappa \equiv 1$ (i.e., the one with $c(\Pi,\kappa)=0$) is called the \Defn{trivial colouring}. An admissible colouring $\kappa$ is \Defn{proper} if every
connected component of $\Pi$ contains a vertex $v$ with $\kappa(v)=2$.

\begin{remark} \label{obstruction}
An elementary admissible colouring is given by assigning the value $2$ to exactly one connected component of $\Pi^\mathrm{adm}$, while all other connected components take value $1$. Therefore, in order to construct all elementary admissible colourings, it suffices to decide which connected components may be assigned the value $2$.

Let $k^\sigma$ be a vertex of $\Pi$. The only obstruction to being able to assign the value $2$ to $k^\sigma$ is being contained in the connected component of $\Pi^{\mathrm{adm}}$ of a vertex that necessarily has to be assigned value $1$. That is, there exists a vertex $i^\sigma$ in the same connected component of $\Pi^{\mathrm{adm}}$ as $k^\sigma$ such that there is a vertex $j^\sigma$ with $n(i,j)=0$ and $n(j,i)=1$.     
\end{remark}

\begin{lemma} \label{lem:max-adm-colouring}
Consider the partial order $\preceq$ on the set of admissible colouring of $\Pi$
where $\kappa\preceq\kappa'$ if $\kappa(v)\leq\kappa'(v)$ for all $v\in V$. 
Then there is a unique
maximal admissible colouring $\kappa_\mathrm{max}$ of $\Pi$ with respect to $\preceq$.
\end{lemma}

\begin{proof}
Suppose $X$ is a set of admissible colourings of $\Pi$. Then one readily
checks that $\kappa:V\to\{1,2\}:v\mapsto\max\{\kappa'(v) \mid \kappa'\in
X\}$ is again an admissible colouring, satisfying $\kappa'\preceq\kappa$ for
all $\kappa'\in X$. Since $V$ is finite, so is its set of admissible
colourings, i.e., there exists a maximal element.
\end{proof}

\begin{remark} \label{combinatoricscolouring}
Given two admissible colourings $\kappa_1$, $\kappa_2$ of $\Pi$, define the \Defn{sum} of $\kappa_1$ and $\kappa_2$ as  \[\kappa_1+\kappa_2 : V\to\{1,2\}:v\mapsto\max\{\kappa_i(v) \mid i = 1, 2 \}.\] 
By the preceding discussion, this is again an admissible colouring.
Each non-trivial admissible colouring of a Dynkin diagram is the sum of (finitely many pairwise distinct) elementary colourings. The maximal admissible colouring $\kappa_{\mathrm{max}}$ is the sum of all elementary colourings. 
\end{remark}

\begin{notation}
Throughout this section,
let $\kappa:V\to\{1,2\}$ be an admissible colouring of $\Pi$. For $i\neq j\in I$, set 
\[
 v_{ij}:=v\ :\Leftrightarrow\ \{i,j\}^\sigma\in E_v(\Pi)\ ,
 \qquad
 \kappa_{ij}:=\frac{\kappa(i^\sigma)+\kappa(j^\sigma)}{2}\in \{1,1.5,2\}\ .
\]
Note that $v_{ij}$ is well-defined since the sets $E_v(\Pi)$ for $v\in\{0,1,2,3,\infty\}$ form
a partition of $\binom{V}{2}$.
Moreover, note that $\kappa_{ij}=1.5$ implies $v_{ij}\in\{0,2,\infty\}$ by Definition~\ref{adtypedef}(b). Furthermore,
$\kappa_{ij}=2$ implies $v_{ij}\in\{0,1,3,\infty\}$ by Definition~\ref{adtypedef}(a).
\end{notation}

\begin{remark}\label{orientation}
The labelling $\sigma$ and the colouring $\kappa$ allow one to extend the direction relation on the augmented Dynkin diagram $\Pi$ to a direction relation between all pairs of vertices $u\neq v\in V$ (cf.\ Remark~\ref{longvsshort}): 
\begin{itemize}
\item If $\{u,v\}$ is a directed edge, retain the direction $u \to v$.
\item If $\kappa(u)\neq\kappa(v)$,
then set $u\to v$ whenever $\kappa(u)>\kappa(v)$.
\item For the remaining cases, use the labelling to introduce a direction as follows:
\[
\forall\ \{i,j\}^\sigma \text{ non-directed with } \kappa(i^\sigma)=\kappa(j^\sigma):
    \qquad i^\sigma\to j^\sigma\ :\Leftrightarrow\ i>j\ .
\]
\end{itemize}
This direction relation is called the \Defn{orientation} of $\Pi$ induced by the labelling $\sigma$ and the colouring $\kappa$.
\end{remark}

In the following we will quite freely use the notation introduced in the sections in which we studied the rank two situation, such as Sections \ref{sec:spin-pin}, \ref{sec:7}, \ref{sec:g2}, \ref{sec:bc2}, \ref{sec:rank-2-residues}.

\begin{notation}\label{31}
For $i\neq j\in I$, let $r:=a(i,j)$ and $s:=a(j,i)$.
Define
\[
    G^{r,s}:=\begin{cases}
    \SO{2}\times\SO{2}, & \text{if }rs=0, \\
    \SO{3},             & \text{if }rs=1, \\
    \U{2},              & \text{if }rs=2, \\
    \SO{4},             & \text{if }rs=3, \\
    K^{r,s} = K_1T *_{T} K_2T & \text{if }rs\geq 4, 
    \end{cases}
\]
where in the final case --- as discussed in Remark~\ref{tcentralizes} --- the action of $T$ on $K_i$ depends on the parities of the entries of the generalized Cartan matrix.
Furthermore, define homomorphisms from $\SO{2}$ into $G^{r,s}$ as follows:
\begin{align*}
    \eps_1^{r,s}:=\begin{cases}
    \iota_1, & \text{if }rs=0, \\
    \eps_{12}, & \text{if }rs=1, \\
    \zeta_p,   & \text{if }rs=2, \\
    \eta_p,    & \text{if }rs=3, \\
    \theta_1^{r,s},  & \text{if }rs\geq 4,
    \end{cases}
    \qquad\text{and}\qquad
    \eps_2^{r,s}:=\begin{cases}
    \iota_2, & \text{if }rs=0, \\
    \eps_{23}, & \text{if }rs=1, \\
    \zeta_l,   & \text{if }rs=2, \\
    \eta_l,    & \text{if }rs=3, \\
    \theta_2^{r,s},  & \text{if }rs\geq 4.
    \end{cases}
\end{align*}
Next, define various covering groups of these:
\[
    \widetilde{G}^{r,s,k}:=\begin{cases}
    G^{r,s},                 & \text{if }k=1, \\
    \Spin{2}\times\Spin{2},  & \text{if }rs=0 \text{ and } k>1, \\
    \Spin{3},                & \text{if }rs=1\text{ and } k=2, \\
    \SO{2}\times\SU{2},      & \text{if }rs=2\text{ and } k=1.5, \\
    \Spin{4},                & \text{if }rs=3\text{ and } k=2, \\
    \wt K^{r,s} = \widetilde K_1 \wt{U} *_{\wt{U}} \widetilde K_1 \wt{U} & \text{if }rs\geq 4\text{ and } k>1.
    \end{cases}
\]
Recall from Notation~\ref{nota:Dalpha-Salpha} the isomorphism
\[\psi:\SO{2}\to \Spin{2} : D(\alpha)\mapsto S(\alpha).\]
Using this, define the following homomorphisms from $\Spin{2}$ into $\widetilde{G}^{(r,s,k)}$:
\[
  \tilde\eps_1^{(r,s,k)}:=
  \begin{cases}
    \eps_1^{r,s}\circ\psi^{-1}
                     & \text{if }k=1 \\
    \tilde\iota_1,   & \text{if }rs=0\text{ and } k>1, \\
    \tilde\eps_{12}, & \text{if }rs=1\text{ and } k=2, \\
    \tilde\zeta_p,   & \text{if }rs=2\text{ and } k=1.5, \\
    \tilde\eta_p,    & \text{if }rs=3\text{ and } k=2, \\
    \tilde\theta_1^{r,s},  & \text{if }rs\geq 4\text{ and } k>1,
  \end{cases}
\quad\text{and}\quad
  \tilde\eps_2^{r,s,k}:=
  \begin{cases}
    \eps_2^{r,s}\circ\psi^{-1}
                     & \text{if }k=1 \\
    \tilde\iota_2,   & \text{if }rs=0\text{ and } k>1, \\
    \tilde\eps_{23}, & \text{if }rs=1\text{ and } k=2, \\
    \tilde\zeta_l,   & \text{if }rs=2\text{ and } k=1.5, \\
    \tilde\eta_l,    & \text{if }rs=3\text{ and } k=2, \\
    \tilde\theta_2^{r,s},  & \text{if }rs\geq 4\text{ and } k>1.
  \end{cases}
\]
Finally, define epimorphisms \[\rho^{r,s,k} : \widetilde{G}^{r,s,k} \to G^{r,s}\] via
\begin{align*}
    \rho^{r,s,1}&:=\id,
&
    \rho^{r,s,1.5}&:=\begin{cases}
    \rho_2\times\psi^{-1},       & \text{if }rs=0\text{ and } i^\sigma\to j^\sigma, \\
    \psi^{-1} \times \rho_2,       & \text{if }rs=0\text{ and } i^\sigma\leftarrow j^\sigma, \\
    \widehat\rho, & \text{if }rs=2, \\
    \hat{\hat\rho}, & \text{if }rs\geq 4,
    \end{cases}
&
    \rho^{r,s,2}&:=\begin{cases}
    \rho_2\times\rho_2,       & \text{if }rs=0, \\
    \rho_3,       & \text{if }rs=1, \\
    \rho_4,       & \text{if }rs=3, \\
    \hat{\hat\rho}, & \text{if }rs\geq 4.
    \end{cases}
\end{align*}
\end{notation}

\begin{definition}\label{lotsofdef}
\begin{enumerate}
\item
An \Defn{$\SO{2}$-amalgam with respect to $\Pi$ and $\sigma$} is an amalgam
\[\AAA=\{ G_{ij}, \phi_{ij}^i, \mid i\neq j\in I \}\]
such that for all $i\neq j\in I$, we have
\[
    G_{ij}=G^{a(i,j),a(j,i)}
    \quad\text{ and }\quad
    \phi_{ij}^i\big(\SO{2}\big)=
    \begin{cases}
    \eps_1^{a(i,j),a(j,i)}\big(\SO{2}\big), & \text{if }i^\sigma \to j^\sigma, \\
    \eps_2^{a(i,j),a(j,i)}\big(\SO{2}\big), & \text{if }i^\sigma\leftarrow j^\sigma.
    \end{cases}
\]
\item
The \Defn{standard $\SO{2}$-amalgam with respect to $\Pi$ and $\sigma$}
is the (continuous) $\SO{2}$-amalgam
\[\AAA\big(\Pi,\sigma,\SO{2}\big):=\{ G_{ij}, \phi_{ij}^i, \mid i\neq j\in I \}\]
with respect to $\Pi$ and $\sigma$
such that for all $i\neq j\in I$, we have
\[
    G_{ij}=G^{a(i,j),a(j,i))}
    \quad\text{ and }\quad
    \phi_{ij}^i=\begin{cases}
    \eps_1^{a(i,j),a(j,i)}, & \text{if }i^\sigma \to j^\sigma, \\
    \eps_2^{a(i,j),a(j,i)}, & \text{if }i^\sigma\leftarrow j^\sigma.
    \end{cases}
\]
\item
A \Defn{$\Spin{2}$-amalgam with respect to $\Pi$, $\sigma$ and $\kappa$} is an amalgam
\[\AAA=\{ G_{ij}, \phi_{ij}^i, \mid i\neq j\in I \}\]
such that for all $i\neq j\in I$, we have
\[
    G_{ij}=\widetilde{G}^{a(i,j),a(j,i),\kappa_{ij}}
    \quad\text{ and }\quad
    \phi_{ij}^i\big(\Spin{2}\big)=
    \begin{cases}
    \tilde\eps_1^{a(i,j),a(j,i),\kappa_{ij}}\big(\Spin{2}\big), & \text{if }i^\sigma \to j^\sigma, \\
    \tilde\eps_2^{a(i,j),a(j,i),\kappa_{ij}}\big(\Spin{2}\big), & \text{if }i^\sigma\leftarrow j^\sigma.
    \end{cases}
\]
\item
The \Defn{standard $\Spin{2}$-amalgam with respect to $\Pi$, $\sigma$ and $\kappa$}
is the (continuous) $\Spin{2}$-amalgam
\[\AAA\big(\Pi,\sigma,\kappa,\Spin{2}\big):=\{ G_{ij}, \phi_{ij}^i, \mid i\neq j\in I \}\]
with respect to $\Pi$ and $\sigma$
such that for all $i\neq j\in I$, we have
\[
    G_{ij}=\widetilde{G}^{a(i,j),a(j,i),\kappa_{ij}}
    \quad\text{ and }\quad
    \phi_{ij}^i=\begin{cases}
    \tilde\eps_1^{a(i,j),a(j,i),\kappa_{ij}}, & \text{if }i^\sigma \to j^\sigma, \\
    \tilde\eps_2^{a(i,j),a(j,i),\kappa_{ij}}, & \text{if }i^\sigma\leftarrow j^\sigma.
    \end{cases}
\]
\item
Let
$\AAA=\{ G_{ij}, \phi_{ij}^i \mid i\neq j\in I\}$ be an $\SO{2}$-amalgam
with respect to $\Pi$ and $\sigma$. Given $i\neq j\in I$,
there is $\tau_{ij}^i\in \Aut(\SO{2})$ such that
\begin{align*}
\phi_{ij}^i=\begin{cases}
\eps_1^{a(i,j),a(j,i),\kappa_{ij}}\circ \tau_{ij}^i,&\text{if }i^\sigma \to j^\sigma, \\
\eps_2^{a(i,j),a(j,i),\kappa_{ij}}\circ \tau_{ij}^i,&\text{if }i^\sigma\leftarrow j^\sigma.
\end{cases}
\end{align*}
Define $\tilde\tau_{ij}^i\in\Aut(\Spin{2})$ as in Lemma~\ref{prop:lift-aut-so2}
and set
\[
    \widetilde{G}_{ij}:=\widetilde{G}^{a(i,j),a(j,i),\kappa_{ij}}
    \quad\text{ and }\quad
    \tilde\phi_{ij}^i:=
    \begin{cases}
    \tilde\eps_1^{a(i,j),a(j,i),\kappa_{ij}}\circ \tilde\tau_{ij}^i,&\text{if }i^\sigma \to j^\sigma, \\
    \tilde\eps_2^{a(i,j),a(j,i),\kappa_{ij}}\circ \tilde\tau_{ij}^i,&\text{if }i^\sigma\leftarrow j^\sigma.
    \end{cases}
\]
Then
$\wAAA:=\{ \widetilde{G}_{ij}, \tilde\phi_{ij}^i \mid i\neq j\in I\}$
is the \Defn{induced $\Spin{2}$-amalgam with respect to $\Pi$, $\sigma$ and $\kappa$}.
\end{enumerate}
\end{definition}

\begin{remark} \label{rem:spin-ama-to-so-ama-and-back-2}
In analogy to Remark~\ref{rem:spin-ama-to-so-ama-and-back}, the construction in Definition~\ref{lotsofdef}(e) is symmetric and can be applied backwards: Starting with a $\Spin{2}$-amalgam
$\hat\AAA$, one can construct an $\SO{2}$-amalgam $\AAA$ such that $\hat\AAA=\wAAA$. As before, denote the corresponding epimorphism of amalgams by \[\pi_{\Pi,\sigma,\kappa} = \{ \id, \rho^i, \rho_{ij} \} : \AAA\big(\Pi,\sigma,\kappa,\Spin{2}\big) \to \AAA\big(\Pi,\sigma,\SO{2}\big),\] where $\rho_{ij} = \rho^{a(i,j),a(j,i),\kappa_{ij}}$ and $\rho^i = \psi^{-1}$ for all $i \in I$ with $\kappa(i^\sigma)=1$ and $\rho^i = \rho_2$ for all $i \in I$ with  $\kappa(i^\sigma)=2$. 
\end{remark}

\begin{proposition}\label{prop:lift-ama-iso-adm}
Let
$\AAA_1$ and $\AAA_2$ be $\SO{2}$-amalgams with respect to $\Pi$ and
$\sigma$, and let $\alpha=\{ \pi, \alpha_{ij}\mid i\neq j\in I\}:\AAA_1\to \AAA_2$
be a colouring-preserving isomorphism of amalgams, i.e., for all $i \in I$ one has $\kappa(i^\sigma) = \kappa(i^{\pi\sigma})$. Then there is a unique isomorphism
$\tilde\alpha=\{ \pi, \tilde\alpha_{ij}\mid i\neq j\in I\}:\wAAA_1\to \wAAA_2$
such that for all $i \neq j \in I$, one has
\[\rho_{\pi(i)\pi(j)}\circ \tilde\alpha_{ij}=\alpha_{ij}\circ \rho_{ij}.\]
\end{proposition}

\begin{proof}
This result generalizes Proposition~\ref{prop:lift-ama-iso-sl}. Its proof relies on the references
\begin{enumerate}
\item Lemma~\ref{lem:ama-embed-and-rho-commute-An},
\item Proposition~\ref{prop:lift-aut-soN}, resp.\ Corollary~\ref{prop:lift-aut-so2xso2}.
\end{enumerate}
We indicate how one may adapt the proof of \ref{prop:lift-ama-iso-sl} in order to deal with the general situation.
The cases $(k_{ij},v_{ij})=(2,0)$ and $(k_{ij},v_{ij})=(2,1)$ are covered by Proposition~\ref{prop:lift-ama-iso-sl}.
For $(k_{ij},v_{ij})=(2,3)$ the proof is virtually identical to the proof for $(k_{ij},v_{ij})=(2,1)$,
except that one uses $\rho_4$ instead of $\rho_3$, and
Lemma~\ref{lem:ama-embed-and-rho-commute-G2}
instead of Lemma~\ref{lem:ama-embed-and-rho-commute-An}.
Similarly, for $(k_{ij},v_{ij})=(2,\infty)$ we use $\hat{\hat\rho}$ instead of $\rho_3$, Proposition~\ref{keycorollary} instead of Lemma~\ref{lem:ama-embed-and-rho-commute-An} and Proposition~\ref{prop:controlautoinfty} instead of Proposition~\ref{prop:lift-aut-soN}.
For $k_{ij}=1$ there is nothing to show.

The case $(k_{ij},v_{ij})=(1.5,0)$ is very similar to $(k_{ij},v_{ij})=(2,0)$ (Case II of the proof of Proposition~\ref{prop:lift-ama-iso-sl}) and uses straightforward adaptions of Lemma~\ref{lem:ama-embed-and-rho-commute-An} and Corollary~\ref{prop:lift-aut-so2xso2}.

Finally, the cases $(k_{ij},v_{ij})=(1.5,2)$ and $(k_{ij},v_{ij})=(1.5,\infty)$ are again similar to the case $(k_{ij},v_{ij})=(2,1)$ (Case I of the proof of Proposition~\ref{prop:lift-ama-iso-sl}):
replace $\rho_3$ by $\widehat\rho$, resp.\ $\hat{\hat\rho}$, Proposition~\ref{prop:lift-aut-soN}
by Proposition~\ref{prop:lift-aut-B2}, resp.\ Proposition~\ref{prop:controlautoinfty} and Lemma~\ref{lem:ama-embed-and-rho-commute-An} by Proposition~\ref{surjB2}, resp.\ Proposition~\ref{keycorollary}.
\end{proof}

\begin{proposition} \label{prop:adm-ama-labelling-irrelevant}
Let $\sigma_1$, $\sigma_2$ be two labellings of $\Pi$.
Then the following hold:
\begin{enumerate}
\item
$\AAA\big(\Pi,\sigma_1,\SO{2}\big)\cong \AAA\big(\Pi,\sigma_2,\SO{2}\big)$.
\item
$\AAA\big(\Pi,\sigma_1,\kappa,\Spin{2}\big)\cong \AAA\big(\Pi,\sigma_2,\kappa,\Spin{2}\big)$.
\end{enumerate}
\end{proposition}

\begin{proof}
Assertion (a) follows by generalizing the proof of Consequence~\ref{14},
and setting $\alpha_{ij}:=\id_{G_{ij}}$ whenever
$a(i,j)$ and $a(j,i)$ have different parities or $\kappa(i^\sigma) \neq \kappa(j^\sigma)$; otherwise, set $\alpha_{ij}:=\id_{G_{ij}}$, if $\pi$ preserves the order relation between $i$ and $j$, and define $\alpha_{ij}$ via $K_1 \to K_2 : D(\phi) \mapsto D(\phi)$ and $K_2 \to K_1 : D(\phi) \mapsto D(\phi)$ in analogy to Lemma~\ref{10}, if $\pi$ switches the order relation between $i$ and $j$.
Observe that the permutation $\pi \in \Sym(I)$ with $\pi\sigma_1=\sigma_2$ preserves $\kappa$, as for $i_1, i_2 \in I$ with $i_1^{\sigma_1}=i_2^{\sigma_2}$ certainly $\kappa(i_1^{\sigma_1})=\kappa(i_2^{\sigma_2})$.
Therefore, since $\AAA\big(\Pi,\sigma_i,\SO{2}\big)$ induces $\AAA\big(\Pi,\sigma_i,\kappa,\Spin{2}\big)$ for $i=1,2$, assertion (b) is an immediate consequence of (a) and Proposition~\ref{prop:lift-ama-iso-adm}.
\end{proof}

We now generalize Definitions \ref{def:std-ama-SO2} and \ref{def:std-ama-Spin2}.
\begin{definition}\label{def:std-ama-adm}
We write
$\AAA\big(\Pi,\SO{2}\big)$ to denote the isomorphism type of $\AAA\big(\Pi,\sigma,\SO{2}\big)$ and $\AAA\big(\Pi,\kappa,\Spin{2}\big)$ to denote the isomorphism type of $\AAA\big(\Pi,\sigma,\kappa,\Spin{2}\big)$. By slight abuse of notation, we also denote any representative of the respective isomorphism types by these symbols. 
They are called the \Defn{standard $\SO{2}$-amalgam with respect to $\Pi$}, resp.\ the \Defn{standard $\Spin{2}$-amalgam with respect to $\Pi$ and $\kappa$}.
Accordingly, we denote the epimorphism of amalgams from Remark~\ref{rem:spin-ama-to-so-ama-and-back-2} by \[\pi_{\Pi,\kappa} : \AAA\big(\Pi,\kappa,\Spin{2}\big) \to \AAA\big(\Pi,\SO{2}\big).\] 
\end{definition}

Theorems \ref{thm:uniqueness-so-sl} and \ref{thm:uniqueness-spin-sl} generalize from simply laced diagram to arbitrary admissible diagrams.
\begin{theorem} \label{thm:uniquenessadmissible}
The following hold:
\begin{enumerate}
\item Any continuous $\SO{2}$-amalgam $\AAA$ with respect to $\Pi$ and $\sigma$ is isomorphic to the standard amalgam $\AAA\big(\Pi,\SO{2}\big)$.
\item Any continuous $\Spin{2}$-amalgam $\wAAA$ with respect to $\Pi$, $\sigma$ and $\kappa$ is isomorphic to the standard amalgam $\AAA\big(\Pi,\kappa,\Spin{2}\big)$.
\end{enumerate}
\end{theorem}
\begin{proof}
\begin{enumerate}
\item
The following is an adaption of the proof of Consequence~\ref{cons:std-ams-SO2-iso},
 which will imply the claim via Proposition~\ref{prop:adm-ama-labelling-irrelevant}.
The only continuous automorphisms of the circle group $\SO{2}$
are $\id$ and the inversion $\inv$.
Since $\AAA$ continuous by hypothesis, for all $i \neq j \in I$,
if $i^\sigma\to j^\sigma$, one has
\begin{align*}
\phi_{ij}^i\in \{ \eps_1^{a(i,j),a(j,i)}, \eps_1^{a(i,j),a(j,i)}\circ \inv\}\ , &&
\phi_{ij}^j\in \{ \eps_2^{a(i,j),a(j,i)}, \eps_2^{a(i,j),a(j,i)}\circ \inv\}\ .
\end{align*}
For $m=1,2,3$ respectively, set $\gamma_B$, $\gamma_C$ as in
Lemma~\ref{lem:rank1-inv-A2}, \ref{lem:rank1-inv-B2} or \ref{lem:rank1-inv-G2}, respectively.
If $m=0$, then set $\gamma_B=\inv\times\id$ and $\gamma_C=\id\times\inv$. If $m=\infty$, then set $\gamma_B$ to be the automorphism induced by the inversion on $K_1$ and the identity on $K_2$ and $\gamma_C$ to be the automorphism induced by the identity on $K_1$ and the inversion on $K_2$.
Using this, for $i^\sigma\to j^\sigma$ let
\[\alpha_{ij}:=\begin{cases}
\id_{G_{ij}}, &\text{if }\phi_{ij}^i=\eps_1^{a(i,j),a(j,i)},\hphantom{{}\circ\inv}\ \phi_{ij}^j=\eps_2^{a(i,j),a(j,i)}, \\
\gamma_B,     &\text{if }\phi_{ij}^i=\eps_1^{a(i,j),a(j,i)}\circ\inv,\ \phi_{ij}^j=\eps_2^{a(i,j),a(j,i)}, \\
\gamma_C,     &\text{if }\phi_{ij}^i=\eps_1^{a(i,j),a(j,i)},\hphantom{{}\circ\inv}\  \phi_{ij}^j=\eps_2^{a(i,j),a(j,i)}\circ \inv, \\
\gamma_B\circ \gamma_C,
              &\text{if }\phi_{ij}^i=\eps_1^{a(i,j),a(j,i)}\circ\inv,\ \phi_{ij}^j=\eps_2^{a(i,j),a(j,i)}\circ \inv.
\end{cases}\]
Then the system
$\alpha:=\{\pi, \alpha_{ij} \mid i\neq j\in I\}:\AAA\to \AAA\big(\Pi,\sigma,\SO{2}\big)$
is an isomorphism of amalgams.
\item
Let $\AAA$ be the continuous $\SO{2}$-amalgam that induces $\wAAA$ (cf.\ Remark~\ref{rem:spin-ama-to-so-ama-and-back-2}).
Assertion (a) implies $\AAA\cong \AAA\big(\Pi,\SO{2}\big)$.
Proposition~\ref{prop:lift-ama-iso-adm} yields the claim, since $\AAA\big(\Pi,\SO{2}\big)$ induces $\AAA\big(\Pi,\kappa,\Spin{2}\big)$.
\qedhere
\end{enumerate}
\end{proof}

We now generalize Theorem~\ref{thm:K-univ-sl}.

\begin{theorem} \label{thm:K-univ-adm}
Let $G(\Pi)$ be the simply connected split real Kac--Moody group associated to $\Pi$, and let $K(\Pi)$ be its maximal compact subgroup, i.e., the subgroup fixed by the Cartan--Chevalley involution. Then there exists a faithful universal enveloping morphism \[\tau_{K(\Pi)} : \AAA\big(\Pi,\SO{2}\big) \to K(\Pi).\]
\end{theorem}

\begin{proof}
The same proof as for Theorem~\ref{thm:K-univ-sl} applies; the uniqueness of
amalgams then follows from Theorem~\ref{thm:uniquenessadmissible} (instead
of Theorem~\ref{thm:uniqueness-so-sl}).
\end{proof}

In view of Theorem~\ref{thm:K-univ-adm}, we generalize
Definition~\ref{defn:sl-spin-group} as follows:

\begin{definition} \label{defspingrp}
The \Defn{spin group $\Spin{\Pi,\kappa}$ with respect to $\Pi$ and $\kappa$} is the canonical
universal enveloping group of the (continuous) amalgam
$\AAA\big(\Pi,\kappa,\Spin{2}\big)=\{ \widetilde{K}_{ij}, \tilde\phi_{ij}^i \mid i\neq j\in I\}$ with the canonical universal enveloping morphism \[\tau_{\Spin{\Pi},\kappa} : \AAA\big(\Pi,\kappa,\Spin{2}\big) \to \Spin{\Pi,\kappa}.\]
The \Defn{maximal spin group $\Spin{\Pi}$ with respect to $\Pi$} is  $\Spin{\Pi,\kappa_\mathrm{max}}$ (cf.\ Lemma~\ref{lem:max-adm-colouring}).
\end{definition}

\begin{observation} \label{naturalhomo}
Whenever $\kappa\preceq\kappa'$ are admissible colourings of $\Pi$, by construction there exists a canonical central extension
\[ \Spin{\Pi,\kappa'}\onto\Spin{\Pi,\kappa}. \]
\end{observation}

\begin{lemma} \label{lem:adm-K-envelops-spin-amalgam}
$K(\Pi)$ is an enveloping
group of the amalgam $\AAA\big(\Pi,\kappa,\Spin{2}\big)$.  There exists a canonical central extension $\rho_{\Pi,\kappa} : \Spin{\Pi,\kappa} \to K(\Pi)$ that makes the following diagram commute: 
\[
\xymatrix{
  \AAA\big(\Pi,\kappa,\Spin{2}\big) \ar[rr]^{\tau_{\Spin{\Pi,\kappa}}} \ar[d]_{\pi_{\Pi,\kappa}} &&
  \Spin{\Pi} \ar@{-->}[d]^{\rho_{\Pi,\kappa}} \\
\AAA\big(\Pi,\SO{2}\big) \ar[rr]^{\tau_{K(\Pi)}} && K(\Pi)
}
\]
\end{lemma}

\begin{proof}
Essentially the same proof as for Lemma~\ref{lem:sl-K-envelops-spin-amalgam}
works, after appropriately substituting definitions and results from
the present Section~\ref{sec:adm-amalgams}.
\end{proof}

\section{Spin covers of arbitrary type} \label{sec:tametypes}

In this section we prove the following theorem concerning the central extension $\rho_{\Pi,\kappa} : \Spin{\Pi,\kappa} \to K(\Pi)$ from Lemma~\ref{lem:adm-K-envelops-spin-amalgam}.

\begin{theorem} \label{m3}
Given a diagram $\Pi$ and an admissible colouring $\kappa$, the universal enveloping group
$\Spin{\Pi,\kappa}$ of $\AAA\big(\Pi,\kappa,\Spin{2}\big)$ is a $2^{c(\Pi,\kappa)}$-fold central extension of the
universal enveloping group $K(\Pi)$ of $\AAA\big(\Pi,\SO{2}\big)$.
\end{theorem}

Recall from Definition~\ref{adtypedef} that the number $c(\Pi,\kappa)$ counts the connected components of $\Pi^\mathrm{adm}$ on which $\kappa$ takes the value $2$.

\begin{remark}\label{simplylacedalreadyknown}
Let $\kappa$ be an admissible colouring of a simply laced diagram $\Pi$. One has $\Pi = \Pi^\mathrm{adm}$ and $\kappa$ is constant on each connected component of $\Pi$. The number $c(\Pi,\kappa)$ counts the components on which it has value $2$. A reduction to the irreducible case as in the proof of Theorem~\ref{m2} therefore immediately implies Theorem~\ref{m3} for simply laced diagrams.    
\end{remark}

Our strategy of proof in the general case is based on a reduction to the simply-laced case. We start our investigation with carrying out Strategy~\ref{unfolding} for doubly laced diagrams.

\begin{definition} \label{unfoldeddiagram}
Let $n \in \NN$, let $I = \{ 1, \ldots, n \}$ and let $\Pi$ be an irreducible doubly laced Dynkin diagram that is not simply laced with generalized Cartan matrix $A=(a(i,j))_{i,j}$, labelling $\sigma : I \to V(\Pi)$ and admissible colouring $\kappa : V \to \{ 1, 2 \}$ such that for each pair of vertices $i^\sigma$, $j^\sigma$ with $i^\sigma \to j^\sigma$ that form a diagram of type $\mathrm{C}_2$ one has $\kappa(i^\sigma)=2$. 

Then the \Defn{unfolded} Dynkin diagram is the Dynkin diagram $\Pi^{\mathrm{un}}$ with labelling \[\sigma^{\mathrm{un}} : I^\mathrm{un} := \{ \pm i \mid i \in I, \kappa(i^\sigma) = 1 \} \cup \{ i \mid i \in I, \kappa(i^\sigma)=2 \} \to V(\Pi^{\mathrm{un}})\] and edges defined via the generalized Cartan matrix $A^{\mathrm{un}}=(a^{\mathrm{un}}(i,j))_{i,j}$ given by \[a^{\mathrm{un}}(i,j) = \begin{cases}
0, & \text{if }\kappa(|i|^\sigma) \neq \kappa(|j|^\sigma)
    \text{ and }a(|i|,|j|) = 0, \\
-1, & \text{if }\kappa(|i|^\sigma) \neq \kappa(|j|^\sigma)
    \text{ and }a(|i|,|j|) \neq 0, \\
a(|i|,|j|), & \text{if }\kappa(|i|^\sigma) = \kappa(|j|^\sigma)
    \text{ and }ij > 0,   \\
0, & \text{if }\kappa(|i|^\sigma) = \kappa(|j|^\sigma)
    \text{ and }ij<0.
\end{cases}\]
Note that the unfolded Dynkin diagram $\Pi^{\mathrm{un}}$ is simply laced. Define the admissible colouring $\kappa^{\mathrm{un}} :\equiv 2$.

For a reducible Dynkin diagram $\Pi$ with admissible colouring $\kappa : V \to \{ 1, 2 \}$ such that for each pair of vertices $i^\sigma$, $j^\sigma$ with $i^\sigma \to j^\sigma$ that form a diagram of type $\mathrm{C}_2$ one has $\kappa(i^\sigma)=2$, define the \Defn{unfolded} Dynkin diagram $\Pi^{\mathrm{un}}$ and its admissible colouring $\kappa^{\mathrm{un}}$ componentwise: replace each connected component that is not simply laced by its unfolded Dynkin diagram and colouring with constant value $2$ as defined above and retain each simply laced component and its colouring.
\end{definition}

\begin{remark}
Note that for an {\em irreducible} doubly laced Dynkin diagram $\Pi$ that is not simply laced with labelling $\sigma : I \to V(\Pi)$ there is at most one admissible colouring $\kappa : V \to \{ 1, 2 \}$ such that for each pair of vertices $i^\sigma$, $j^\sigma$ with $i^\sigma \to j^\sigma$ that form a diagram of type $\mathrm{C}_2$ one has $\kappa(i^\sigma)=2$ and, if it exists, actually is equal to the maximal admissible colouring $\kappa_{\mathrm{max}}$ from Lemma~\ref{lem:max-adm-colouring}.

This is not necessarily true for reducible such Dynkin diagrams as any simply laced connected components allows the two colourings constant one and constant two.
\end{remark}

\begin{proposition} \label{propunfolding}
Let $\Pi$ be a doubly laced Dynkin diagram with labelling $\sigma : I \to V$ and admissible colouring $\kappa : V \to \{ 1, 2 \}$ such that for each pair of vertices $i^\sigma$, $j^\sigma$ with $i^\sigma \to j^\sigma$ that form a diagram of type $\mathrm{C}_2$ one has $\kappa(i^\sigma)=2$, let $\Pi^{\mathrm{un}}$ be the unfolded Dynkin diagram, let $G(\Pi)$ and $G(\Pi^{\mathrm{un}})$ be the corresponding simply connected split real Kac--Moody groups, let $K(\Pi)$ and $K(\Pi^{\mathrm{un}})$ be their maximal compact subgroups, let \[\tau_{K(\Pi)} : \AAA\big(\Pi,\SO{2}\big) \to K(\Pi)\] and \[\tau_{K(\Pi^{\mathrm{un}})} : \AAA\big(\Pi^{\mathrm{un}},\SO{2}\big) \to K(\Pi^{\mathrm{un}})\] be the respective (faithful) universal enveloping morphisms (cf.\ Theorem \ref{thm:K-univ-adm}), for each $i \in I$ let \[K_i := (\tau_{K(\Pi)} \circ \phi^i_{ij})(\SO{2}) \leq K(\Pi),\] and for each $i \in I^{\mathrm{un}}$ let \[K_i^{\mathrm{un}} := (\tau_{K(\Pi^{\mathrm{un}})} \circ (\phi^i_{ij})^{\mathrm{un}})(\SO{2}) \leq K(\Pi^{\mathrm{un}}),\] where $(\phi^i_{ij})^{\mathrm{un}}$ denote the connecting homomorphisms of the amalgam $ \AAA\big(\Pi^{\mathrm{un}},\SO{2}\big)$.

Then the assignment 
\begin{align*}
\forall i \in I \text{ with } \kappa(i^\sigma)=2: \quad\quad K_i &\to K_i^{\mathrm{un}} \\ g &\mapsto (\tau_{K(\Pi^{\mathrm{un}})} \circ (\phi^i_{ij})^{\mathrm{un}}) \circ (\tau_{K(\Pi)} \circ \phi^i_{ij})^{-1}(g) \\
\forall i \in I \text{ with } \kappa(i^\sigma)=1: \quad\quad K_i &\to K_i^{\mathrm{un}} \times K_{-i}^{\mathrm{un}} \\ g &\mapsto \Big(\big(\tau_{K(\Pi^{\mathrm{un}})} \circ (\phi^i_{ij})^{\mathrm{un}}\big) \times \big(\tau_{K(\Pi^{\mathrm{un}})} \circ (\phi^{-i}_{-ij})^{\mathrm{un}}\big) \Big) \circ (\tau_{K(\Pi)} \circ \phi^i_{ij})^{-1}(g)
\end{align*}
induces a monomorphism \[\Omega_{K(\Pi)} : K(\Pi) \to K(\Pi^{\mathrm{un}}).\]
\end{proposition}

\begin{proof}
The existence of the homomorphism $\Omega_{K(\Pi)} : K(\Pi) \to K(\Pi^{\mathrm{un}})$ is straightforward using Consequence~\ref{concretedescription} and the universal property of $\tau_{K(\Pi)} : \AAA\big(\Pi,\SO{2}\big) \to K(\Pi)$. Injectivity follows from the faithfulness of the universal enveloping morphism $\tau_{K(\Pi^{\mathrm{un}})} : \AAA\big(\Pi^{\mathrm{un}},\SO{2}\big) \to K(\Pi^{\mathrm{un}})$.
\end{proof}

\begin{corollary} \label{corunfolding}
Retain the notation and hypotheses from Proposition~\ref{propunfolding}.
Moreover, let \[\tau_{\Spin{\Pi,\kappa}} : \AAA\big(\Pi,\kappa,\Spin{2}\big) \to \Spin{\Pi,\kappa}\] and \[\tau_{\Spin{\Pi^{\mathrm{un}},\kappa^{\mathrm{un}}}} : \AAA\big(\Pi^{\mathrm{un}},\kappa^{\mathrm{un}},\Spin{2}\big) \to \Spin{\Pi^{\mathrm{un}},\kappa^{\mathrm{un}}}\] be the respective universal enveloping morphisms (cf.\ Definitions~\ref{defn:sl-spin-group} and \ref{defspingrp}), for each $i \in I$ let
\begin{align*}
\widetilde K_i & := (\tau_{\Spin{\Pi,\kappa}} \circ \wt\phi^i_{ij})(\Spin{2}) \leq \Spin{\Pi,\kappa},
\end{align*}
 and for each $i \in I^{\mathrm{un}}$ let
\begin{align*}
\widetilde K_i^{\mathrm{un}} &:= (\tau_{\Spin{\Pi^{\mathrm{un}},\kappa^{\mathrm{un}}}} \circ (\wt\phi^i_{ij})^{\mathrm{un}})(\Spin{2}) \leq \Spin{\Pi^{\mathrm{un}},\kappa^{\mathrm{un}}},
\end{align*}
 where $(\wt\phi^i_{ij})^{\mathrm{un}}$ denote the connecting homomorphisms of the amalgam $ \AAA\big(\Pi^{\mathrm{un}},\kappa^{\mathrm{un}},\Spin{2}\big)$.

Then the assignment 
\begin{align*}
\forall i \in I \text{ with } \kappa(i^\sigma)=2: \quad\quad \wt K_i &\to \wt K_i^{\mathrm{un}} \\ g &\mapsto (\tau_{\Spin{\Pi^{\mathrm{un}},\kappa^{\mathrm{un}}}} \circ (\wt\phi^i_{ij})^{\mathrm{un}}) \circ (\tau_{\Spin{\Pi,\kappa}} \circ \wt\phi^i_{ij})^{-1}(g) \\
\forall i \in I \text{ with } \kappa(i^\sigma)=1: \quad\quad \wt K_i &\to \wt K_i^{\mathrm{un}} \cdot \wt K_{-i}^{\mathrm{un}} \\ g &\mapsto \Big(\big(\tau_{\Spin{\Pi^{\mathrm{un}},\kappa^{\mathrm{un}}}} \circ (\wt\phi^i_{ij})^{\mathrm{un}}\big) \cdot \big(\tau_{\Spin{\Pi^{\mathrm{un}},\kappa^{\mathrm{un}}}} \circ (\wt\phi^{-i}_{-ij})^{\mathrm{un}}\big) \Big) \circ (\tau_{\Spin{\Pi,\kappa}} \circ \wt\phi^i_{ij})^{-1}(g)
\end{align*}
induces a homomorphism $\Omega_{\Spin{\Pi,\kappa}} : \Spin{\Pi,\kappa} \to \Spin{\Pi^{\mathrm{un}},\kappa^{\mathrm{un}}}$ that makes the following diagram commute:
\[
\xymatrix{
  \AAA\big(\Pi,\kappa,\Spin{2}\big) \ar[rrd]^{\tau_{\Spin{\Pi,\kappa}}} \ar[ddd]_{\pi_{\Pi,\kappa}} && && && \AAA\big(\Pi^{\mathrm{un}},\kappa^{\mathrm{un}},\Spin{2}\big) \ar[ddd]^{\pi_{\Pi^{\mathrm{un}},\kappa^{\mathrm{un}}}} \ar[dll]_{\tau_{\Spin{\Pi^{\mathrm{un}},\kappa^{\mathrm{un}}}}} \\ 
&&
  \Spin{\Pi,\kappa} \ar[d]^{\rho_{\Pi,\kappa}} \ar@{-->}[rr]^{\Omega_{\Spin{\Pi,\kappa}}}&& \Spin{\Pi^{\mathrm{un}},\kappa^{\mathrm{un}}} \ar[d]_{\rho_{\Pi^{\mathrm{un}},\kappa^{\mathrm{un}}}}\\
 && K(\Pi) \ar[rr]_{\Omega_{K(\Pi)}} && K(\Pi^{\mathrm{un}}) \\
\AAA\big(\Pi,\SO{2}\big) \ar[urr]_{\tau_{K(\Pi)}} &&&&&& \AAA\big(\Pi^{\mathrm{un}},\SO{2}\big) \ar[ull]^{\tau_{K(\Pi^{\mathrm{un}})}}
}
\]
In particular, if the admissible colouring $\kappa$ is non-trivial, then \[\rho_{\Pi,\kappa} : \Spin{\Pi,\kappa} \to K(\Pi)\] is a non-trivial central extension.
\end{corollary}

\begin{proof}
The existence of the homomorphism $\Omega_{\Spin{\Pi,\kappa}} : \Spin{\Pi,\kappa} \to \Spin{\Pi^{\mathrm{un}},\kappa^{\mathrm{un}}}$ is straightforward using Lemma~\ref{concretedescriptiontilde} and the universal property of $\tau_{\Spin{\Pi},\kappa} : \AAA\big(\Pi,\kappa,\Spin{2}\big) \to \Spin{\Pi,\kappa}$.

Let $i \in I$ with $\kappa(i^\sigma) = 2$ and define 
\begin{align*}
z_i &:= (\tau_{\Spin{\Pi,\kappa}} \circ \wt\phi^i_{ij})(S(\pi)) \in \widetilde K_i, \\
z_i^{\mathrm{un}}  &:= (\tau_{\Spin{\Pi^{\mathrm{un}},\kappa^{\mathrm{un}}}} \circ (\wt\phi^i_{ij})^{\mathrm{un}})(S(\pi)) \in \widetilde K_i^{\mathrm{un}}.
\end{align*}
Then $\Omega_{\Spin{\Pi,\kappa}}(z_i) = z_i^{\mathrm{un}} = -1_{\Spin{\Pi^{\mathrm{un}},\kappa^{\mathrm{un}}},\mathcal{K}(i)}$ (cf.\ Definition~\ref{minus1}).
Since the central extension $\rho_{\Pi^{\mathrm{un}}} : \Spin{\Pi^{\mathrm{un}},\kappa^{\mathrm{un}}} \to K(\Pi^{\mathrm{un}})$ is non-trivial by the simply-laced version of Theorem~\ref{m3} (see also Corollary~\ref{m1cor}), so is the central extension $\rho_{\Pi,\kappa} : \Spin{\Pi,\kappa} \to K(\Pi)$.
\end{proof}

Next we carry out Strategy~\ref{simplelacing} in order to reduce arbitrary diagrams to doubly laced ones.

\begin{definition} \label{doublylacing}
For an augmented Dynkin diagram $\Pi$ with labelling $\sigma : I \to V$ and an admissible colouring $\kappa$, let $\Pi^{\mathrm{dl}\kappa}$ be the doubly
laced diagram with identical orientation (cf.\ Remark~\ref{orientation}) obtained by
\begin{itemize}
\item removing all edges $\{i,j\}^\sigma \in E_\infty(\Pi)$ with $a(i,j)$, $a(j,i)$ even,
\item replacing all edges $\{i,j\}^\sigma \in E_\infty(\Pi)$ with $a(i,j)$ odd, $a(j,i)$ even and $\kappa(i^\sigma) = 2$ by double edges,
\item retaining all edges $\{i,j\}^\sigma \in E_2(\Pi)$ with $a(i,j)$ odd, $a(j,i)$ even and $\kappa(i^\sigma) = 2$, 
\item replacing all other edges in $\Pi$ by simple edges.
\end{itemize}
\end{definition}

\begin{remark}
The diagram $\Pi^{\mathrm{dl}\kappa}$ is a doubly laced with admissible colouring $\kappa$ such that for each pair of vertices $i^\sigma$, $j^\sigma$ with $i^\sigma \to j^\sigma$ that form a diagram of type $\mathrm{C}_2$ one has $\kappa(i^\sigma)=2$.

Note that $c(\Pi,\kappa)=c(\Pi^{\mathrm{dl}\kappa},\kappa)$.

The orientation of $\Pi^{\mathrm{dl}\kappa}$ induced by its labelling and colouring in general differs from the one that $\Pi^{\mathrm{dl}\kappa}$ inherits from the orientation of $\Pi$ induced by its labelling and colouring. Note that, of course, it is possible to change the labellings of $\Pi$ and $\Pi^{\mathrm{dl}\kappa}$ so that the two orientations coincide.
\end{remark}

The correspondence of the indices $i$, $j$ of the epimorphisms $\alpha_{ij}$ and $\tilde\alpha_{ij}$ in the following proposition to the indices $p$ and $l$ in Definition~\ref{lotsofdef} depends on the latter orientation of $\Pi^{\mathrm{dl}\kappa}$ not the former.

\begin{proposition}\label{hom-amalgams}
Given a diagram $\Pi$ and an admissible colouring $\kappa$, there exist epimorphisms of amalgams $\tilde\alpha = (\id,\tilde\alpha_{ij}) : \AAA\big(\Pi,\kappa,\Spin{2}\big) \to \AAA\big(\Pi^{\mathrm{dl}\kappa},\kappa,\Spin{2}\big)$ and $\alpha = (\id,\alpha_{ij}) : \AAA\big(\Pi,\SO{2}\big) \to \AAA\big(\Pi^{\mathrm{dl}\kappa},\SO{2}\big)$ making the following diagram commute:
\[
\xymatrix{
 \AAA\big(\Pi,\kappa,\Spin{2}\big) \ar@{-->}[rr]^{\tilde\alpha} \ar[d]_{\pi_{\Pi,\kappa}} &&  \AAA\big(\Pi^{\mathrm{dl}\kappa},\kappa,\Spin{2}\big) \ar[d]^{\pi_{\Pi^{\mathrm{dl}\kappa},\kappa}} \\
\AAA\big(\Pi,\SO{2}\big) \ar@{-->}[rr]^{\alpha} &&  \AAA\big(\Pi^{\mathrm{dl}\kappa},\SO{2}\big)
}
\]
\end{proposition}

\begin{proof}
There exist the following group epimorphisms that provide the epimorphism $\alpha$:
\begin{align*}
\mathrm{U}_2(\CC) & \to \SO{3} & \text{(see Corollary~\ref{corsurjB2b})} \\
\SO{4} & \to \SO{3} & \text{(see  Lemma~\ref{lemsurjG2})} \\
K(A_1) & \to \SO{2} \times \SO{2} & \text{(see Proposition~\ref{keyproposition})} \\
K(A_2) & \to \SO{3} & \text{(see Proposition~\ref{keyproposition})} \\
K(A_3) & \to \mathrm{U}_2(\CC) & \text{(see Proposition~\ref{keyproposition})}
\end{align*}
where $A_1 = \begin{pmatrix} 2 & -r_1 \\ -s_1 & 2 \end{pmatrix}$ with $r_1$, $s_1$ even,  $A_2 = \begin{pmatrix} 2 & -r_2 \\ -s_2 & 2 \end{pmatrix}$ with $r_2$, $s_2$ odd,  $A_3 = \begin{pmatrix} 2 & -r_3 \\ -s_3 & 2 \end{pmatrix}$ with one of $r_3$, $s_3$ even, the other odd.
The respective lifts to the spin covers exist by Corollary~\ref{surjB2b}, Proposition~\ref{surjG2}, Corollary~\ref{keycorollary}, yielding the epimorphism $\tilde\alpha$.
By  \ref{lemsurjG2}, \ref{surjG2} and \ref{corsurjB2b} and \ref{keyproposition}, \ref{keycorollary} the above diagram commutes.
\end{proof}

\begin{proposition} \label{thm:Spin(Delta)-covers-Spin(Delta-sl)}
Given a diagram $\Pi$ and an admissible colouring $\kappa$, there exist (uniquely determined) epimorphisms
\[\Xi_{\Spin{\Pi,\kappa}}: \Spin{\Pi,\kappa} \to \Spin{\Pi^{\mathrm{dl}\kappa},\kappa} \qquad
\text{and} \quad \quad
\Xi_{K(\Pi)} : K(\Pi) \to K(\Pi^{\mathrm{dl}\kappa})\]
resulting in the following commutative diagram:
\[
\xymatrix{
 \AAA\big(\Pi,\kappa,\Spin{2}\big) \ar[rrrrrr]^{\tilde\alpha} \ar[drr]_{\tau_{\Spin{\Pi,\kappa}}} \ar[ddd]_{\pi_{\Pi,\kappa}} && && &&  \AAA\big(\Pi^{\mathrm{dl}\kappa},\kappa,\Spin{2}\big) \ar[lld]^{\tau_{\Spin{\Pi^{\mathrm{dl}\kappa},\kappa}}} \ar[ddd]^{\pi_{\Pi^{\mathrm{dl}\kappa},\kappa}} \\
&&   \Spin{\Pi,\kappa} \ar@{-->}[rr]^{\Xi_{\Spin{\Pi,\kappa}}} \ar[d]^{\rho_\Pi} &&
  \Spin{\Pi^{\mathrm{dl}\kappa},\kappa} \ar[d]^{\rho_{\Pi^{\mathrm{dl}\kappa},\kappa}} \\
&& K(\Pi) \ar@{-->}[rr]^{\Xi_{K(\Pi)}} && K(\Pi^{\mathrm{dl}\kappa}) \\
\AAA\big(\Pi,\SO{2}\big) \ar[rrrrrr]^{\alpha}  \ar[urr]^{\tau_{K(\Pi)}}  && && &&  \AAA\big(\Pi^{\mathrm{dl}\kappa},\SO{2}\big) \ar[llu]_{\tau_{K(\Pi^{\mathrm{dl}\kappa})}}
}
\]
\end{proposition}

\begin{proof}
The left-hand and the right-hand commutative squares exist by Lemma~\ref{lem:adm-K-envelops-spin-amalgam}. The outer commutative square exists by Proposition~\ref{hom-amalgams}. The composition \[{\tau_{\Spin{\Pi^{\mathrm{dl}\kappa},\kappa}}} \circ \tilde\alpha : \AAA(\Pi,\kappa,\Spin{2}) \to \Spin{\Pi^{\mathrm{dl}\kappa},\kappa}\] is an enveloping morphism, whence the universal property of $\Spin{\Pi,\kappa}$ yields the uniquely determined epimorphism $\Xi_{\Spin{\Pi},\kappa}: \Spin{\Pi,\kappa} \to \Spin{\Pi^{\mathrm{dl}\kappa},\kappa}$, which makes the upper square commute. Similarly, one obtains the uniquely determined epimorphism $\Xi_{K(\Pi)} : K(\Pi) \to K(\Pi^{\mathrm{dl}\kappa})$, that makes the lower square commute. Also \[\tau_{K(\Pi^{\mathrm{dl}\kappa})} \circ \alpha \circ \pi_{\Pi,\kappa} = \tau_{K(\Pi^{\mathrm{dl}\kappa}),\kappa} \circ \pi_{\Pi^{\mathrm{dl}\kappa},\kappa} \circ \tilde\alpha : \AAA(\Pi,\kappa,\Spin{2}) \to K(\Pi^{\mathrm{dl}\kappa})\] is an enveloping morphism. By the universal property of $\tau_{\Spin{\Pi,\kappa}} : \AAA(\Pi,\kappa,\Spin{2}) \to \Spin{\Pi,\kappa}$ in connection with the outer commuting squares one has $\rho_{\Pi^{\mathrm{dl}\kappa},\kappa} \circ \Xi_{\Spin{\Pi},\kappa} = \Xi_{K(\Pi)} \circ \rho_{\Pi,\kappa}$. That is, also the inner square commutes.
\end{proof}

\begin{proof}[Proof of Theorem~\ref{m3}.]
Lemma~\ref{lem:adm-K-envelops-spin-amalgam} provides a canonical central extension
$\rho_{\Pi,\kappa} : \Spin{\Pi,\kappa} \to K(\Pi)$.
By Proposition~\ref{thm:Spin(Delta)-covers-Spin(Delta-sl)}, there exist an epimorphism $\Xi_{\Spin{\Pi,\kappa}} : \Spin{\Pi,\kappa} \to
\Spin{\Pi^{\mathrm{dl}\kappa},\kappa}$ resulting in the commutative diagram:
\[
\xymatrix{\Spin{\Pi,\kappa} \ar[rr]^{\Xi_{\Spin{\Pi,\kappa}}} \ar[d]^{\rho_{\Pi,\kappa}} && \Spin{\Pi^{\mathrm{dl}\kappa},\kappa} \ar[d]^{\rho_{\Pi^{\mathrm{dl}\kappa},\kappa}} \\
K(\Pi) \ar[rr]^{\Xi_{K(\Pi)}} && K(\Pi^{\mathrm{dl}\kappa})
}
\]
Observe first that the claim is obvious if $\kappa$ is trivial (i.e., $c(\Pi,\kappa)=0$). Moreover, it is also true if $\kappa$ is elementary (i.e., $c(\Pi,\kappa)=1$). Indeed, since $\ker(\rho_{\Pi^{\mathrm{dl}\kappa},\kappa}) \neq \{ 1 \}$ by Corollary~\ref{corunfolding}, necessarily also $\ker(\rho_{\Pi,\kappa}) \neq \{ 1 \}$ by the homomorphism theorem of groups. 
By Remark~\ref{allodd} each connected component of $\Pi^{\mathrm{adm}}$ contributes at most a factor $2$ to the order of $\ker(\rho_{\Pi,\kappa})$. Since only those connected components that admit a vertex $i^\sigma$ with $\kappa(i^\sigma)=2$ actually can contribute, one has $|\ker(\rho_{\Pi,\kappa})| \leq 2^{c(\Pi,\kappa)}$. The claim follows for elementary colourings. 

Since $c(\Pi,\kappa)=c(\Pi^{\mathrm{dl}\kappa},\kappa)$, once more the homomorphism theorem of groups implies that is suffices to show the claim for the central extension \[\rho_{\Pi^{\mathrm{dl}\kappa},\kappa} : \Spin{\Pi^{\mathrm{dl}\kappa},\kappa} \to K(\Pi^{\mathrm{dl}\kappa}).\] Since the upper bound that has just been established also holds for $|\ker(\rho_{\Pi^{\mathrm{dl}\kappa},\kappa})|$, it suffices to prove $2^{c(\Pi^{\mathrm{dl}\kappa},\kappa)} \leq |\ker(\rho_{\Pi^{\mathrm{dl}\kappa},\kappa})|$.
To this end decompose $\kappa = \kappa_1 + \cdots + \kappa_{c(\Pi^{\mathrm{dl}\kappa},\kappa)}$ into a sum of pairwise distinct elementary admissible colourings of $\Pi^{\mathrm{dl}\kappa}$.
Let \[\left\{ i_t \in I \mid  1 \leq t \leq c\left(\Pi^{\mathrm{dl}\kappa},\kappa\right), \kappa_t(i_t)=2 \right\}\] be a set of representatives of the connected components of $(\Pi^{\mathrm{dl}\kappa})^\mathrm{adm}$ on which $\kappa$ takes value $2$.
By Observation~\ref{naturalhomo}, for each $1 \leq t \leq c(\Pi^{\mathrm{dl}\kappa},\kappa)$ there exists a central extension \[\pi_t : \Spin{\Pi^{\mathrm{dl}\kappa},\kappa} \to \Spin{\Pi^{\mathrm{dl}\kappa},\kappa_t}.\] Since for each $1 \leq t \leq c(\Pi^{\mathrm{dl}\kappa},\kappa)$ the colouring $\kappa_t$ is elementary, the central extension \[\rho_{\Pi^{\mathrm{dl}\kappa},\kappa_t} : \Spin{\Pi^{\mathrm{dl}\kappa},\kappa_t} \to K(\Pi^{\mathrm{dl}\kappa})\] has kernel of order two. The kernel of the natural homomorphism \[\Spin{\Pi^{\mathrm{dl}\kappa},\kappa} \longrightarrow \prod_{t=1}^{c(\Pi^{\mathrm{dl}\kappa},\kappa)} \Spin{\Pi^{\mathrm{dl}\kappa},\kappa_t} \longrightarrow \prod_{t=1}^{c(\Pi^{\mathrm{dl}\kappa},\kappa)} K(\Pi^{\mathrm{dl}\kappa})\] equals $\ker(\rho_{\Pi^{\mathrm{dl}\kappa},\kappa})$.
By the proof of Corollary~\ref{corunfolding} one has $\pi_t(z_{i_s}) \neq 1$ in $\Spin{\Pi^{\mathrm{dl}\kappa},\kappa_t}$ if and only if $s=t$. We conclude that \[2^{c(\Pi^{\mathrm{dl}\kappa},\kappa)} \leq |\langle z_{i_t} \mid 1 \leq t \leq c(\Pi^{\mathrm{dl}\kappa},\kappa) \rangle| \leq |\ker(\rho_{\Pi^{\mathrm{dl}\kappa},\kappa})|.\]
The claim follows.
\end{proof}

\part{Applications}

\section{Spin-extended Weyl groups} \label{extended} \label{sec14}

Let $A$ be a generalized Cartan matrix with corresponding augmented Dynkin diagram $\Pi:=\Pi(A)$.
In this section we formally construct the spin cover $\Wspin(\Pi,\kappa)$ whose existence has been postulated in \cite[Section~3.5]{DamourHillmann}, for arbitrary diagrams $\Pi$ with admissible colouring $\kappa$ (cf.\ Definition~\ref{adtypedef}). In fact, we provide both a concrete embedding of $\Wspin(\Pi,\kappa)$ into $\Spin{\Pi,\kappa}$ and a presentation by generators and relations.

\begin{notation}
Throughout this section, let $\Pi$ be an augmented Dynkin diagram with vertex set $V$, labelling $\sigma:I\to V$ and 
admissible colouring $\kappa:V\to\{1,2\}$.
Let $A=(a(i,j))_{i,j\in I}$
be the generalized Cartan matrix associated to $\Pi$.  Let $n:=\abs{V}$,
i.e., $I=\{1,\ldots,n\}$, and let \[J:=\{i\in I \mid \kappa(i^\sigma)=1 \}\subset I.\]
\end{notation}

\begin{defn}
For $i\neq j\in I$, recall from \ref{adtypedef}
\[
n(i,j)=
\begin{cases}
0, & \text{if } a(i,j) \text{ is even}, \\
1, & \text{if } a(i,j) \text{ is odd}.
\end{cases}
\]
Define
\[
m_{ij} := 
\begin{cases}
 2, &\text{ if } a(i,j)a(j,i)=0, \\
 3, &\text{ if } a(i,j)a(j,i)=1, \\
 4, &\text{ if } a(i,j)a(j,i)=2, \\
 6, &\text{ if } a(i,j)a(j,i)=3, \\
 0, &\text{ if } a(i,j)a(j,i)\geq 4. \\
\end{cases}
\]
\end{defn}

\begin{defn} \label{defwspin}
Consider the standard amalgams with respect to $\Pi$ and $\kappa$
\[
\AAA\big(\Pi,\kappa,\Spin{2}\big)=\{\widetilde{G}_{ij},\tilde\phi_{ij}^i \mid i\neq j\in I \}
\quad\text{and}\quad
\AAA\big(\Pi,\SO{2}\big)=\{ G_{ij},\phi_{ij}^i \mid i\neq j\in I\}
\]
with enveloping homomorphisms $\tilde\psi_{ij} : \widetilde{G}_{ij} \to \Spin{\Pi,\kappa}$ and $\psi_{ij} : G_{ij} \to K(\Pi)$ to the respective universal enveloping groups.

For $i\neq j\in I$ let
\[ \wh{s}_i:= 
\begin{cases}
\tilde\psi_{ij}(\tilde\phi_{ij}^i(S(\tfrac{\pi}{4}))), & \text{ if } i\notin J, \\
\tilde\psi_{ij}(\tilde\phi_{ij}^i(S(\tfrac{\pi}{4})))^2=\tilde\psi_{ij}(\tilde\phi_{ij}^i(S(\tfrac{\pi}{2}))), & \text{ if } i\in J,
\end{cases}
\]
and set
\[ \wh W:=\wh{W}(\Pi,\kappa):=\cgen{ \wh{s}_i }{ i\in I } \leq \Spin{\Pi,\kappa}. \]
Similarly, let
\[ \wt{s}_i:=  \psi_{ij}(\phi_{ij}^i(D(\tfrac{\pi}{2})))
\qquad\text{ and }\qquad
\wt W:=\wt{W}(\Pi):=\cgen{ \wt{s}_i }{ i\in I } \leq K(\Pi). \]

Note that the elements $\wh{s}_i$ and $\wt{s}_i$ are well-defined --- in particular, independent of the
choice of $j$ --- due to the definition of enveloping groups.
\end{defn}

\begin{defn} \label{preswspin}
The \Defn{Weyl group $W(\Pi)$} associated to $\Pi$ is given by the finite presentation
\begin{align*}
W :=W(\Pi):=
\Big\langle s_1,\ldots,s_n
\,\vert\,
& s_i^2=1 \text{ for } i \in I, \\
& \underbrace{s_i s_j s_i \cdots}_{m_{ij} \text{ factors}} =
  \underbrace{s_j s_i s_j \cdots}_{m_{ij} \text{ factors}} \text{ for } i\neq j\in I
\Big\rangle.
\end{align*}
The \Defn{extended Weyl group $\Wext(\Pi)$} associated to $\Pi$ is given by the finite presentation
\begin{align}
\Wext :=\Wext(\Pi) :=
\Big\langle t_1,\ldots,t_n
\,\vert\,
& t_i^4=1 \text{ for } i \in I, \label{rels:Wext-ri8} \tag{T1} \\
& t_j^{-1} t_i^2 t_j = t_i^2 t_j^{2n(i,j)} \quad \text{ for } i\neq j\in I, \label{rels:Wext-D-conj} \tag{T2} \\
& \underbrace{t_i t_j t_i \cdots}_{m_{ij} \text{ factors}} =
  \underbrace{t_j t_i t_j \cdots}_{m_{ij} \text{ factors}} \label{rels:Wext-braid} \tag{T3} \text{ for } i\neq j\in I
\Big\rangle.
\end{align}
The \Defn{spin-extended Weyl group $\Wspin(\Pi)$} associated to $\Pi$ is given by the finite presentation
\begin{align}
\Wspin :=\Wspin(\Pi):=
\Big\langle r_1,\ldots,r_n
\,\vert\,
& r_i^8=1, \text{ for } i \in I,
        \label{rels:Wspin-ri8} \tag{R1} \\
& r_j^{-1} r_i^2 r_j = r_i^2 r_j^{2n(i,j)} \quad\text{ for } i\neq j\in I,
        \label{rels:Wspin-D-conj} \tag{R2} \\
& \underbrace{r_i r_j r_i \cdots}_{m_{ij} \text{ factors}} =
  \underbrace{r_j r_i r_j \cdots}_{m_{ij} \text{ factors}}
    \label{rels:Wspin-braid} \tag{R3}
  \text{ for } i\neq j\in I \Big\rangle.
\end{align}
\end{defn}

\begin{remark}\label{differentrelations}
The group $\Wext(\Pi)$ is studied in \cite{Tits:1966} and \cite{Kac/Peterson:1985}. By \cite[Corollary 2.4]{Kac/Peterson:1985} and its proof there exists an isomorphism \[\Wext(\Pi) \cong \wt W(\Pi)\] induced by $t_i \mapsto \wt s_i$.
The elements $(\wt s_i)^2$ are torus elements of $K(\Pi)$ of order two, so that the relation (\ref{rels:Wext-D-conj}) \[t_j^{-1} t_i^2 t_j = t_i^2 t_j^{2n(i,j)} \quad\quad \Longleftrightarrow \quad\quad t_i^{-2} t_j^{-1} t_i^2 = t_j^{2n(i,j)}t_j^{-1}\] is equivalent to the centralize-or-invert relation (\ref{crucialidentity}) for torus elements of order two discussed in Remark~\ref{tcentralizes}. If one prefers left-conjugation, one can therefore use one of the relations
\begin{align}
t_j t_i^2 t_j^{-1} = t_i^2 t_j^{-2n(i,j)} \quad\quad & \Longleftrightarrow \quad\quad t_i^{-2} t_j t_i^2 = t_j^{-2n(i,j)}t_j  \tag{T2'}\\ \text{or} \quad\quad\quad\quad\quad \notag \\
t_j t_i^2 t_j^{-1} = t_j^{2n(i,j)} t_i^2 \quad\quad & \Longleftrightarrow \quad\quad t_i^{2} t_j^{-1} t_i^{-2} = t_j^{-1}t_j^{2n(i,j)} \tag{T2''}
\end{align}
instead of (\ref{rels:Wext-D-conj}).

This accounts for the missing minus sign in relation (\ref{rels:Wext-D-conj}) of our presentation of $\Wext(\Pi)$ compared to relation (n1) given in \cite[Corollary 2.4]{Kac/Peterson:1985}. Since, because of relation (\ref{rels:Wext-ri8}) $t_i^4=1$, only the parity of $a(i,j)$ matters, there is no harm in replacing $a(i,j)$ in relation (n1) from \cite[Corollary 2.4]{Kac/Peterson:1985} with $n(i,j)$, as we have done in our relation (\ref{rels:Wext-D-conj}). 

The version of the relations of $\Wext(\Pi)$ we chose can then be directly generalized -- simply by weakening relation (\ref{rels:Wext-ri8}) to (\ref{rels:Wspin-ri8}) -- in order to obtain suitable relations of $\Wspin(\Pi)$ to make it the appropriate spin-extended Weyl group for the spin group $\Spin{\Pi}$.
Valid alternatives for relation (\ref{rels:Wspin-D-conj}) are
 \begin{align}
r_j r_i^2 r_j^{-1} = r_i^2 r_j^{-2n(i,j)} \quad\quad & \Longleftrightarrow \quad\quad r_i^{-2} r_j r_i^2 = r_j^{-2n(i,j)}r_j  \tag{R2'}\\ \text{or} \quad\quad\quad\quad\quad \notag \\
r_j r_i^2 r_j^{-1} = r_j^{2n(i,j)} r_i^2 \quad\quad & \Longleftrightarrow \quad\quad r_i^{2} r_j^{-1} r_i^{-2} = r_j^{-1}r_j^{2n(i,j)} \tag{R2''}
\end{align}
as all are equivalent to the centralize-or-invert relation (\ref{crucialidentity2}) from Proposition~\ref{keycorollary}.
\end{remark}

\begin{lemma} \label{manyrelations}
In $\Wspin(\Pi)$ for all $i,j\in I$ the following relations hold:
\begin{enumerate}
\item $[r_i^2,r_j^2] = r_j^{4n(i,j)}$ and $r_i^{4n(j,i)} = r_j^{4n(i,j)}$,
\item if $n(i,j)=1=n(i,j)$, then $r_i^4=r_j^4$; otherwise $[r_i^2,r_j^2] = 1$,
\item if $n(i,j)=0$, $n(j,i)=1$, then $r_i^4=1$,
\item $r_j r_i^4r_j^{-1}
= \begin{cases}
 r_i^4 & \text{ if } n(i,j)=0, \\
 r_i^2 r_j^2 r_i^2 r_j^2 & \text{ if } n(i,j)=1,
\end{cases}$
\item $[r_j,r_i^4]=1$.
\end{enumerate}
\end{lemma}

\begin{proof}
One computes
\begin{align*}
& r_j^{-1}r_i^2r_j = r_i^2r_j^{2n(i,j)} \\
\Longleftrightarrow  \quad & r_j^{-2}r_i^2r_j^{2} = r_j r_i^2r_j^{-1+2n(i,j)} = r_i^2r_j^{2n(i,j)}r_j^{2n(i,j)}\\
\Longleftrightarrow  \quad & [r_i^{2},r_j^2] = r_j^{4n(i,j)}
\end{align*}
Moreover:
\begin{eqnarray}
r_i^{4n(j,i)} = r_i^{-4n(j,i)} = [r_j^2,r_i^{2}]^{-1} = [r_i^{2},r_j^2] = r_j^{4n(i,j)}. \label{chainofequalities}
\end{eqnarray}
If $n(i,j)=1=n(j,i)$ this immediately implies $r_i^4=r_j^4$. Furthermore, if one of $n(i,j)$, $n(j,i)$ is $0$, then each term in (\ref{chainofequalities}) is equal to $1$, in particular $[r_i^2,r_j^2] = 1$. In case $n(i,j)=0$ and $n(j,i)=1$ one obtains $r_i^4=r_i^{4n(j,i)}=r_j^{4n(i,j)}=1$.  

Also:
\[r_j^{-1} r_i^4r_j = \left(r_j^{-1}r_i^2r_j\right)^2
= \left(r_i^2 r_j^{2n(i,j)}\right)^2
= \begin{cases}
 r_i^4 & \text{ if } n(i,j)=0, \\
 r_i^2 r_j^2 r_i^2 r_j^2 & \text{ if } n(i,j)=1.
\end{cases}\]
If $n(i,j)=1$ and $n(j,i)=0$, then $[r_i^2,r_j^2]=1$ and $r_j^4=1$,
whence $r_j^{-1} r_i^4r_j = r_i^2 r_j^2 r_i^2 r_j^2 = r_i^4r_j^4 = r_i^4$.
Finally, if $n(i,j)=1=n(j,i)$, then $r_i^4=r_j^4$, so we also have $r_j^{-1} r_i^4r_j = r_j^{-1} r_j^4r_j = r_j^4 = r_i^4$.
Therefore always $[r_j,r_i^4]=1$.
\end{proof}

\begin{consequence} \label{Zabelian}
The subgroup $Z:=\gen{r_i^4\mid i\in I}$
of $\Wspin(\Pi)$ is central.
\end{consequence}

\begin{definition}
The \Defn{coloured spin-extended Weyl group $\Wspin(\Pi,\kappa)$} associated to $\Pi$ and $\kappa$ is given by the finite presentation
\begin{align}
\Wspin(\Pi,\kappa):=
\Big\langle r_1,\ldots,r_n
\,\vert\,
& r_i^8=1, \text{ for } i \in I,
         \tag{R1} \\
& r_j^{-1} r_i^2 r_j = r_i^2 r_j^{2n(i,j)} \quad\text{ for } i\neq j\in I,
       \tag{R2} \\
& \underbrace{r_i r_j r_i \cdots}_{m_{ij} \text{ factors}} =
  \underbrace{r_j r_i r_j \cdots}_{m_{ij} \text{ factors}}
     \tag{R3}
  \text{ for } i\neq j\in I \\ 
& r_i^4=1, \text{ for } i \in J \label{rels:Wspin-ri4} \tag{R4} \Big\rangle. 
\end{align}
\end{definition}

\begin{proposition} \label{orderZ}
\begin{enumerate}
\item For all $i^\sigma$, $j^\sigma$ in the same connected component of $\Pi^{\mathrm{adm}}$ one has $r_i^4=r_j^4$ in $\Wspin(\Pi)$. 
\item One has $\Wspin(\Pi) = \Wspin(\Pi,\kappa_{\mathrm{max}})$.
\end{enumerate}
\end{proposition}

\begin{proof}
Assertion (a) is immediate from Lemma~\ref{manyrelations}(b).

By Remark~\ref{combinatoricscolouring}, the maximal admissible colouring $\kappa_{\mathrm{max}}$ is the sum of all elementary admissible colourings. By Remark~\ref{obstruction} the only obstruction to being able to assign the value $2$ to a vertex $k^\sigma$ of $\Pi$ is the existence of a vertex $i^\sigma$ in the same connected component of $\Pi^{\mathrm{adm}}$ as $k^\sigma$ such that there is a vertex $j^\sigma$ with $n(i,j)=0$ and $n(j,i)=1$. This implies $r_k^4=r_i^4=1$ by assertion (a) and Lemma~\ref{manyrelations}(c). 

We conclude that for all $i$ with $\kappa_{\mathrm{max}}(i^\sigma)=1$ one has $r_i^4=1$ in $\Wspin(\Pi)$. Assertion (b) follows.
\end{proof}

\begin{definition}
Let $\Dext(\Pi):=\cgen{ t_i^2 }{ i\in I } \leq \Wext(\Pi)$
and $\Dspin(\Pi):=\cgen{ r_i^2 }{ i\in I } \leq \Wspin(\Pi)$.
\end{definition}

In case $\Pi$ is the Dynkin diagram of type $E_{10}$, the group
$\Dspin(\Pi)$ is isomorphic to the groups $\mathcal{D}^{vs}$ and
$\mathcal{D}^s$ from \cite[Propositions~1 and 2]{DamourHillmann}.

\medskip
The similarities of the presentations of $\Wspin(\Pi)$, $\Wext(\Pi)$, $W(\Pi)$ immediately yield the following:

\begin{observation} \label{upperrowexact}
The following hold:
  \begin{enumerate}
  \item There is a unique epimorphism $\Wspin(\Pi)\to \Wext(\Pi)$ mapping $r_i$ to $t_i$
  and with kernel $Z$.
  \item There is a unique epimorphism $\Wspin(\Pi)\to W(\Pi)$ mapping $r_i$ to $s_i$
  and with kernel $\Dspin$.
  \item There is a unique epimorphism $\Wext(\Pi)\to W(\Pi)$ mapping $t_i$ to $s_i$
  and with kernel $\Dext$.
  \end{enumerate}
\end{observation}

Moreover, from the literature one can extract:

\begin{proposition} \label{KacPeterson}
The following hold:
\begin{enumerate}
\item There is a unique isomorphism $\Wext(\Pi)\to\wt{W}(\Pi)$ mapping $t_i$ to $\wt{s}_i$.
\item There is a unique epimorphism $\wt{W}(\Pi)\to W(\Pi)$ mapping $\wt{s}_i$ to $s_i$.
\item $\Dext\cong\ker\left(\wt{W}(\Pi)\to W(\Pi)\right)$ is an elementary abelian group of order $2^{\abs{I}}$.
\end{enumerate}
\end{proposition}

\begin{proof}
\begin{enumerate}
\item This follows from \cite[Corollary 2.4]{Kac/Peterson:1985} and its proof.
\item
This follows from \cite[Corollary 2.3]{Kac/Peterson:1985}
or, alternatively, by combining (a) with Observation~\ref{upperrowexact}.
\item
This follows from \cite[Lemma 2.2(a) and Corollary 2.3]{Kac/Peterson:1985}.
\qedhere
\end{enumerate}
\end{proof}

\begin{remark}\label{straightforwardcomputation}
Let $n\geq2$ and let $e_1,\ldots,e_n$ be the standard basis of $\RR^n$. In the following, we frequently use variations of the following basic computations in $\Cl{n}$ (cf.\ Section~\ref{sec:spin-pin}):
\[
  (e_1e_2)^2=-1,
  \qquad
  \left(\tfrac{1}{\sqrt{2}}(1+e_1e_2)\right)^2
  =\tfrac{1}{2} (1 + 2 e_1e_2 + (e_1e_2)^2)
  = e_1 e_2,
\]
\[
  \left(\tfrac{1}{\sqrt{2}}(1+e_1e_2)\right)^{-1}
  = \tfrac{1}{\sqrt{2}}(1+e_1e_2) \cdot (e_1e_2)^{-1}
  = \tfrac{1}{\sqrt{2}} (1-e_1e_2).
\]
Moreover, one has 
\[ \frac{1}{\sqrt{2}}(1+e_1e_2)\frac{1}{\sqrt{2}}(1+e_2e_3)+\frac{1}{\sqrt{2}}(1+e_2e_3)\frac{1}{\sqrt{2}}(1+e_1e_2) = 1+e_1e_2 + 1+e_2e_3 - 1.
\] 
\end{remark}

\begin{proposition} \label{uniqueepi}
For each admissible colouring $\kappa$ there is a unique epimorphism $f : \Wspin(\Pi) \to \wh{W}(\Pi,\kappa)$ mapping $r_i$ to $\wh{s}_i$ for $i\in I$, which factors through $\Wspin(\Pi,\kappa)$:
\[\xymatrix{
\Wspin(\Pi) \ar[drr]^f \ar[d] \\
\Wspin(\Pi,\kappa) \ar[rr]_{\overline{f}} && \wh{W}(\Pi,\kappa)
}\]
\end{proposition}

\begin{proof}
The $r_i$ generate $\Wspin$, so if $f$ exists, it is unique. Moreover, the $\wh{s}_i$
generate $\wh{W}$, so if $f$ exists, it is also surjective. 

It remains to show
that $f$ maps every relator of $\Wspin$ to $1_{\wh{W}}$.
Since all relators of $\Wspin$ involve at most two generators it suffices to consider the local situation in all rank 2 subgroups, i.e., to look
at all pairs $i\neq j \in I$ and verify that $r_i$, $r_j$ and the relations
between them are mapped correctly.

For $i\neq j \in I$
define
\[
x_j:= \begin{cases}
    \tilde\phi_{ij}^j(S(\tfrac{\pi}{4})), & \text{if }j\in I\backslash J, \\
    \tilde\phi_{ij}^j(S(\tfrac{\pi}{2})), & \text{if }j\in J.
\end{cases}
\]
Thus for all $j \in I$ one has $\wh{s}_j = \tilde\psi_{ij}(x_j)$ according to Definition~\ref{defwspin}.
Certainly, the relations \ref{rels:Wspin-ri8} (and \ref{rels:Wspin-ri4}) hold for the $\wh{s}_j$; one has \[x_j^8 = \left\{ \begin{array}{cc} \tilde\phi_{ij}^j(S(2\pi)), & \text{if $j\in I\backslash J$}, \\
\tilde\phi_{ij}^j(S(4\pi)), & \text{if $j\in J$} \end{array} \right\} = 1.\]

Define $v:=v_{ij}$, i.e., $\{i,j\}^\sigma\in E_v(\Pi)$, and $k:=k_{ij}$.
We check the relations for each of the possible values of $(v,k)$ in Notation~\ref{31}.

\begin{description}
\item[$\mathbf{k=1}$] This case is well-known, see e.g.\ Proposition~\ref{KacPeterson}(a).
%
%
\item[$\mathbf{v=0,\ k>1}$]
  We have $G_{ij}=\Spin{2} \times \Spin{2}$. Recall Notation~\ref{iotaspin}. Then
  \begin{align*}
  x_i &=\tilde\iota_1(S(\tfrac{\pi}{4}))\ , \\
  x_j &=\tilde\iota_2(S(\tfrac{\pi}{4}))\ .
  \end{align*}
%
%
Relation \ref{rels:Wspin-D-conj}:
\begin{align*}
  x_i^{-1} x_j^2 x_i
   & = x_j^2 = x_j^2 x_i^{2n(i,j)}\ , \\
  x_j^{-1} x_i^2 x_j
  &= x_i^2 = x_i^2 x_j^{2n(j,i)}\ .
\end{align*}
Relation \ref{rels:Wspin-braid}:
\begin{align*}
  x_j x_i
    &= x_i x_j\ .
\end{align*}

%
%
\item[$\mathbf{(v,k)=(1,2)}$]
  We have $G_{ij}=\Spin{3}$. Recall Remark~\ref{6}. Then
  \begin{align*}
  x_i 
  &=\tilde\eps_{23}(S(\tfrac{\pi}{4}))
  = \tfrac{1}{\sqrt{2}}(1+e_2e_3)\ ,
  \\
  x_j 
  &=\tilde\eps_{12}(S(\tfrac{\pi}{4}))
  = \tfrac{1}{\sqrt{2}}(1+e_1e_2)\ .
  \end{align*}
%
%
Relation \ref{rels:Wspin-D-conj}:
\begin{align*}
 x_i^{-1} x_j^2 x_i
 &= \tfrac{1}{\sqrt{2}} (1-e_2e_3) \cdot e_1e_2 \cdot \tfrac{1}{\sqrt{2}} (1+e_2e_3) \\
 &=\tfrac{1}{2} (e_1 e_2 + e_1e_2e_2e_3 - e_2e_3e_1e_2 - e_2e_3e_1e_2e_2e_3) \\
 &=\tfrac{1}{2} (e_1 e_2 - e_1e_3 - e_1e_3 - e_1e_2) \\
 &= -e_1e_3 = e_1 e_2 e_2 e_3 = x_j^2 x_i^2
 = x_j^2 x_i^{2 n(i,j)}\ , \\
  x_j^{-1} x_i^2 x_j
  &= x_i^2 x_j^2 = x_i^2 x_j^{2 n(j,i)}\ .
\end{align*}
Relation \ref{rels:Wspin-braid}:
\begin{align*}
  x_j x_i x_j
    &= x_j \tfrac{1}{\sqrt{2}}(1 + e_2e_3) x_j
     = \tfrac{1}{\sqrt{2}}(x_j^2 + x_j x_i^2 x_j) \\
    &= \tfrac{1}{\sqrt{2}}(x_j^2 + x_j (x_i^2 x_j^2) x_j^{-1}) \\
    &= \tfrac{1}{\sqrt{2}}(x_j^2 + x_j (x_j^{-1} x_i^2 x_j) x_j^{-1}) \\
    &= \tfrac{1}{\sqrt{2}}(x_j^2 + x_i^2)\ .
\end{align*}
By symmetry, we  conclude $x_j x_i x_j = x_i x_j x_i$.

%
%
\item[$\mathbf{(v,k)=(2,1.5),\ i^\sigma\to j^\sigma}$]
  We have $G_{ij}=\SO{2}\times \SU{2}$ and $j\in J$. Then
  \begin{align*}
  x_i 
  &=\tilde{\zeta}_p(S(\tfrac{\pi}{4}))
  =\tfrac{1}{\sqrt{2}} \left(
    \begin{pmat} 1 & 1 \\ -1 &  1\end{pmat},
    \begin{pmat}
      1 & -1 & & \\
      1 & 1 & & \\
        &   & 1 & 1 \\
        &   & -1 & 1
    \end{pmat}
   \right)\ ,  & (\text{cf.\ Notation~\ref{zetaspin}}) \\
  x_j
  &=\tilde{\zeta}_l(S(\tfrac{\pi}{2}))
  =\left( 1_{\SO{2}}, \begin{pmat} &&1 \\ &&&1 \\ -1&&& \\ &-1&& \end{pmat}
  \right)\ .
  \end{align*}
Relation \ref{rels:Wspin-D-conj}:
\begin{align*}
 x_i^{-1} x_j^2 x_i
 &= (1_{\SO{2}},-1_{\SU{2}}) = x_j^2
 = x_j^2 x_i^{2\cdot 0}
  = x_j^2 x_i^{2n(i,j)}
 \ , \\
 x_j^{-1} x_i^2 x_j
 &= \left(
    \begin{pmat}  & 1 \\ -1 &  \end{pmat},
    \begin{pmat}
       & 1 & & \\
      -1 &  & & \\
        &   &  & -1 \\
        &   & 1 & 
    \end{pmat}
   \right) 
  = x_i^2 x_j^2
  = x_i^2 x_j^{2\cdot 1}
  = x_i^2 x_j^{2n(j,i)}
  \ .
\end{align*}
Relation \ref{rels:Wspin-braid}:
\begin{align*}
  x_j x_i x_j x_i = (x_j x_i)^2
  &= \left(
    \begin{pmat}  & 1 \\ -1 &  \end{pmat},
    -1_{\SU{2}}
    \right) 
  =(x_i x_j)^2 = x_i x_j x_i x_j\ ..
\end{align*}

%
%
\item[$\mathbf{(v,k)=(3,2),\ i^\sigma\to j^\sigma}$]
  We have $G_{ij}=\Spin{4}$. Recall Notation~\ref{etaspin}. Then
  \begin{align*}
  x_i 
  &=\tilde\eta_p(S(\tfrac{\pi}{4}))
  =\tilde\eps_{34}\big(S(\tfrac{\pi}{4})\big)
  = \tfrac{1}{\sqrt{2}}(1+e_3e_4)\ ,
  \\
  x_j 
  &=\tilde\eta_l(S(\tfrac{\pi}{4}))
  =\tilde\eps_{14}\big(S(2\cdot\tfrac{\pi}{4})\big)\cdot \tilde\eps_{23}\big(S(-\tfrac{\pi}{4})\big)
  =e_1e_4 \cdot \tfrac{1}{\sqrt{2}}(1-e_2e_3) \\
  &=\tfrac{1}{\sqrt{2}}(e_1e_4 + e_1e_2e_3e_4)\ .
  \end{align*}
%
Relation \ref{rels:Wspin-D-conj}:
\begin{align*}
 x_i^{-1} x_j^2 x_i
 &= e_2 e_4 = x_j^2 x_i^2
 = x_j^2 x_i^{2 \cdot 1}
 = x_i^2 x_j^{2n(i,j)}
 \ , \\
 x_j^{-1} x_i^2 x_j
 &= -e_2 e_4 =  x_i^2 x_j^2
 = x_j^2 x_i^{2 \cdot 1}
 = x_j^2 x_i^{2n(j,i)}
 \ .
\end{align*}
Relation \ref{rels:Wspin-braid}:
$
  x_j x_i x_j x_i x_j x_i = (x_j x_i)^3
  = -e_1e_2e_3e_4 = 
  (x_i x_j)^3 = x_i x_j x_i x_j x_i x_j\ .
$

\item[$\mathbf{v=\infty,\ k>1,\ i^\sigma \to j^\sigma}$]
  We have $G_{ij}=\wt K_1 \wt U *_{\wt U} \wt K_2 \wt U$. Recall Definition~\ref{thetaspin}. Then
  \begin{align*}
  x_i &= \tilde\theta_1(S(\tfrac{\pi}{4}))\ ,
  \\
  x_j &=\begin{cases}
\tilde\theta_2(S(\tfrac{\pi}{2})), & \text{if }j\in J
    \qquad (\iff n(i,j)=1,\ n(j,i)=0) , \\
\tilde\theta_2(S(\tfrac{\pi}{4})), & \text{if }j\in I\backslash J.
\end{cases}
  \end{align*}
Relation \ref{rels:Wspin-D-conj}: As discussed in Remark~\ref{differentrelations} one has
\begin{align*}
x_i^{-1} x_j^2 x_i = x_j^2 x_i^ {2n(j,i)} & \Longleftrightarrow x_j^{-2} x_i^{-1} x_j^{2} = x_i^{2n(j,i)}x_i^{-1}, \\
x_j^{-1} x_i^2 x_j = x_i^2 x_j^ {2n(i,j)} & \Longleftrightarrow x_i^{-2} x_j^{-1} x_i^{2} = x_j^{2n(i,j)}x_j^{-1},
\end{align*} 
Since $x_i \in \wt K_1$, $x_j \in \wt K_2$, $x_i^2 = \tilde t_1$, $x_j^2 = \tilde t_2$ (cf.\ Remark~\ref{dichotomy}), the respective equalities on the right hand sides hold by identity (\ref{crucialidentity2}) in Proposition~\ref{keycorollary}.
\qedhere
\end{description}  
\end{proof}

\begin{theorem} \label{constwspin}
The group homomorphism 
 \[
  \overline{f} : \Wspin(\Pi,\kappa) \to \wh{W}(\Pi,\kappa)
\]
mapping $r_i \mapsto \wh{s}_i$
is an isomorphism. In particular \[\Wspin(\Pi) \cong \wh{W}(\Pi,\kappa_{\mathrm{max}}).\]
\end{theorem}

\begin{proof}
We already established that $\overline{f}$ is an epimorphism.
It remains to prove that it is injective.

\medskip

Let $\rho:\Spin{\Pi,\kappa}\to K(\Pi)$ be the central extension
from Theorem~\ref{m3}.
Then, using
Lemmas \ref{lem:ama-embed-and-rho-commute-An}, \ref{lem:ama-embed-and-rho-commute-G2}, \ref{lem:ama-embed-and-rho-commute-B2} and Proposition~\ref{keycorollary},
it follows that
\begin{align*}
 \rho(\wh{s}_i)
&=
\begin{cases}
(\rho\circ\tilde\psi_{ij}\circ\tilde\phi_{ij}^i\circ S)(\tfrac{\pi}{4}) & \text{ if } i\notin J \\
(\rho\circ\tilde\psi_{ij}\circ\tilde\phi_{ij}^i\circ\sq\circ S)(\tfrac{\pi}{4}) & \text{ if } i\in J
\end{cases} \\
&=
\begin{cases}
(\psi_{ij}\circ\rho^{(v_{ij})}\circ\tilde\phi_{ij}^i\circ S)(\tfrac{\pi}{4}) & \text{ if } i\notin J \\
(\psi_{ij}\circ\rho^{(v_{ij})}\circ\tilde\phi_{ij}^i\circ\sq\circ S)(\tfrac{\pi}{4}) & \text{ if } i\in J
\end{cases} \\
&= (\psi_{ij}\circ\phi_{ij}^i\circ\rho_2\circ S)(\tfrac{\pi}{4})
 = (\psi_{ij}\circ\phi_{ij}^i\circ D)(\tfrac{\pi}{2})
 = \wt{s}_i .
\end{align*}
Since $\ker\rho=\langle (\wh{s}_i)^4 \mid i \in I \rangle \subseteq \wh{W}(\Pi,\kappa)$, 
the restriction of $\rho$ to an epimorphism $\wh{W}(\Pi,\kappa)\to\wt{W}(\Pi)$
is a $2^{c(\Pi,\kappa)}$-fold central extension by Theorem~\ref{m3}.
Consider the following diagram:

\[\xymatrix{
 1 \ar[rr] & & \langle (r_i)^4 \mid i \in I \rangle \ar[d]^{(r_i)^4 \mapsto (\wh{s}_i)^4} \ar[rr] &
 & \Wspin(\Pi,\kappa) \ar[d]^{\overline{f} : r_i\mapsto \wh{s}_i}\ar[rr]^{r_i\mapsto t_i} &
 & \Wext(\Pi) \ar[d]^{t_i\mapsto\wt{s}_i}\ar[rr] & & 1 \\
 1 \ar[rr] & & \langle (\wh{s}_i)^4 \mid i \in I \rangle \ar[rr] &
 & \wh{W}(\Pi,\kappa) \ar[rr]^{\wh{s}_i\mapsto \wt{s}_i} &
 & \wt{W}(\Pi) \ar[rr] & & 1
}
\]
The upper row of this diagram is exact by Lemma~\ref{upperrowexact}(a), the lower row by the above discussion.
It is not hard to verify that the diagram is also commutative. 
By Consequence~\ref{Zabelian} the group $\langle (r_i)^4 \mid i \in I \rangle$ is abelian and by Proposition~\ref{orderZ} it has at most $2^{c(\Pi,\kappa)}$ elements. Since $|\langle (\wh{s}_i)^4 \mid i \in I \rangle| = 2^{c(\Pi,\kappa)}$ by the above discussion, the left-hand arrow is an isomorphism.
The right-hand arrow is an isomorphism by Proposition~\ref{KacPeterson}(a). Thus $\overline{f}$ is also an isomorphism.
The final statement follows from Proposition~\ref{orderZ}(b).
\end{proof}

We have proved Theorem~\ref{mainthm:sl-weyl} from the introduction.

\section{Finite-dimensional compact Lie group quotients of $\Spin{\Pi}$ and $K(\Pi)$} \label{sec13}

In this section we prove that for each simply laced diagram $\Pi$ the group $\Spin{\Pi}$ admits an epimorphism onto a non-trivial compact Lie group afforded by the representation constructed in Theorem~\ref{m1}. By virtue of Theorem~\ref{thm:Spin(Delta)-covers-Spin(Delta-sl)} such an epimorphism then in fact exists for any diagram $\Pi$ and any admissible colouring $\kappa$.

\begin{theorem}
Let $\Pi$ be a simply laced diagram. Then the target of the (continuous) group epimorphism $\Xi : \Spin{\Pi} \to X$ from Remark~\ref{inparticular} is a non-trivial compact Lie group, as is the target of the induced (continuous) group epimorphism $\overline \Xi : K(\Pi) \to X / \left\langle \Xi(Z) \right\rangle$.
\end{theorem}

\begin{proof}
By \cite[Theorem~4.11]{Hainke/Koehl/Levy} the image of the generalized spin representation $\mu : \mathfrak{k} \to \End(\CC^s)$ from Theorem~\ref{m1} is compact. Hence the claim for $\Xi$, and by the homomorphism theorem for groups also for $\overline\Xi$, follows.
\end{proof}

\begin{corollary} [{cf.\ \cite[Lemma~2, p.~49]{DamourHillmann}}] \label{DamourHillmann}
Let $\Pi$ be a simply laced diagram. Then the image $\Xi\left(\widehat W(\Pi)\right)$ is finite.
\end{corollary}

\begin{proof}
As in the proof of \cite[Lemma~2, p.~49]{DamourHillmann} the key observation is that $\Xi\left(\widehat W(\Pi)\right) \leq X$ is a discrete subgroup of the compact group X. Consider the commutative diagram obtained from Theorem~\ref{m1}
\[
\xymatrix{
\{ k \frac{\pi}{4} \mid k \in \ZZ/8\ZZ \} \ar@{^{(}->}[dd]_S \ar@{^{(}->}[rr]^{\tilde\psi_{ij} \circ \tilde\phi_{ij}^i \circ S} && \Spin\Pi \ar[drr]^\Xi \\
&& \AAA(\Pi,\Spin 2) \ar[rr]^{\Psi_\AAA} \ar[u]^{\widetilde\Psi} && X \\
\Spin{2} \ar@/^1pc/[uurr]^{\tilde\psi_{ij} \circ \tilde\phi_{ij}^i} \ar[urr]^{\tilde\phi_{ij}^i} \ar@/_/[urrrr]_{\psi_{ij} \circ \tilde\phi_{ij}^i} 
}
\]
where \[\widetilde\Psi = \{ \tilde \psi_{ij} : \widetilde G_{ij} \to \Spin\Pi \} := \tau_{\Spin\Pi}= \{ \tau_{ij} : \widetilde G_{ij} \to \Spin\Pi \}\] is the canonical universal enveloping morphism of the amalgam $\AAA(\Pi,\Spin 2)= \{ \widetilde G_{ij}, \tilde\phi_{ij}^i \mid i \neq j \in I \}$.
Then \[\widehat W(\Pi) = \langle \widehat s_i \mid i \in I \rangle\] where $\widehat s_i := \tilde\psi_{ij}(\tilde\phi_{ij}^i(S(\frac{\pi}{4})))$ and \[\Xi\left(\widehat W(\Pi) \right) = \left\langle \psi_{ij}(\tilde\phi_{ij}^i(S(\frac{\pi}{4}))) \mid i \in I \right\rangle = \left\langle \cos(\frac{\pi}{4}) + \sin(\frac{\pi}{4})X_i \mid i \in I \right\rangle.\] 
For $R_i := \cos(\frac{\pi}{4}) + \sin(\frac{\pi}{4})X_i=\frac{\sqrt{2}}{2}+\frac{\sqrt{2}}{2}X_i$ the commutator relations
\[X_iX_j=\begin{cases}
-X_jX_i,&\text{if }\{i,j\}^\sigma\in E(\Pi)\ , \\
X_jX_i,&\text{if }\{i,j\}^\sigma\notin E(\Pi)\ 
\end{cases}\]
established in \cite[Remark~4.5]{Hainke/Koehl/Levy} (cf.\ the proof of Theorem~\ref{m1}) imply that
\[R_iR_j = \begin{cases}
    -R_jR_i + \sqrt{2}R_i + \sqrt{2}R_j - 1, &
              \text{if }\{i,j\}^\sigma \in E(\Pi), \\
    R_jR_i, & \text{if }\{i,j\}^\sigma \not\in E(\Pi),
\end{cases}\]
as in Remark~\ref{straightforwardcomputation} or in the proof of \cite[Lemma~2, p.~49]{DamourHillmann};
moreover, $X_i^2 = -\id$ implies $R_i^2 = \sqrt{2}R_i-1$.

We conclude that any product of the above generators of $\Xi\left(\widehat W(\Pi)\right)$ can be written as a polynomial in the $R_i$, $i \in I$, such that each $R_i$ appears at most linearly in each polynomial, with coefficients in $\ZZ[\sqrt{2}]$. Hence $\Xi\left(\widehat W(\Pi)\right)$ is a discrete set. As it is a subgroup of the compact group $X$, it has to be finite. 
\end{proof}

Not much is known about the kernels of the maps $\Xi$ and $\overline{\Xi}$. However, Pierre-Emmanuel Caprace pointed out to us that these kernels generally cannot be abstractly simple. The argument relies on the concept of acylindrical hyperbolicity (see \cite{Osin:2013}, also \cite{CapraceHume}).

\begin{corollary}
Let $\Pi$ be an irreducible non-spherical, non-affine simply laced Dynkin diagram. Then the kernels of the generalized spin representations $\Xi : \Spin{\Pi} \to X$ and $\overline \Xi : K(\Pi) \to X / \left\langle \Xi(Z) \right\rangle$ are acylindrically hyperbolic and, in particular, not abstractly simple.
\end{corollary} 

\begin{proof}
$\Spin{\Pi}$ and $K(\Pi)$ naturally act on each half of the twin building of the associated Kac--Moody group $G(\Pi)$ and, thus, so do $\ker(\Xi)$ and $\ker(\overline{\Xi})$. The stabiliser in $K(\Pi)$ of the fundamental chamber consists of the elements of the standard torus of $G(\Pi)$ that square to the identity, the stabiliser in $\Spin{\Pi}$ is a double extension of this group. In particular, chamber stabilisers in $K(\Pi)$ and in $\Spin{\Pi}$ are finite and, hence, so are chamber stabilisers in $\ker(\overline{\Xi})$ and $\ker(\Xi)$. By Corollary~\ref{DamourHillmann} the intersection $\ker(\overline{\Xi}) \cap \wt W(\Pi)$ has finite index in $\wt W(\Pi)$ and the intersection $\ker(\Xi) \cap \wh W(\Pi)$ has finite index in $\wh W(\Pi)$ (see Section~\ref{extended} for definitions). Hence \cite[Theorem~1.4]{CapraceHume} applies and $\ker(\Xi)$ and $\ker(\overline{\Xi})$ are acylindrically hyperbolic. Therefore by \cite[Theorem~8.5]{Dahmani/Guirardel/Osin} both groups contain a non-trivial proper free normal subgroup. 
\end{proof}

\begin{corollary}
Let $\Pi$ be a Dynkin diagram with admissible colouring $\kappa$. Then the target of the group homomorphism $\Spin{\Pi,\kappa} \stackrel{\Xi_{\Spin{\Pi,\kappa}}}{\longrightarrow} \Spin{\Pi^{\mathrm{dl\kappa}},\kappa} \stackrel{\Omega_{\Spin{\Pi^{\mathrm{dl\kappa}},\kappa}}}{\longrightarrow} \Spin{(\Pi^{\mathrm{dl}})^{\mathrm{un}},\kappa^{\mathrm{un}}} \stackrel{\Xi}{\longrightarrow} X$ is a non-trivial compact Lie group, as is the target of the induced group epimorphism $K(\Pi) \stackrel{\Xi_{K(\Pi)}}{\longrightarrow} K((\Pi^{\mathrm{dl}\kappa}) \stackrel{\Omega_{K(\Pi^{\mathrm{dl\kappa}})}}{\longrightarrow} K(\Pi^{\mathrm{dl}})^{\mathrm{un}}) \stackrel{\overline{\Xi}}{\longrightarrow} X/ \left\langle \Xi(Z) \right\rangle$.
\end{corollary}

\begin{proof}
It suffices to observe that $\Spin{\Pi^{\mathrm{dl\kappa}},\kappa} \stackrel{\Omega_{\Spin{\Pi^{\mathrm{dl\kappa}},\kappa}}}{\longrightarrow} \Spin{(\Pi^{\mathrm{dl}})^{\mathrm{un}},\kappa^{\mathrm{un}}}$ and $K(\Pi^{\mathrm{dl}\kappa}) \stackrel{\Omega_{K(\Pi^{\mathrm{dl\kappa}})}}{\longrightarrow} K((\Pi^{\mathrm{dl}})^{\mathrm{un}})$ have closed images with respect to the Kac--Peterson topology, as they can be realized as fixed points of continuous involutions.
\end{proof}

\begin{remark}
For $\Pi=E_n$ for some $n \in \NN$, the isomorphism type of $X$ can be extracted from the Cartan--Bott periodicity described in \cite[Theorem~A]{Horn/Koehl}.
\end{remark}

\bibliographystyle{alpha}
\bibliography{References}

\end{document}